\documentclass[oneside,11pt]{article}

\usepackage{amsmath}
\usepackage{amssymb}
\usepackage{pxfonts}
\usepackage{graphicx}
\usepackage{eucal}
\usepackage{mathrsfs}
\usepackage{theorem}
\usepackage{pifont}
\usepackage{sectsty}
\usepackage{amscd}
\usepackage{color}
\usepackage{fancyhdr}
\usepackage{framed}
\usepackage[all]{xy}
\usepackage[backref=page]{hyperref}
\usepackage[normalem]{ulem}
\usepackage{enumitem}
\usepackage{setspace}
\usepackage[margin=2cm]{geometry}
\definecolor{shadecolor}{rgb}{0.8,0.8,0.8}
\usepackage[font=footnotesize]{caption}
\usepackage{mathtools}
\usepackage[nottoc, numbib]{tocbibind}

\theoremheaderfont{\fontfamily{pzc}\bfseries\large}
\newtheorem{theorem}{Theorem}[section]

\newtheorem{lemma}[theorem]{Lemma}
\newtheorem{proposition}[theorem]{Proposition}

\newtheorem{corollary}[theorem]{Corollary}
\newtheorem{definition}[theorem]{Definition}

\theorembodyfont{\rmfamily}

\newtheorem{remark}[theorem]{Remark}
\newcommand{\specexercise}[1]{}

\newenvironment{proof}{{\flushleft \emph{Proof:}}}{\hfill\ding{110}}
\newenvironment{proof1}[1]{{\flushleft \emph{Proof #1}}}{\hfill\ding{110}}

\newcommand{\dist}{\operatorname{dist}}
\newcommand{\SO}{\operatorname{SO}}

\newcommand{\tr}{\operatorname{tr}}

\newcommand{\brk}[1]{\left(#1\right)}          % \brk{.}     => (.)
\newcommand{\BRK}[1]{\left\{#1\right\}}        % \BRK{.}     => {.}
\newcommand{\Abs}[1]{\left| #1 \right|}        % \Abs{.}     => |.|
\newcommand{\inner}[1]{\left\langle#1\right\rangle}      % \inner{.} => <.>
\newcommand{\at}[1]{\left.#1\right|}
\newcommand{\mat}[1]{\left(\begin{matrix} #1 \end{matrix}\right)}

\newcommand{\beq}{\begin{equation}}
\newcommand{\eeq}{\end{equation}}

\newcommand{\Emph}[1]{{\slshape\bfseries #1}}  % \Emph{.}

\newcommand{\weakly}{\rightharpoonup}
\newcommand{\conv}[1]{\stackrel{#1}{\longrightarrow}}
\newcommand{\wconv}[1]{\xrightharpoonup{#1}}

\newcommand{\R}{\mathbb{R}}
\newcommand{\calR}{\mathcal{R}}

\newcommand{\Rot}{P}
\newcommand{\calW}{\mathcal{W}}
\newcommand{\calQ}{\mathcal{Q}}
\newcommand{\tcalQ}{\tilde{\calQ}}
\newcommand{\tcalL}{\tilde{\mathcal{L}}}

\newcommand{\e}{\varepsilon}
\newcommand{\W}{\Omega}
\newcommand{\w}{\omega}

\newcommand{\pl}{\partial}

\newcommand{\M}{\mathcal{M}}
\newcommand{\g}{\mathfrak{g}}
\newcommand{\tg}{\tilde{\g}}
\newcommand{\dVol}{\textup{dVol}}
\newcommand{\E}{\mathcal{E}}
\newcommand{\tE}{\tilde{\E}}
\newcommand{\sym}{\operatorname{sym}}
\newcommand{\cof}{\operatorname{cof}}
\newcommand{\n}{\mathfrak{n}}
\newcommand{\tn}{\tilde{\n}}
\newcommand{\N}{\mathfrak{N}}
\newcommand{\II}{\textup{II}}
\renewcommand{\a}{\mathfrak{a}}

\newcommand{\calS}{\mathcal{S}}
\newcommand{\Sw}{\calS_w}
\newcommand{\hSw}{\hat{\calS}_w}
\newcommand{\tSw}{\tilde{\calS}_w}
\newcommand{\tS}{\tilde{\calS}}
\newcommand{\DG}{\delta^G}
\newcommand{\DC}{\delta^C}
\newcommand{\tQ}{\tilde{P}}

\newcommand{\dL}{\alpha}

\newcommand{\dM}{\beta}

\newcommand{\dN}{\gamma}
\newcommand{\bdN}{\overline{\dN}}
\newcommand{\vp}{\varphi}
\newcommand{\Wiso}{W^{2,2}_\text{iso}}
\newcommand{\bQ}{\bar{\calQ}}
\newcommand{\bl}{l}
\newcommand{\bm}{m}
\newcommand{\bn}{n}
\newcommand{\Rbar}{\hat{R}}
\newcommand{\olda}{\sigma}
\newcommand{\oldg}{\lambda}
\newcommand{\olds}{\theta}
\newcommand{\Fo}{F_0}
\newcommand{\Ao}{A_0}
\newcommand{\QQw}{{{\mathcal Q}_w}}

\newcommand{\bb}{{\nu}}
\newcommand{\Qo}{Q_0}
\newcommand{\f}{f}
\newcommand{\WW}{S}
\newcommand{\Bo}{B_0}
\newcommand{\Go}{G_0}
\renewcommand{\thefootnote}{\fnsymbol{footnote}}

\def\Xint#1{\mathchoice
   {\XXint\displaystyle\textstyle{#1}}%
   {\XXint\textstyle\scriptstyle{#1}}%
   {\XXint\scriptstyle\scriptscriptstyle{#1}}%
   {\XXint\scriptscriptstyle\scriptscriptstyle{#1}}%
   \!\int}
\def\XXint#1#2#3{{\setbox0=\hbox{$#1{#2#3}{\int}$}
     \vcenter{\hbox{$#2#3$}}\kern-.5\wd0}}

\def\dashint{\Xint-}

\numberwithin{equation}{section}

\begin{document}

\title{Rigorous analysis of shape transitions in frustrated elastic ribbons}
\author{Cy Maor and Maria Giovanna Mora}

\newcommand{\Addresses}{{
  \bigskip
  \footnotesize

  C.~Maor, \textsc{Einstein Institute of Mathematics,
  The Hebrew University of Jerusalem, Israel}\par\nopagebreak
  \textit{E-mail address}: \texttt{cy.maor@mail.huji.ac.il}

  \medskip

  M.G.~Mora, \textsc{Dipartimento di Matematica, Universit\`a di Pavia, Italy}\par\nopagebreak
  \textit{E-mail address}: \texttt{mariagiovanna.mora@unipv.it}

}}

\date{}
\maketitle
\renewcommand{\thefootnote}{\arabic{footnote}}

\abstract{
\noindent
Ribbons are elastic bodies of thickness $t$ and width $w$ with $t\ll w\ll 1$ (after appropriate nondimensionalization).
Many ribbons in nature have a non-trivial internal geometry, making them incompatible with Euclidean space. This incompatibility --- expressed mathematically as a failure of the Gauss--Codazzi equations for surfaces --- can trigger shape transitions between narrow and wide ribbons.
These transitions depend on the internal geometry: ribbons whose incompatibility arises from failure of the Gauss equation always exhibit a transition, whereas those whose incompatibility arises from failure of the Codazzi equations, may or may not.
We give the first rigorous analysis of this phenomenon, mainly for ribbons whose first fundamental form is flat. For Gauss-incompatible ribbons we identify the natural energy scaling of the problem and prove the existence of a shape transition. For Codazzi-incompatible ribbons we give a necessary condition for a transition to occur.
Furthermore, our study reveals a fundamental distinction: the transition is ``microscopic'' for Gauss-incompatible ribbons, persisting as the width tends to 0, whereas it is ``mesoscopic'' for Codazzi-incompatible ribbons, observable only at small but finite width.
The results are obtained by calculating the $\Gamma$-limits, as $t,w\to 0$, for narrow ribbons ($w^2 \ll t$), and wide ribbons (taking $t$ to zero and then $w$), in the natural energy scalings dictated by the internal geometry.\medskip

%74B20  	Nonlinear elasticity
%Mechanics of deformable solids: 74Kxx		Thin bodies, structures
%49J45  	Methods involving semicontinuity and convergence; relaxation
%53Z05  	Applications of differential geometry to physics
\medskip


}

\setcounter{tocdepth}{2}
{\footnotesize 
\tableofcontents
}

\nocite{FMP12,FMP13,FHMP16b,FJM02b,MS19,BLS16,SLS21,LSSM21,GSD16,ZGDS19,AESK11,KS14,Efr15,KA24}

%%%%%%%%%%%%%%%%%%%%%%%%%%%%%%%%%%%%%%%%%%%%%
\section{Introduction}
Ribbons are elastic bodies that have two small scales: they are both narrow and thin, but their thickness is much smaller than their width.
Their study has a long history, from Sadowsky in the 1930s \cite{Sad30}, through Wunderlich in the 1960s \cite{Wun62} to a large body of works in the 21st century (e.g., \cite{Her06,PS10,AESK11, FMP12,DA14,CDD14, DA15, FF16, FHMP16,PT19,NB21,SLS21,LSSM21}).
Such bodies are ubiquitous in nature \cite{AESK11,HWQSH18,ZGDS19}, from plants to molecules to applications in electro- and nano-technology \cite{SB09, FKS12}.
Many ribbons have internal geometries that are incompatible with Euclidean space, e.g., due to inhomogeneous swelling, plastic deformations or differential growth.
Such ribbons do not have a stress-free configuration: they exhibit stress even in the absence of external forces or prescribed boundary conditions.  
Their study aims to understand the relation between the local intrinsic geometry of the ribbon and the global shape in space they obtain.
This analysis explains the shapes of some molecular assemblies \cite{Her06, ZGDS19} and plants \cite{AESK11,HWQSH18}, and is also used to ``design'' ribbons of various shapes by prescribing their intrinsic local geometry \cite{SYUTGSS11, ADK17,SLS21,LSSM21}.

One of the most interesting feature of incompatible ribbons is that ribbons with the same intrinsic geometry often exhibit a sharp shape transition between ``wide'' and ``narrow'' ribbons.
This phenomenon was observed experimentally \cite{GSD16,ZGDS19,SLS21,LSSM21}, and was explained using formal asymptotics in \cite{GSD16,LSSM21}. 
The present work provides the first rigorous analysis of this shape transition in ribbons:
We prove the existence/non-existence of transitions for many physically relevant geometries, and outline many open questions that arise for geometries and regimes that are beyond the scope of this work.
We hope this will inspire further rigorous investigations into the behavior of incompatible ribbons.

%%%%%%%%%%%%%%%%%%%%%%%
\paragraph{The reduced shell model, Gauss and Codazzi incompatibilities.}
We start by presenting the typical model for incompatible ribbons used in the physics literature:
The ribbon is modeled as a two-dimensional body $\Sw = (0,L)\times (-w/2,w/2)$, where $w\ll L$ is the width of the ribbon (we assume that $L \sim 1$ is a fixed quantity).
We denote the natural coordinates on $\Sw$ by $z' = (z_1,z_2)$.
Its \Emph{midline} is the set $\ell = (0,L)\times \{0\} \subset \Sw$.
The ribbon is associated with two fields: 
\begin{itemize}
\item Its reference metric, or first fundamental form, a symmetric, positive-definite tensor $\a:T\Sw\times T\Sw\to \R^{2\times2}$.
	We assume that $z_1$ is an arclength coordinate along the midline $\ell$, and that the $z_2$ direction is always perpendicular to the $z_1$ direction.
	Thus, the metric is of the type
	\[
	\a(z_1,z_2) = \mat{1 + O(|z_2|) & 0 \\ 0 & 1}.
	\]
\item The reference second fundamental form, a symmetric tensor $\II:T\Sw\times T\Sw\to \R^{2\times2}$.
\end{itemize}
A \Emph{configuration} is a (smooth enough) map $f : \Sw \to \R^3$, whose associated \Emph{elastic energy} is given by
\beq\label{eq:Ereduced}
E^\text{red}_{t,w}(f) = \dashint_{\Sw} |\a_f - \a|^2 \,\dVol_\a + t^2 \dashint_{\Sw} |\II_f - \II|^2\, \dVol_\a,
\eeq
where $t\ll w$ is the thickness of the ribbon, and $\a_f = \nabla f^T\nabla f$ and $\II_f = -\nabla f^T\nabla \nu_f$ are the first and second fundamental forms associated with $f$, respectively (here $\nu_f$ is the normal of $f$).
This model is sometimes called a \emph{reduced shell model}, or a \emph{Kirchhoff shell model}.

The forms $\a_f$ and $\II_f$ are related by a system of three differential equations, the \emph{Gauss equation} and the two \emph{Codazzi equations}, that will be detailed later on.
If the reference forms $\a$ and $\II$ do not satisfy these equations, then no configuration can relax the energy completely, and in fact $\inf E^\text{red}_{t,w}>0$, as was long assessed by physicists and 
recently proved (for a similar form of the energy) in \cite{AKM22,AKM24}.
We say that the ribbon is \Emph{Gauss-incompatible} if $\a$ and $\II$ fail to satisfy the Gauss equation, and \Emph{Codazzi-incompatible} if it is Gauss-compatible but $\a$ and $\II$ fail to satisfy the Codazzi equations.\footnote{While Gauss-incompatibility is the more studied phenomenon in physics, it was recently shown that the shape of rose petals is dictated by Codazzi incompatibility \cite{ZCMS25}.}
More precisely, in this work we will be interested in zeroth-order compatibility, meaning whether the equations are satisfied along the midline $\ell$ (see Definition~\ref{def:deficit}).

%%%%%%%%%%%%%%%%%%%%%%%%%%
\begin{figure}
\begin{center}
\includegraphics[height=7.5cm]{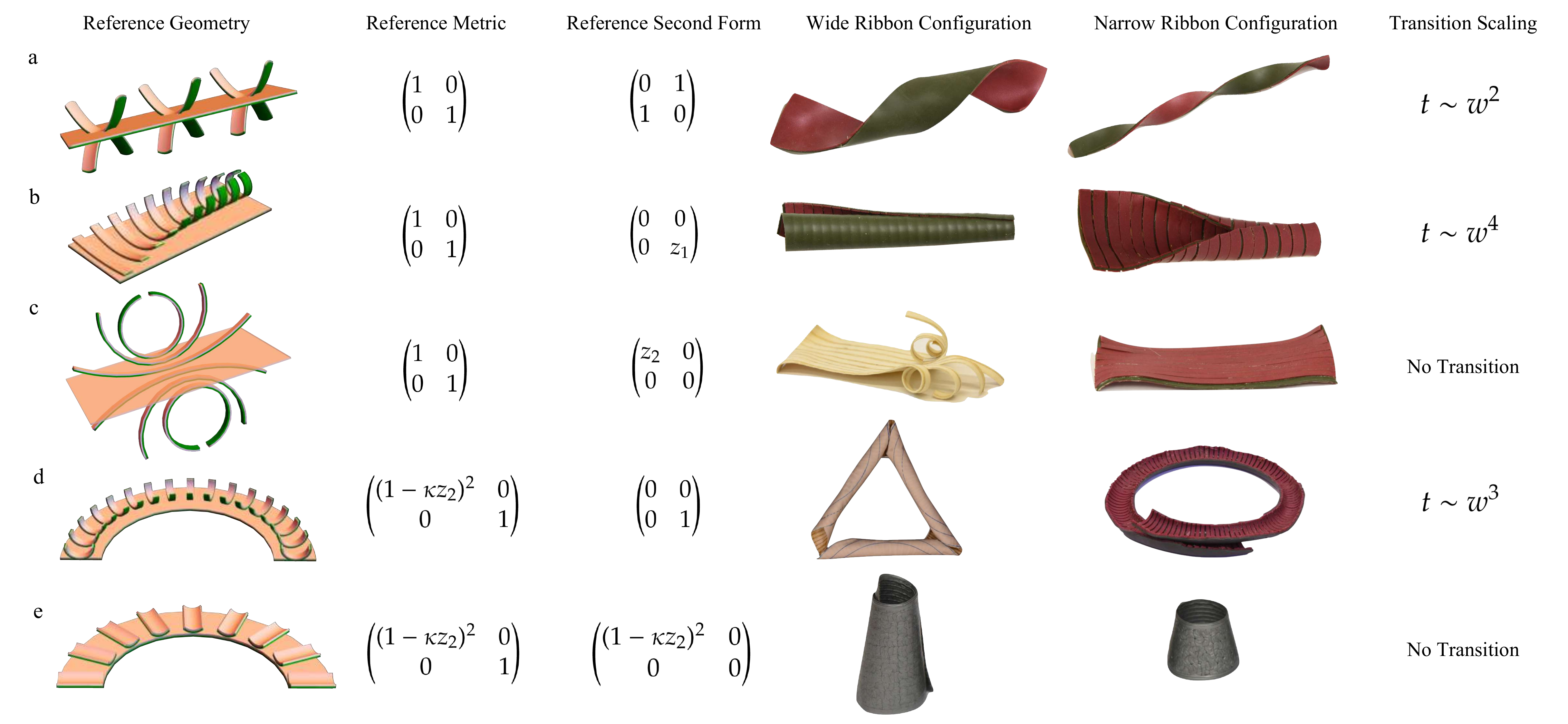}
\end{center}
\caption{Examples of shape transitions in ribbons (figure adapted from \cite{LSSM21}). 
The first column depicts the reference forms $\a$ (flat figure) and $\II$ (curved), that appear in the next two columns. 
Case (a) is Gauss-incompatible, (b)--(e) are Codazzi-incompatible.
The fourth and fifth columns display experimental pictures of elastic ribbons with these reference forms, in the wide and narrow regimes (in case (c) the right end of the wide ribbon is cut to reveal the second form the body wants to achieve).
The sixth column indicates whether there is a (first order) shape transition between narrow and wide ribbons, and the scaling at which it occurs, as predicted by formal asymptotics and confirmed by experiments; for (b) and (d), the predicted scalings are only approximate and the exact scaling is unknown, see \S\ref{sec:wide_examples} for details.
The present work rigorously proves the existence and scaling of shape transitions in case (a) (as well as in all Gauss-incompatible ribbons with a flat midsurface), confirms the presence of a transition in (d), and shows that no transition occurs in (c) and (e).}
\label{fig:ribbons}
\end{figure}

%%%%%%%%%%%%%%%%%%%%%%%%%%
\paragraph{Shape transitions in non-Euclidean ribbons.}
Ribbons interpolate between the behavior of plates ($w\approx L$), which is constrained by the Gauss--Codazzi equations, and the behavior of rods ($w\approx t$), in which there are no constraints and thus any given second fundamental form can always be achieved by an isometry.
As such, for the same intrinsic geometry, one may obtain a major difference between the emerging shapes (i.e., minimizers of \eqref{eq:Ereduced}) depending on the relation between the thickness $t$ and the width $w$.
As mentioned above, this shape transition was observed in plants \cite{AESK11}, molecules \cite{ZLRMD11,GSD16,ZGDS19}, as well as in controlled experiments using various ``material-programing'' techniques \cite{SLS21,LSSM21}.

Loosely speaking, when the ribbon is narrow enough, it adopts a shape whose second fundamental form coincides with the reference form $\II$, at least to a leading order (i.e., along the midline).
A wider ribbon needs to be close to an isometric immersion of the metric $\a$ (similar to a plate/shell) to high enough order, so that the second fundamental form $\II$ might not be achieved due to the Gauss--Codazzi compatibility conditions.

When the reference forms $\a,\II$ are Gauss-incompatible along the midline, formal asymptotic analysis of \eqref{eq:Ereduced} in the small parameters $t,w$ shows that this shape transition indeed occurs when $t\sim w^2$ \cite{AESK11,GSD16}.\footnote{More accurately, one should non-dimensionalize $t$ and $w$ first, e.g., by comparing $t/L$ and $w/L$; to make notation less cumbersome we will think of $t$ and $w$ being already non-dimensionalized.}

As was observed in \cite{SLS21,LSSM21}, the case of Codazzi-incompatibility is more complicated: for some (incompatible) geometries there is no transition at all, while for others there is a transition at $t \sim w^{\alpha}$ for some $\alpha \ne 2$, as shown in Figure~\ref{fig:ribbons}.
These various scenarios were studied in \cite{LSSM21}, using a careful formal asymptotic analysis of the energy \eqref{eq:Ereduced} and the compatibility equations.
For a given geometry, this analysis can predict whether a shape transition will occur, and approximately at which exponent~$\alpha$.

%%%%%%%%%%%%%%%%%%%%%%%%%%%
\paragraph{The three-dimensional model.}
The aim of this paper is to analyze this phenomenon rigorously, using $\Gamma$-convergence, starting from a fully three-dimensional model (rather than a reduced one).
In this model, the elastic ribbon is a smooth Riemannian manifold $(\M_{t,w},\g)$, where, for a natural choice of coordinates, we have that $\M_{t,w}=(0,L)\times (-w/2,w/2) \times (-t/2,t/2)$ and $\g$ is given by
\beq\label{eq:g_estimate}
\begin{split}
\g(z) 
	&= \brk{\begin{matrix}
			(1-\kappa(z_1)z_2)^2 - K^\calS(z_1,0)z_2^2 + O(z_2^3) & 0 & 0\\
			0 & 1 & 0 \\
			0 & 0 & 1
			\end{matrix}}
		-2z_3 \brk{\begin{matrix}
			\II(z_1,z_2) & 0 \\
			0 & 0
			\end{matrix}}
		+O(z_3^2).
\end{split}
\eeq
We shall assume that $\g$ can be smoothly extended to the closure $[0,L]\times [-w/2,w/2] \times [-t/2,t/2]$.
As shown in \S\ref{sec:metrics}, this structure is the general form of a metric in which the ribbon is a $t$-tubular neighborhood of the \Emph{mid-surface} $\Sw = \{z_3=0\}$, which, in turn, is a $w$-tubular neighborhood of the \Emph{midline} $\ell = \{z_2=z_3=0\}$.
In \eqref{eq:g_estimate}, $\kappa$ is the geodesic curvature of $\ell$ in $\Sw$, $K^\calS$ is the Gaussian curvature of $\Sw$, and $\II$ is the second fundamental form of $\Sw$ in $\M_{t,w}$.
The first fundamental form $\a$ in the reduced model above is the restriction of $\g$ to its main $2\times 2$ sub-matrix at $z_3 = 0$.

The energy associated with the elastic ribbon is $E_{t,w}:W^{1,2}(\M_{t,w};\R^3) \to \R\cup \{+\infty\}$, defined by
\[
E_{t,w}(y_t) = \dashint_{\M_{t,w}} \calW(\nabla y_t(z)\, P^{-1}(z))\,\dVol_\g(z)
\]
for $y_t\in W^{1,2}(\M_{t,w};\R^3)$,
where $\dVol_\g$ is the volume form of $\g$, $\dashint$ is the integral divided by the volume of the domain, and $\calW:\R^{3\times 3} \to [0,\infty]$ is the energy density satisfying standard assumptions in elasticity, described in detail in \S\ref{sec:settings}.
Here $P:T\M_{t,w}\to \R^{3\times 3}$ is the so-called \emph{prestrain} or \emph{implant map} (or in context of material defects also \emph{plastic strain} or \emph{crystal scaffold}, see \cite{EKM20}, \cite[\S 2]{KM23})
and is a given orientation-preserving smooth map satisfying $P^T P = \g$.

The analysis is essentially the same for any choice of such $P$, and thus we will choose one that slightly simplifies the analysis; see \S\ref{sec:settings} below.

To quantify the incompatibility we introduce the following definitions.

\begin{definition}\label{def:deficit}
The \textbf{Gauss-deficit} of the ribbon along the midline is
	\beq\label{eq:Gauss_def}
	\DG(z_1) := \det\II(z_1,0) - K^\calS(z_1,0)
	\eeq
	for $z_1\in[0,L]$.
	The ribbon is \Emph{Gauss-incompatible} (along its midline) if $\DG\nequiv 0$.

The \textbf{Codazzi-deficit} of the ribbon along the midline is
	\beq\label{eq:Codazzi_def}
	\DC(z_1) = \left.\brk{\begin{matrix}  \pl_2 \II_{11} - \pl_1 \II_{12} + \kappa( \II_{11}+  \II_{22}) \\  \pl_2 \II_{12} - \pl_1 \II_{22} -\kappa \II_{12} \end{matrix}}\right|_{(z_1,0)}
	\eeq
	for $z_1\in[0,L]$. The ribbon is \Emph{Codazzi-incompatible} (along its midline) if $\DG\equiv 0$ but $\DC\nequiv 0$.
\end{definition}
Note that these quantities coincide with the curvatures of $\g$ along the midline; more precisely, $\DG = R_{1212}$ and $\DC_i = -R_{12i3}$.

%%%%%%%%%%%%%%%%%%%%%%%%%
\paragraph{Main results.}
We are interested in the limiting behavior of $E_{t,w}$ as 
\beq\label{eq:w_ribbon}
\lim_{t\to 0} w= 0, \qquad \lim_{t\to 0}\frac{t}{w}= 0,
\eeq
both in terms of the energy scaling and the behavior of minimizers.
As $w$ tends to zero more and more quickly, we expect the ribbon to become less and less constrained, and thus $\inf E_{t,w}$ to decrease to zero faster.
Since for quite general metrics (see \cite{KS14}, also Corollary~\ref{cor:plates}) non-Euclidean plates (that is, when $w \sim 1$) satisfy $\inf E_{t,w} = O(t^2)$, this bound is expected to hold also in the ribbon case, though it may not be tight.

On a non-technical level, our main results can be summarized by the following two theorems.
%%%%%%%%%
\begin{theorem}[Narrow ribbons]
\label{thm:narrow_informal}
Consider a narrow ribbon in the sense
	\beq\label{eq:w_narrow_ribbon}
	\lim_{t\to 0} \frac{t}{w^2} = \infty,
	\eeq
that is,
\[
w^2 \ll t \ll w \ll 1.
\]
We then have
\begin{enumerate}
	\item Energy scaling: $\inf E_{t,w} \sim w^4$ if the ribbon is Gauss-incompatible, and $\inf E_{t,w} \sim t^2w^2$ if it is Codazzi-incompatible.
	\item Approximate minimizers behave as follows:
		\begin{enumerate}
		\item They tend to an isometric immersion of the midline. 
		\item The geodesic curvature of the immersed midline tends to $\kappa$.
		\item The second fundamental form of the mid-surface along the midline tends to $\II|_\ell$.
		\item The Gaussian curvature of the mid-surface along the midline tends to $K^\calS|_\ell$ if and only if the ribbon is Gauss-compatible.
		\end{enumerate}
\end{enumerate}
\end{theorem}

%%%%%%%%
\begin{theorem}[Wide ribbons]
\label{thm:wide_informal}
For ribbons whose mid-surface is flat (that is, $K^\calS = 0$), consider the wide ribbon limit in which we first take $t\to 0$ and then $w\to 0$.
We then have that approximate minimizers tend to isometric immersions of the midline, and the geodesic curvature of the immersed midline tends to $\kappa$.
Moreover,
\begin{enumerate}
	\item if the ribbon is Gauss-incompatible, then $\inf E_{t,w} \sim t^2$, and the second fundamental form of the mid-surface along the midline converges to a limit that \emph{differs} from $\II|_\ell$;
	\item if the ribbon is Codazzi-incompatible, then $\inf E_{t,w} = o(t^2)$, and the second fundamental form of the mid-surface along the midline \emph{does} tend to $\II|_\ell$ (strongly in $L^2$).
	\begin{enumerate}
		\item If, furthermore, we assume that
			\beq
			\label{eq:assumption_II11}
			 \II_{11}(z_1,0) \ne 0 \quad \text{ for every } z_1\in [0,L],
			 \eeq
			 then $\inf E_{t,w} \sim t^2w^2$, and the fluctuation from $\II|_\ell$ of the second fundamental form of approximate minimizers is of order $\ll w$.
		\item If assumption \eqref{eq:assumption_II11} is violated, there are examples where $\inf E_{t,w} \gg t^2 w^{1+\e}$, and in addition the fluctuation from $\II|_\ell$ of approximate minimizers  is of order $\gg w^{(1+\e)/2}$, for any $\e>0$.
	\end{enumerate} 
\end{enumerate}
\end{theorem}

Even though our results for wide ribbons are restricted to flat mid-surfaces, the analysis presented above covers all the geometries in Figure~\ref{fig:ribbons}, as well as other physically-motivated geometries \cite{AESK11,GSD16,ZGDS19,SLS21}.
This assumption arises since wider ribbons are more isometrically-constrained, and our understanding of the space of isometric immersions (of $W^{2,2}$ regularity) is much more complete in the flat case.

%%%%%%%%%
Building on Theorems~\ref{thm:narrow_informal}--\ref{thm:wide_informal} we deduce the following regarding the behavior of the minimal energy.
\begin{corollary}\label{cor:informal}
For ribbons whose mid-surface is flat, the following holds:
\begin{enumerate}
\item For Gauss-incompatible ribbons $\inf E_{t,w} \sim \min\{t^2,w^4\}$ and there is a transition in the behavior of the second fundamental form between wide and narrow ribbons.
The energy scaling statement holds for non-flat ribbons as well, as long as their mid-surface can be $W^{2,\infty}$-isometrically immersed in $\R^3$.
\item For Codazzi-incompatible ribbons, if condition \eqref{eq:assumption_II11} holds, then $\inf E_{t,w} \sim t^2w^2$ in both regimes we consider (no transition in the energy scaling).
Otherwise, there are examples in which a transition in the energy scaling occurs.
\end{enumerate}
\end{corollary}
%%%%%%%

\paragraph{Necessary condition for Codazzi shape transition and its geometrical meaning.}
Condition~\eqref{eq:assumption_II11} is satisfied in the ribbon geometries (c) and (e) in Figure~\ref{fig:ribbons}, whereas it is violated in the geometries (b) and (d) in Figure~\ref{fig:ribbons}.
We thus understand the vanishing of $\II_{11}|_\ell$ (i.e., the violation of \eqref{eq:assumption_II11}) as a necessary condition for the occurrence of a shape transition in a Codazzi-incompatible ribbon.

The non-vanishing condition \eqref{eq:assumption_II11} has a geometric interpretation.
Since $\Sw$ is assumed to be flat, its isometric immersions are necessarily ruled surfaces, with ruling directions given by the null-vectors of the second fundamental form.
Condition~\eqref{eq:assumption_II11} ensures that these rulings intersect the midline transversally; this, in turn, guarantees the existence of an isometric immersion of $\Sw$, whose second fundamental form coincides with $\II|_\ell$ along the midline. The elastic energy of these isometric immersions is of order $t^2w^2$.
This will be detailed in the proof of Theorem~\ref{thm:Gconv_wide_Codazzi}. 
In other words, when \eqref{eq:assumption_II11} is satisfied, although no isometry realizes $\II$ as the second fundamental form on the entire midsurface $\Sw$, there do exist isometries matching $\II$ along the midline. 
The incompatibility is therefore mild and not sufficient to trigger a shape transition.

\paragraph{``Microscopic'' vs. ``mesoscopic'': The different nature of Gauss and Codazzi shape transitions.}
It follows from Theorem~\ref{thm:narrow_informal}(2c) and Theorem~\ref{thm:wide_informal}(2) that in all Codazzi-incompatible ribbons (both narrow and wide) the second fundamental forms converge to $\II|_\ell$. 
However, in wide Codazzi-incompatible ribbons for which \eqref{eq:assumption_II11} is violated, this convergence may be slower, giving rise to the visible shape transition observed in these cases.
Our analysis, therefore, reveals a fundamental difference between shape transitions in Gauss-incompatible and Codazzi-incompatible ribbons:
In the former, the transition manifests in a different limiting second form along the midline (i.e., it is ``microscopic'', persisting all the way as $w\to 0$); 
in the latter, when a transition occurs, it is instead reflected in the rate at which the second form converges to the reference one along the midline (i.e., it is ``mesoscopic'', observable for sufficiently small $w$ but not in the limit $w\to 0$).
\medskip

Theorems~\ref{thm:narrow_informal}--\ref{thm:wide_informal} are proved by calculating several $\Gamma$-limits (and by establishing their corresponding compactness results):
\begin{itemize}
	\item Theorem~\ref{thm:narrow_informal} is obtained by computing the $\Gamma$-limits of $\frac{1}{w^4}E_{t,w}$ (for Gauss-incompatibility) and of $\frac{1}{t^2w^2} E_{t,w}$ (for Codazzi-incompatibility), as ${t,w\to 0}$ simultaneously, under the assumption \eqref{eq:w_narrow_ribbon}. In both cases we show that the limit energy is finite and strictly positive, thereby establishing the energy scaling.
		This analysis is presented in \S\ref{sec:narrow}.
		Furthermore, in \S\ref{sec:wide_from_narrow} we build upon these results to prove the general scaling statement in Corollary~\ref{cor:informal}(1).
	\item To prove Theorem~\ref{thm:wide_informal} we first compute (in \S\ref{sec:plate_limit}) the $\Gamma$-limit of $\frac{1}{t^2} E_{t,w}$, as ${t\to 0}$, for a fixed~$w$.
	This computation follows a similar approach to previous works on non-Euclidean plates and shells \cite{KS14,BLS16,LL20}.
	Next, in \S\ref{sec:wide_gauss} we obtain the ribbon limit by letting ${w\to 0}$, similarly to \cite{FHMP16b}. 
		We analyze this limit and show that it is strictly positive if and only if the ribbon is Gauss-incompatible (Proposition~\ref{prop:Jbehavior}). Additionally, we establish the behavior of the limiting second form as stated in Theorem~\ref{thm:wide_informal}(1).
		For the Codazzi-incompatible case, in \S\ref{sec:wide_codazzi} we examine the behavior of minimizers of the same energy (Proposition~\ref{prop:basic}). Under the assumption \eqref{eq:assumption_II11}, we compute the finer limit of the functional  $\Gamma\text{-}\lim_{t\to 0} \frac{1}{t^2} E_{t,w}$ scaled by $w^2$, as $w\to0$, showing that it is finite and positive, and conclude the proofs of Theorem~\ref{thm:wide_informal}(2b) and Corollary~\ref{cor:informal} (see Corollary~\ref{cor:liminf}).
		Finally (in \S\ref{sec:wide_examples}) we analyze the geometry of Figure~\ref{fig:ribbons}(d) where \eqref{eq:assumption_II11} is violated and show that it corresponds to a higher energy scaling (Proposition~\ref{prop:ribbon_d}).
\end{itemize}

The initial functional in all these $\Gamma$-limits  is the full three-dimensional energy, and the limit is a one-dimensional model along the mid-line (in the case of the iterated limit, the first limit is a two-dimensional plate model).
The limit energies are one-dimensional like rods, however they retain the information of the second fundamental form of the two-dimensional mid-surface (along the midline). 

%%%%%%%
\paragraph{Relation to previous works.}
These $\Gamma$-limits build upon existing results for incompatible plates (see~\cite{Lew23} for a recent extensive review), and for ribbons, both Euclidean \cite{FMP12,FMP13} and non-Euclidean \cite{FHMP16b}.
In this work, several new challenges emerge compared to the existing literature. We now outline some of these key difficulties and how they are addressed.

For narrow ribbons, the natural energy scaling (either $w^4$ or $t^2w^2$) falls between $t^2$ and $t^4$, placing it within the so-called ``linearized Kirchhoff regime''.
All previous analyses in this regime --- Euclidean plates \cite{FJM06}, weakly-prestrained plates \cite[\S14.7]{Lew23}, shells \cite{LMP11}, shallow shells \cite{LMP11b}, and Euclidean ribbons \cite{FMP13} --- results in limiting energies that involve only bending (out-of-plane) contributions. This is because in-plane displacements can be neglected in the lower bound, an assumption that is then matched in the upper bound construction.
This is no longer the case for Gauss-incompatible ribbons in their natural $w^4$-scaling: Gauss-incompatibility manifests in the stretching energy, and configurations whose stretching are $o(w^4)$ do not exist.
To address this issue, we show that the in-plane stresses possess additional structure (Lemma~\ref{lemma:G}) and use this structure to derive a sharper lower bound, which can be matched by a recovery sequence.
As a result, an extra term that cannot vanish appears in the energy, proving the tightness of the energy scaling.

A new geometric challenge is that the coordinates $z=(z_1,z_2,z_3)$ in $\M_{t,w}$ are natural for describing the ribbon and its geometry (as in \eqref{eq:g_estimate}), but they are insufficient for the analysis in the incompatible case, as they do not suggest how low energy configurations look like (from a technical viewpoint, one can see that the metric \eqref{eq:g_estimate} differs from the identity matrix by terms that are much larger than our energy scaling).
Thus, we need to define an \emph{associated Euclidean ribbon} to the incompatible ribbon, which is an embedded ribbon $\Psi:\M_{t,w} \to \R^3$, such that $\Psi(\M_{t,w})$ is a curved ribbon in Euclidean space whose geometry is ``close enough'' to the original geometry.
The analysis then alternates between the coordinates $z$ and the Euclidean coordinates on $\Psi(\M_{t,w})$.

For wide ribbons, as mentioned before, the analysis is a refinement of the one of \cite{FHMP16b} --- our analysis requires a slightly more general setting, and we make a more detailed analysis of the limiting energy  and the relation of its minimum to the underlying geometry. We then analyze a finer (lower) energy scaling, for which a more complicated construction of a recovery sequence is required.

Finally, it is interesting to compare our results to \cite{MS19}: there it was shown that for plates ($w\sim 1$) the energy scaling is $t^2$ unless the Gauss and Codazzi deficits are zero (see also \cite{Lew23} for the dependence of the energy scaling of incompatible plates on their curvature), and for rods ($w\sim t$) it is $t^4$ (generically).
This work shows how ribbons interpolate between these two scalings, and how the Gauss and Codazzi deficits are reflected differently in the energy scaling when ribbons (rather than plates) are considered.

\paragraph{Future directions.}
This work is the first to rigorously address shape transitions in non-Euclidean ribbons. Yet our understanding is far from complete and there are many directions for further investigation:
\begin{itemize}
	\item \emph{Energy scaling of transitions in Codazzi-incompatible ribbons}: 
		In this work we provide a necessary condition for the occurrence of shape transitions in Codazzi-incompatible ribbons. However, a full analysis of the energy scaling of a transition, when it occurs, is still missing.
		These transitions seem to be geometry-dependent (unlike for Gauss-incompatible ribbons), so an interesting open direction is to study the transition in specific examples, in particular the ones of (b) and (d) in Figure~\ref{fig:ribbons}.
		Theorem~\ref{thm:narrow_informal} provides the energy scaling and asymptotic behavior when the ribbon is narrow enough. In \S\ref{sec:wide_examples} we summarize the best known ansatzes and the conjectured energy scalings for wide ribbons in these geometries.
	\item \emph{The wide ribbon limit}:
		In this work we analyze the double limit, as $t,w\to 0$ simultaneously, for a narrow ribbon (i.e., $w^2 \ll t \ll w \ll 1$), and the iterated limit $\lim_{w\to 0} \lim_{t\to 0}$, which is interpreted as a ``very wide'' ribbon (one could think of this iterated limit as implying $t\ll w^\alpha$ for any $\alpha >0$).
		This leaves a gap of obtaining the wide ribbon double limit $t\ll w^2 \ll 1$, in particular in the Gauss-incompatible regime $t^2$.
		This regime is open even for Euclidean ribbons.\footnote{Another regime not considered in this work is the ``very narrow ribbon'' limit, in which one sets $w= t/\alpha$ (so that $\alpha=t/w$ is the thickness-to-width aspect ratio) and considers the iterated limit $\lim_{\alpha \to 0}\lim_{t\to 0}$. The first limit yields a rod model whose effective energy depends on the cross-sectional aspect ratio $\alpha$, while the subsequent limit $\alpha\to0$ leads to a ribbon regime.
%; that is, a rod whose cross section's aspect ratio tends to zero.
		As in the narrow ribbon case, one can show that the natural energy scaling is $w^4$ for Gauss-incompatible ribbons and $t^2w^2$ for Codazzi-incompatible ones.
		Since this limit does not provide additional insight into the shape transition studied here, we do not pursue it in the present paper and leave it for future work.}
	\item \emph{The non-flat case}: While our analysis of narrow ribbons is valid for arbitrary geometries, the (very) wide ribbons analysis is currently valid only for ribbons with flat mid-surface $\Sw$.
		Many ribbons that were studied in the physics literature are of this type \cite{AESK11,GSD16,ZGDS19,SLS21,LSSM21}, however not all of them are \cite{Efr15,LS16,HWQSH18}.
		It would be interesting to extend this study to ribbons whose mid-surface is elliptic or hyperbolic.
		Since elliptic (resp.\ hyperbolic) isometric immersions are more (resp.\ less) constrained than flat ones, one might encounter different behaviors, in particular in the Codazzi-incompatible case.
	\item \emph{Non-smooth metrics}: In this work the reference metric $\g$ is assumed to be smooth.
		While this assumption is common (both in the physics and in the mathematical studies of incompatible elasticity), in practice many physical examples have non-smooth metrics (in particular piecewise constant); see, e.g., the bilayer structure in the experiments depicted in  Figure~\ref{fig:ribbons}(a)--(d).
		Thus, it would be desirable to extend the analysis to this case as well.
		There, one should expect that the energy scaling is $t^2$ for all regimes, due to an excess energy that appears in rough metrics (similar to the $t^2$ scaling in multilayered non-Euclidean plates \cite{Sch07} and rods \cite{KO18}); yet, we expect the behavior of minimizers to be similar to the smooth case and exhibit different ``narrow'' and ``wide'' regimes.
\end{itemize}

%%%%%%%%
\paragraph{Structure of the paper.}
In \S\ref{sec:geometry} we analyze the geometry of non-Euclidean ribbons, and define their associated Euclidean ones.
In \S\ref{sec:settings} we define the elastic energy we analyze in this work.
In \S\ref{sec:narrow} we analyze narrow ribbons, i.e., the $\Gamma$-limit under the assumption $w^2 \ll t\ll w\ll1$.
From this analysis we deduce in \S\ref{sec:wide_from_narrow} the general energy scaling of Gauss incompatible ribbons.
In \S\ref{sec:wide} we analyze wide ribbons, i.e., the double $\Gamma$-limit, first with respect to $t\to0$ and then with respect to $w\to0$. 

%%%%%%%%%%%%%%%%%%%%%%%%%%%%%%%
\paragraph{Acknowledgements.}
Both authors acknowledge support from the Vigevani Foundation and the Hausdorff Institute for Mathematics at the University of Bonn, funded by the Deutsche Forschungsgemeinschaft (DFG, German Research Foundation) under Germany Excellence Strategy  EXC-2047/1 390685813, as part of the Trimester Program ``Mathematics for Complex Materials''.
Part of this work was written when CM was visiting the University of Toronto and the Fields Institute; CM is grateful for their hospitality.
CM was partially supported by ISF grants 1269/19 and 2304/24 and BSF grant 2022076. 
MGM acknowledges support from PRIN 2022 (Project no. 2022J4FYNJ), funded by MUR, Italy, and the European Union -- Next Generation EU, Mission~4 Component~1 CUP~F53D23002760006.
Finally, the authors thank Meital Maor for her help with the figure.

\paragraph{Statements and declarations.} The authors have no competing interests to declare that are relevant to the content of this article.
Data sharing not applicable to this article as no datasets were generated or analyzed during the current study.

%%%%%%%%%%%%%%%%%%%%%%%%%%%%%%%%%%%%%%%%%%%%%%
\section{The geometry of ribbons}
\label{sec:geometry}

%%%%%%
\subsection{Metrics on ribbons}
\label{sec:metrics}
We abstractly define a non-Euclidean ribbon as follows:
Let $(\M,\g)$ be a three-dimensional Riemannian manifold and let $\calS\subset \M$ be a two-dimensional submanifold.
Let $\ell:(0,L)\to \calS$ be a curve in $\calS$, in a natural parametrization.
Let $\calS_w$ be a $w$-tubular neighborhood of $\ell$ in $\calS$, and let $\M_{t,w}$ be a $t$-tubular neighborhood of $\calS_w$ in $\M$.
The set $\M_{t,w}$ is then a \Emph{non-Euclidean ribbon}, whose \Emph{mid-surface} is $\Sw$ and its \Emph{midline} is $\ell$.

We now present the natural (Fermi) coordinates on $\M_{t,w}$, and show the expansion formula \eqref{eq:g_estimate} in these coordinates: 
Let $\n:\calS\to T\M|_S$ be a unit normal vector field, and for $p\in \calS$, let $\nu_p:(-t/2,t/2)\to\M$ be the geodesic defined by
\[
\nu_p(0) = p, \quad \dot\nu_p(0) = \n(p).
\]
Let $\N:\ell \to T\calS|_\ell$ be a smooth normal vector field in $\calS$, and for $p\in \ell$, let $\eta_p:(-w/2,w/2)\to \calS$ be the $\calS$-geodesic defined by
\[
\eta_p(0) = p, \quad \dot \eta_p(0) = \N(p).
\]
We now define natural coordinates $\psi: (0,L)\times (-w/2,w/2)\times (-t/2,t/2) \to \M_{t,w}$ by
\[
\psi(z_1,z_2,z_3) = \nu_{\eta_{\ell(z_1)}(z_2)}(z_3).
\]
That is, in these coordinates $z_1$ is the arclength coordinate along $\ell = \{z_2 = z_3 = 0\}$, from which the direction $z_2$ emanates orthogonally along geodesics of $\Sw = \{z_3 = 0\}$, and $z_3$ emanates orthogonally from $\Sw$ along geodesics in $\M$.

%%%%%%%%
\begin{lemma}\label{lem:ribbon_metric}
The metric $\g$ on $\M_{t,w}$ expressed in the natural coordinates given by $\psi$ is given by \eqref{eq:g_estimate}.
\end{lemma}

\begin{proof}
By definition, we have that 
\[
\at{\frac{\pl}{\pl z_3}}_{(z_1,z_2,0)} = \n(\psi(z_1,z_2,0)),
\]
which is perpendicular to $T_{\psi(z_1,z_2,0)}\calS$.
Thus
\beq\label{eq:g_and_a}
\g(z_1,z_2,0) 
	= \brk{\begin{matrix}
			\a(z_1,z_2) & 0\\
			0 & 1
			\end{matrix}}.
\eeq
Also, we have that $\frac{\pl}{\pl z_3}$ at any point is the $\dot\nu_p(z_3)$ for some $p$, and thus a unit vector; thus $\g_{33} = 1$ everywhere.
Furthermore, for $i,j=1,2$, at a point $(z_1,z_2,0)$, we have
\[
\begin{split}
\pl_3 \g_{ij} &= \g(\nabla_{\pl_3} \pl_i, \pl_j) + \g( \pl_i, \nabla_{\pl_3} \pl_j) \\
	&= \g(\nabla_{\pl_i} \n, \pl_j) + \g( \pl_i, \nabla_{\pl_j} \n) = -2\II_{ij},
\end{split}
\]
where $\nabla$ is the Levi-Civita connection of $\g$, whose symmetry was used in the transition to the second line.
Similarly,
\[
\begin{split}
\pl_3 \g_{i3} &= \g(\nabla_{\pl_3} \pl_3, \pl_i) + \g( \pl_3, \nabla_{\pl_3} \pl_i) = \g(\pl_3, \nabla_{\pl_3} \pl_i) \\
	&\g( \pl_3, \nabla_{\pl_i} \pl_3) = \frac{1}{2} \pl_i \g_{33} = 0,
\end{split}
\]
where we used that $\nabla_{\pl_3} \pl_3 = 0$ since $\n$ is a geodesic.
This last equation holds everywhere.
Thus
\[
\g(z) 
	= \brk{\begin{matrix}
			\a(z_1,z_2) & 0\\
			0 & 1
			\end{matrix}}
		-2z_3 \brk{\begin{matrix}
			\II(z_1,z_2) & 0 \\
			0 & 0
			\end{matrix}}
		+O(z_3^2),
\]
where the $O(z_3^2)$ term is only in the upper $2\times 2$ minor.

It remains to obtain the expansion of $\g|_\calS(z_1,z_2)$.
Since $z_2$ is in the direction of normal geodesics to $\ell$ in $\calS$ (like $z_3$ is with respect to $\calS$ in $\M$), and since $\kappa$ is the second fundamental form of $\ell$ in $\calS$, the same argument as before shows that
\[
\a(z_1,z_2) = \brk{\begin{matrix}
			1 - 2\kappa(z_1)z_2 + O(z_2^2) & 0\\
			0 & 1
			\end{matrix}}.
\]
Now note that $\psi(z_1,z_2,0) = \eta_{\ell(z_1)}(z_2)$ is a family of $\calS$-geodesics with parameter $z_1$.
Therefore, for a fixed $z_1$,
\[
J(z_2) := \pl_1 \psi(z_1,z_2,0)
\]
is a Jacobi field along $\eta_{\ell(z_1)}$ satisfying the initial conditions
\[
J(0) = \pl_1\psi(z_1,0,0) = \pl_1|_{(z_1,0,0)}
\]
and
\[
\begin{split}
J'(0) & := \at{\frac{D}{d\tau} J}_{\tau=0} = \at{\frac{D}{\pl \tau}}_{\tau=0} \at{\frac{\pl}{\pl s}}_{s=z_1} \psi(s,\tau,0) = \at{\frac{D}{\pl s}}_{s=z_1}  \at{\frac{\pl}{\pl \tau}}_{\tau=0} \psi(s,\tau,0) \\
& = \at{\frac{D}{\pl s}}_{s=z_1} \N(\ell(z_1)) = \nabla^\calS_{\pl_1} \pl_2  = -\kappa(z_1)\pl_1,
\end{split}
\]
where $\frac{D}{\pl \tau}$ is the covariant derivative along $\tau$ with respect to the Levi-Civita connection $\nabla^\calS$ of $\calS$, and similarly for $\frac{D}{\pl s}$.
The change of order of differentiation follows from the symmetry of the connection \cite[\S3, Lemma~3.4]{DoC92}.\footnote{The expression of $J'(0)$ could also be obtained by noting that $\g(J'(0),\pl_1|_{(z_1,0)}) = \frac{1}{2}\pl_2\g_{11}(z_1,0) = -\kappa(z_1)$ and $\g(J'(0),\pl_2|_{(z_1,0)}) \allowbreak = \pl_2 \g_{12}(z_1,0)- \g(\pl_1, \nabla_{\pl_2}\pl_2) = 0$, where in the last equality we use the fact that $\nabla_{\pl_2}\pl_2=0$ since the $z_2$ direction is along geodesics.}
We can now use the Jacobi equation \cite[\S3, eq.~(1)]{DoC92}\footnote{The sign convention for the curvature in \cite{DoC92} is converse to the one used here (and in most textbooks).}
\[
J''(\tau) = \calR^\calS(\eta'_{\ell(z_1)}(\tau), J(\tau)) \eta'_{\ell(z_1)}(\tau),
\]
where $\calR^\calS$ is the Riemann curvature tensor of $\calS$, to obtain
\[
\begin{split}
\at{\pl_\tau^2 \g(J(\tau),J(\tau))}_{\tau=0} 
	&= 2\at{\pl_\tau \g(J'(\tau),J(\tau))}_{\tau=0} = 2\g(J''(0),J(0)) + 2\g(J'(0),J'(0)) \\
	&= 2\g(\calR^\calS(\pl_2,\pl_1)\pl_2,\pl_1) + 2\kappa^2(z_1) = -2K^S(z_1,0) + 2\kappa^2(z_1).
\end{split}
\]
Since $\a_{11}(z_1,z_2) = \g(J(z_2),J(z_2))$, this completes the proof.
\end{proof}

%%%%%%
\subsection{Flat mid-surfaces}

In this section we specify Lemma~\ref{lem:ribbon_metric} to the case of a flat mid-surface.

\begin{lemma}\label{lem:ribbon_metric_flat}
When the mid-surface $\Sw$ is flat, that is, $K^\calS \equiv 0$, the metric expression \eqref{eq:g_estimate} simplifies to
\beq
\label{eq:g_estimate_flat}
\begin{split}
\g(z) 
	&= \brk{\begin{matrix}
			(1-\kappa(z_1)z_2)^2 & 0 & 0\\
			0 & 1 & 0 \\
			0 & 0 & 1
			\end{matrix}}
		-2z_3 \brk{\begin{matrix}
			\II(z_1,z_2) & 0 \\
			0 & 0
			\end{matrix}}
		+O(z_3^2).
\end{split}
\eeq
That is, the induced metric on $\Sw$ is given by
\beq\label{eq:a_flat}
\a = \mat{(1-\kappa(z_1) z_2)^2 & 0 \\ 0 & 1}.
\eeq
\end{lemma}

\begin{proof}
We already know from Lemma~\ref{lem:ribbon_metric} that $\a_{11}(z_1,z_2) = 1 -2\kappa(z_1)z_2 + O(z_2^2)$.
As shown in the last part of the proof of Lemma~\ref{lem:ribbon_metric}, the Jacobi equation for the Jacobi field $J(z_2) = \pl_1 \psi(z_1,z_2,0)$ is
\[
J''(\tau) = 0, \quad J(0) = \pl_1|_{(z_1,0,0)},\,\, J'(0) = -\kappa(z_1)\pl_1|_{(z_1,0,0)},
\]
where we used that the flatness of $\Sw$ implies $\calR^\calS \equiv 0$.
Noting that for a fixed $z_1$, $\a_{11}(z_1,z_2) = \a(\pl_1 \psi(z_1,z_2,0),\pl_1 \psi(z_1,z_2,0)) = \a(J(z_2),J(z_2))$, we obtain that
\[
\begin{split}
\pl_\tau^2 \a(J(\tau),J(\tau)) &= 2\pl_\tau \a(J'(\tau), J(\tau)) = 2\a(J'(\tau), J'(\tau)) + \a(J''(\tau), J(\tau)) \\
	&= 2\a(J'(\tau), J'(\tau)) \\
	&= 2\a(J'(0), J'(0)) = 2\kappa(z_1)^2 \a_{11}(z_1,0,0) = 2\kappa(z_1)^2,
\end{split}
\]
where in the transition to the second line we used that $J''(\tau) = 0$, in the transition to the last line that $J'' = 0$ implies that $J'$ is a parallel vector field and thus of a constant norm.
Thus we know that
\[
\pl_{z_2}^2 \a_{11}(z_1,z_2) = 2\kappa(z_1)^2, \quad \a_{11}(z_1,0) = 1, \quad \pl_{z_2} \a_{11}(z_1,0) = -2\kappa(z_1),
\] 
from which the result follows.
\end{proof}
\medskip

Let $\Rot:[0,L] \to \SO(2)$ be given by
\beq\label{eq:R_def}
\Rot' = \Rot \brk{\begin{matrix} 0 & -\kappa \\ \kappa &0 \end{matrix}},
\eeq
for some initial condition $\Rot(0)\in \SO(2)$.
For $w_0>0$ small enough we define $\chi:[0,L]\times [-w_0/2,w_0/2]\to \R^2$ via
\beq\label{eq:def_chi}
\chi(z_1,z_2) = \int_0^{z_1} \Rot(\tau)e_1\,d\tau+ z_2 \Rot(z_1)e_2.
\eeq
It is immediate to check that, for every $w\in (0,w_0)$, $\chi {|_{\Sw}}$ is an isometric immersion of $(\Sw,\a)$ into $\R^2$.
We denote the image of $\chi|_{\Sw}$ by $\tSw$, and its midline by $\tilde{\ell}$, that is,
\beq\label{eq:def_tell}
\tSw = \chi(\Sw), \qquad \tilde{\ell} = \chi([0,L]\times \{0\}).
\eeq

%%%%%%
\subsection{The associated Euclidean ribbon}

For a given non-Euclidean ribbon $(\M_{t,w},\g)$ (whose mid-surface is not necessarily flat) we now construct a ribbon that is embedded in $\R^3$, and agrees with the geometry of the non-Euclidean ribbon to leading orders.
That is, we construct a map
\[
\Psi: (0,L) \times (-w/2,w/2) \times (-t/2,t/2) \to \R^3
\]
such that $D\Psi^T D\Psi$ is close enough to $\g$.
We call this map \Emph{the Euclidean ribbon associated with $(\M_{t,w},\g)$} and denote its image by $\W_{t,w}$.
This map will be essential in the analysis of narrow ribbons, as low-energy configurations will become asymptotically close to $\Psi$, and the (Euclidean) coordinates on $\W_{t,w}$ will be required for the analysis.

We start by constructing the two-dimensional mid-surface:
\begin{proposition}\label{prop:Phi}
Given a smooth function $\kappa: [0,L]\to \R$ and a symmetric matrix field $\II^0: [0,L] \to \R^{2\times 2}$, there exists $w$ small enough such that there exists an immersion
\[
\Phi: (0,L) \times (-w/2,w/2) \to \R^3,
\]  
whose first fundamental form is
\[
D\Phi^T D\Phi(z_1,z_2) = \brk{\begin{matrix}
			(1-\kappa(z_1)z_2)^2 - \det \II^0(z_1)z_2^2  & 0 \\
			0 & 1  
			\end{matrix}} + O(z_2^3),
\]
and its second fundamental form is
\[
D^2\Phi \cdot \tn(z_1,z_2) = \II^0(z_1) + O(z_2),
\]
where $\tn$ is the normal to $\Phi$.
\end{proposition}

\begin{proof}
We note that $z_1\mapsto \Phi(z_1,0)$ has to be a curve immersed in the surface define by $\Phi$, and, assuming that $D\Phi^TD\Phi(z_1,0) = I$, the frame $(e_1(z_1),e_2(z_1),e_3(z_1)) = (\pl_1 \Phi,\pl_2\Phi, \tn)|_{(z_1,0)}$ is the Darboux frame of this curve, with geodesic curvature $\kappa$, normal curvature $\II^0_{11}$ and relative torsion $\II^0_{12}$.
Thus, this frame satisfies the Darboux frame equations
\[
\brk{\begin{matrix}e_1 \\ e_2 \\ e_3 \end{matrix}}'
= 
\brk{\begin{matrix}0 & \kappa & \II^0_{11} \\ -\kappa & 0 & \II^0_{12} \\ -\II^0_{11} & -\II^0_{12} & 0 \end{matrix}}
\brk{\begin{matrix}e_1 \\ e_2 \\ e_3 \end{matrix}}.
\]
We solve this system of ODEs and set
\[
\Phi(z_1,z_2) = \int_0^{z_1} e_1(s)\,ds + z_2 e_2 + \frac{1}{2}z_2^2 \II^0_{22} e_3 - \frac{1}{6}z_2^3  \II^0_{22} \brk{  \II^0_{12}e_1 + \II^0_{22} e_2},
\]
where $e_1,e_2$ and $\II^0$ are evaluated at $z_1$.
Noting that the normal to $\Phi$ at $z_2=0$ is $e_3$ by construction, a direct calculation shows that the first and second forms of $\Phi$ have the correct expansion (the last term in the definition of $\Phi$ is chosen to eliminate $z_2^2$ terms from the $12$ and $22$ terms of the first fundamental form).
\end{proof}
\medskip

Note that, even for small $w$, $\Phi$ is only guaranteed to be an immersion, not an embedding (i.e., it may fail to be injective), depending on whether its midline intersects itself or not.
In the following we will assume, for simplicity, that the midline does not intersect itself, and thus $\Phi$ is indeed an embedding for $w$ small enough. 
If this is not the case, one can always find a finite open cover of $[0,L]$ such that the restriction of the midline to each subinterval is an embedded curve. One can then perform locally all the subsequent analysis.

Let now $(\M_{t,w},\g)$ be a non-Euclidean ribbon. 
We apply Proposition~\ref{prop:Phi} with $\kappa$ as in \eqref{eq:g_estimate} (that is, the geodesic curvature of $\ell$ in $\Sw$) and
\[
\II^0(z_1) = \II(z_1,0), 
\]
that is, $\II^0$ is the restriction of the second fundamental form of  $\Sw$ in $\M_{t,w}$ to the midline.
We define $\Psi : (0,L)\times (-w/2,w/2) \times (-t/2,t/2) \to \R^3$ as the three-dimensional ribbon of thickness $t$ associated with $\Phi$, that is,
\[
\Psi(z) = \Phi(z_1,z_2) + z_3 \tn(z_1,z_2).
\]
We denote its image by $\W_{t,w}\subset \R^3$.
Since $\Phi$ is an embedding by assumption, so is $\Psi$ (for small enough $t$), and thus $\W_{t,w}$ is bi-Lipschitz equivalent to $(0,L)\times (-w/2,w/2) \times (-t/2,t/2)$, with a bi-Lipschitz constant independent of $t$ and $w$.
A direct application of Proposition~\ref{prop:Phi} shows that
\beq\label{eq:DPsi_euc}
\begin{split}
D\Psi^TD\Psi  
	&= \brk{\begin{matrix}
			(1-\kappa(z_1)z_2)^2 - \det \II^0(z_1)z_2^2 & 0 & 0\\
			0 & 1 & 0 \\
			0 & 0 & 1
			\end{matrix}}
		-2z_3 \brk{\begin{matrix}
			\II^0(z_1)+O(z_2) & 0 \\
			0 & 0
			\end{matrix}}
		+O(|z_2|^3).
\end{split} 
\eeq
In particular, by Lemma~\ref{lem:ribbon_metric} we have that
\beq\label{eq:DPsi_g}
D\Psi^TD\Psi = \g +
	\begin{cases} O(z_2^2,z_3^2) & \text{ if }\DG\nequiv 0, \\
			O(|z_2z_3|,|z_2|^3,z_3^2) & \text{ if }\DG\equiv 0,
	\end{cases}
\eeq
where $\DG$ is the Gauss-deficit, according to Definition~\ref{def:deficit},
Furthermore, since the metric $D\Psi^TD\Psi$ satisfies the Gauss-Codazzi equations by definition, the $11$ and $12$ components of the term of order $z_2z_3$ can be expressed in terms
of $\II^0$ alone. More explicitly, these components are $2z_2z_3\partial_2(\partial_i\Phi \cdot \partial_1\tn)|_{(z_1,0,0)}$, $i=1,2$ and the Codazzi equations at $(z_1,0,0)$ give
\beq\label{eq:CDzzi}
\begin{split}
\partial_2(\partial_1\Phi \cdot \partial_1\tn) & =  - \pl_1 \II^0_{12} + \kappa( \II^0_{11}+  \II^0_{22}),
\\
\partial_2(\partial_2\Phi \cdot \partial_1\tn) & =   - \pl_1 \II^0_{22} -\kappa \II^0_{12}.
\end{split}
\eeq

Endow now $(0,L)\times (-w/2,w/2) \times (-t/2,t/2)$ with the metric $\g$, as pulled-back by $\psi$ from $\M$ (and whose expression in coordinates is given by \eqref{eq:g_estimate}).
Since $\Psi$ is an embedding, we can pushforward this metric through $\Psi$ to $\W_{t,w}$.
Denoting this metric by $\tg$, we have
\begin{equation}\label{deftg}
\tg\circ\Psi = (D\Psi)^{-T} \g (D\Psi)^{-1} = I + (D\Psi)^{-T} ( \g-  D\Psi^TD\Psi) (D\Psi)^{-1} . 
\end{equation}
Combining this equation with \eqref{eq:g_estimate}, \eqref{eq:DPsi_euc}, and \eqref{eq:CDzzi}, we deduce the following coordinate expression of~$\tg$ in~$\W_{t,w}$:
\beq\label{eq:tg_estimate}
\tg(\Psi(z)) = 
I + (D\Psi)^{-T}\brk{ \begin{matrix} \DG(z_1)z_2^2 -2 \DC_1(z_1)z_2z_3 & -2\DC_2(z_1)z_2z_3 & 0 \smallskip\\  -2\DC_2(z_1)z_2z_3 & -2X_{22}(z_1) z_2z_3 & 0 \smallskip\\ 0 & 0 & 0 \end{matrix}} (D\Psi)^{-1}  + O(z_2^3,z_3^2),
\eeq
where $\DC$ is the Codazzi-deficit, introduced in Definition~\ref{def:deficit}, whereas
$X_{22}$ is a quantity depending on the choice of $\Psi$ (and not only on the given geometric data).

We conclude this section with some estimates that will be useful in the following.
We denote by $\Qo \in C^\infty((0,L);\SO(3))$ the map
\beq\label{eq:Q_0}
\Qo (z_1) := D \Psi(z_1,0,0) = (\pl_1 \Psi \,|\, \pl_2 \Psi \,|\, \tn)|_{(z_1,0,0)}.
\eeq
The fact that $\Qo \in \SO(3)$ follows from \eqref{eq:DPsi_euc}.
Since $\Psi$ is smooth, we have that
\beq\label{eq:Q_0_Psi}
\|D \Psi(z) - \Qo (z_1) \|_{C^0} + \|(D \Psi(z))^{-1} - \Qo ^T(z_1) \|_{C^0} \le Cw
\eeq
for some $C>0$.
Again by \eqref{eq:DPsi_euc} we also have that
\[
\pl_2^2 \Psi \cdot \tn|_{(z_1,0,0)} = \pl_2^2 \Psi \cdot \pl_3 \Psi|_{(z_1,0,0)} 
	= - \pl_2 \Psi \cdot \pl_2 \pl_3 \Psi|_{(z_1,0,0)} = - \frac{1}{2}\pl_3 \left|\pl_2 \Psi\right|^2 (z_1,0,0)
	= \II^0_{22}(z_1),
\]
and similarly,
\[
\begin{split}
\pl_2^2 \Psi \cdot \pl_2\Psi |_{(z_1,0,0)} = \pl_2^2 \Psi \cdot\pl_1\Psi|_{(z_1,0,0)} = 0.
\end{split}
\]
Thus we obtain
\beq
\label{eq:pl2pl2Psi}
\pl_2^2 \Psi (z)= \II^0_{22}(z_1) \Qo (z_1)e_3 + O(w).
\eeq
Similarly, one can show that
\[
\pl_2\tn (z)= -\II^0_{12}(z_1) \Qo (z_1)e_1-\II^0_{22}(z_1) \Qo (z_1)e_2 + O(w).
\]

\begin{remark}\label{remak:geom}
Our analysis in the narrow ribbon case relies on the existence of the map $\Psi$ satisfying estimates \eqref{eq:tg_estimate}, \eqref{eq:Q_0_Psi} and \eqref{eq:pl2pl2Psi}. 
One could consider a more general model, in which the geometry of the ribbon changes with $t$, namely, $\M_{t,w}= (0,L_t)\times (\frac{-w}{2}, \frac{w}{2}) \times (\frac{-t}{2}, \frac{t}{2})$, with smooth metrics $\g^t$ of the form \eqref{eq:g_estimate}, where the geodesic curvature $\kappa_t$, the Gaussian curvature $K^\calS_t$,
and the second fundamental form $\II_t$ depend on $t$.
Assuming $\kappa_t$, $K^\calS_t$, $\II_t$ converge strongly enough as $t\to 0$,\footnote{More precisely, that $\kappa_t$, $K^\calS_t$, $\II_t$ and $D\II_t$ converge uniformly, that $\partial_1^2\II_t(\cdot, 0)$, $\partial_1 \kappa_t$ are uniformly bounded, and that the remainder terms in \eqref{eq:g_estimate} are uniformly bounded with respect to $t$.}
the derivatives $D\Psi^t$ of the Euclidean ribbons maps $\Psi^t$ converge to some rotation map $\Qo$.
In this way, the estimates \eqref{eq:Q_0_Psi} and \eqref{eq:pl2pl2Psi} hold with an $o(1)$ remainder, and \eqref{eq:tg_estimate} holds with $t$-dependent deficits $\delta^G_t$ and $\delta^C_t$, which converge uniformly to some limits $\delta^G_0$ and $\delta^C_0$.
The analysis in the narrow ribbon limit below will remain essentially the same in this more general model.
For the sake of readability we will present the results for non-changing geometries, with the exception of \S\ref{sec:wide_from_narrow} where changing geometries is essential.
\end{remark}

%%%%%%%%%%%%%%%%%%%%%%%%%%%%%%%%%%%%%%%%%%%%%
\section{Mechanical setting and notation}
\label{sec:settings}
Let $\M_{t,w}$ be a ribbon as defined in the previous section.
Via the coordinate map $\psi$, we will henceforth identify $\M_{t,w}$ with the set $(0,L)\times (-w/2,w/2) \times (-t/2,t/2)$ endowed with the metric $\g$ as in \eqref{eq:g_estimate}.
We further assume that $t,w$ are small enough such that the map $\Psi: (0,L)\times (-w/2,w/2) \times (-t/2,t/2) \to \W_{t,w}$, introduced in the previous section, is a diffeomorphism.
We denote by $z$ the coordinates on $\M_{t,w}=(0,L)\times (-w/2,w/2) \times (-t/2,t/2)$ and by $\zeta=\Psi(z)$ the standard Euclidean coordinates on $\W_{t,w}$.

Let $P: T\M_{t,w}\to \R^{3\times3}$ be an orientation-preserving map such that $P^TP = \g$.
The elastic energy associated with $\M_{t,w}$ (and $P$) is the functional $E_{t,w}:W^{1,2}(\M_{t,w};\R^3) \to \R\cup \{+\infty\}$ defined by
\[
E_{t,w}(y) = \dashint_{\M_{t,w}} \calW(\nabla y(z) P^{-1}(z))\,\dVol_\g(z) = \dashint_{\W_{t,w}} \calW(\nabla \tilde{y}(\zeta)\tQ^{-1}(\zeta))\,\dVol_{\tg}(\zeta)
\]
for $y\in W^{1,2}(\M_{t,w};\R^3)$, where $\dVol_\g$ is the volume form of $\g$ (and similarly for $\tg$), $\dashint$ is the integral divided by the volume of the domain, $\tilde{y}  = y\circ \Psi^{-1}\in W^{1,2}(\W_{t,w};\R^3)$, $\tQ = P\circ d\Psi^{-1}:T\W_{t,w}\to \R^{3\times 3}$, and $\calW:\R^{3\times 3} \to [0,\infty]$ is the energy density, a continuous function satisfying the following standard assumptions:
\begin{enumerate}[label=(\alph*)]
\item \label{itm:frame_ind} Frame indifference: $\calW(RA) = \calW(A)$ for all $A\in \R^{3\times 3}$ and $R\in \SO(3)$; 
\item \label{itm:well} $\calW(A) = 0$ if and only if $A\in \SO(3)$;
\item \label{itm:coercive} Coercivity: There exists $c>0$ such that $\calW(A) \ge c \dist^2(A,\SO(3))$ for all $A\in \R^{3\times 3}$;
\item \label{item:regularity} Regularity: There exists a neighborhood of $\SO(3)$ in which $\calW$ is $C^2$:
	\beq
	\left|\calW(I + B) - \calQ_3(B)\right| \le \omega(|B|), \qquad \calQ_3(B) := \frac{1}{2}D_I^2\calW(B,B)\label{eq:C_2_regularity}
	\eeq
		where $\omega:[0,\infty)\to [0,\infty]$ is a function satisfying $\lim_{t\to 0} \omega(t)/t^2 = 0$. 
\end{enumerate}
Note that assumptions \ref{itm:well} and \ref{itm:coercive} imply that 
\beq\label{eq:Q_symmetric}
\calQ_3(B) = \calQ_3\brk{{\rm sym}\, B}\geq c\,|{\rm sym}\, B|^2
\eeq
for all $B\in \R^{3\times 3}$.

The energy density $\calW$ represents the archetypical elastic behavior of a single point, and the map $P^{-1}$ indicates how $\calW$ implants into the body (hence the name \emph{implant map} \cite{EKM20}).
Any two implant maps $P^1,P^2$ differ from each other by a rotation, i.e., $P^2_z = R(z)P^1_z$, where $R(z)\in \SO(3)$.
In particular, we have that $P_z = R(z) \sqrt{\g}_z$, and similarly for $\tQ$.
For simplicity, and since it does not change the general behavior of the problem considered here, we will assume that $\tQ = \sqrt{\tg}$ (in the case of an isotropic $\calW$, the energy does not change at all by this choice).

In order to study the $\Gamma$-limit as $t,w\to 0$, we rescale the domain, namely we let
\[
z_1 = x_1, \quad z_2 = wx_2, \quad z_3 = t x_3.
\] 
For $y\in W^{1,2}(\M_{t,w};\R^3)$ we define $u \in W^{1,2}(U;\R^3)$, where $U = (0,L)\times (-1/2,1/2)^2$, via $u(x) = y(z(x))$.
Similarly we define 
\[
\Psi_t(x) = \Psi(z(x)), \qquad \g_t(x) = \g(z(x)), \qquad \tg_t(x) = \tg(\Psi_t(x)), \qquad \tn_t(x_1,x_2) = \tn(x_1,wx_2).
\]
It follows from \eqref{eq:tg_estimate} that $\tg(\Psi_t(x)) = I + O(w^2)$, and, when $\DG\equiv 0$ and \eqref{eq:w_narrow_ribbon} holds, that $\tg(\Psi_t(x)) = I + O(wt)$.
After changing variables, the energy takes the form,
\[
\dashint_{U} \calW(\nabla_{t} u\, (\nabla_t \Psi_t)^{-1} \tg_t^{-1/2})\,\det \nabla_t\Psi_t\,\dVol_{\tg_t},
\]
where
\[
\nabla_t  = \left( \pl_1   \,\left|\, \frac{1}{w}\pl_2    \,\right|\, \frac{1}{t}\pl_3  
			\right).
\]
By \eqref{eq:DPsi_euc} and \eqref{eq:tg_estimate} it follows that 
\[
\det \nabla_t\Psi_t\,\dVol_{\tg_t} = (1+O_{L^\infty}(w))\,dx,
\]
and thus we replace this volume form by $dx$, as it does not affect the analysis below in any way, but makes the notation lighter.
To conclude, we consider the energy $\E_{t,w}:W^{1,2}(U;\R^3) \to \R\cup \{+\infty\}$
\beq\label{eq:energy_rescaled}
\E_{t,w}(u) =  \dashint_{U} \calW(\nabla_{t} u \, (\nabla_t \Psi_t)^{-1} \tg_t^{-1/2})\,dx.
\eeq

%%%%%%%%%%
We now define several functions that will appear as energy densities in the limiting problems.
For $x_1\in(0,L)$ and $A\in \R^{3\times 3}$ we denote
\[
\tcalQ_3(x_1,A) = \calQ_3(\Qo (x_1)A \Qo (x_1)^T),
\] 
where $\calQ_3$ is the quadratic form defined in \eqref{eq:C_2_regularity} and $\Qo $ is defined in \eqref{eq:Q_0}.
We define two $x_1$-dependent quadratic forms. 
For $x_1\in(0,L)$ and $A\in \R^{2\times 2}$ let
\beq\label{eq:tcalQ2}
\tcalQ_2(x_1,A) = \min_{B\in \R^{3\times 3}} \big\{ \tcalQ_3(x_1, B)~:~ B_{2\times 2} = A\big\},
\eeq
where $B_{2\times 2}$ is the upper $2\times 2$ minor of $B$.
For $x_1\in(0,L)$ and $a,b\in\R$ let
\[
\tcalQ_2^\circ(x_1,a,b) = \min_{A\in \R^{2\times 2}} \big\{\tcalQ_2(x_1,A) ~:~ A^T = A, \, A_{11} = a, \, A_{12} = b \big\}.
\]
Additionally, for $x_1\in(0,L)$ we define the function
\[
\tcalQ_1(x_1) = \min_{A\in \R^{2\times 2}} \BRK{\tcalQ_2(x_1,A) ~:~ A_{11} = 1}.
\]

\paragraph{Notation.}
In the following, for two positive quantities $a,b$, we write $a\ll b$ if the limit $a/b$ tends to zero (typically when $t\to 0$ or $w\to 0$, depending on the context), and $a\lesssim b$ if there exists a constant $C>0$ independent of $t,w$ such that $a\le Cb$.
For a family of functions $(f_t)_{t\in (0,t_0)}$ we write $f_t=O_{L^2}(t)$ if there exists $C>0$ such that $\|f_t\|_{L^2} \le Ct$ for every $t\in (0,t_0)$. 
We similarly write $O_{X}(\cdot)$ for bounds in a function space $X$ (typically positive or negative Sobolev spaces).
Finally, for a function $h:(0,L)\times (-1/2,1/2)\to \R^d$ we denote
\[
\bar{h}(x_1) = \int_{-1/2}^{1/2} h(x_1,x_2) \,dx_2.
\]

%%%%%%%%%%%%%%%%%%%%%%%%%%%%%%%%%%%%%%%%%%%%%
\section{Narrow ribbons}\label{sec:narrow}
In this section we study the behavior of the energy $\E_{t,w}$, introduced in \eqref{eq:energy_rescaled}, in the narrow ribbon regime $w^2 \ll t \ll w \ll 1$.
We start with a compactness result for deformations with energy of order $\e^2$, where $\e=\e(t,w)$ is an appropriate rescaling.

%%%%%%%%%
\begin{theorem}[Narrow ribbons, compactness]
\label{thm:narrow_compactness}
Let $(u_t)$ be a sequence in $W^{1,2}(U;\R^3)$ satisfying
\beq\label{eq:enb}
\E_{t,w}(u^t) \le C\e^2,
\eeq
where $\e\to0$ is a function of $t$ satisfying the following condition:
if $\DG\nequiv 0$, 
$$
w^2 \lesssim \e \ll t,
$$ 
or otherwise
$$
wt \lesssim \e \ll t \quad \text{ and } \quad w^2 \ll t.
$$
Then, there exist rotations $\Rbar_t\in \SO(3)$ and translations $c_t\in \R^3$ such that
\[
\hat{u}_t = \Rbar_t^T u_t - c_t
\]
satisfy
\beq\label{eq:u_to_Psi}
\|\hat{u}_t - \Psi_t\|_{L^2(U)} \le C\frac{\e}{t} \qquad \text{and} \qquad \|\nabla_t \hat{u}_t - \nabla_t \Psi_t\|_{L^2(U)} \le C\frac{\e}{t},
\eeq
and
\beq\label{eq:u_to_g}
\| (\nabla_t \hat{u}_t)^T (\nabla_t \hat{u}_t) - \g_t\|_{L^1(U)} \le C\e.
\eeq
In particular, $\hat{u}_t$ converges strongly in $W^{1,2}$ to $\Psi(x_1,0,0)$, and $\nabla_t \hat{u}_t$ converges strongly in $L^2$ to $\Qo $.
Furthermore, there exist functions $\dL,\dM\in L^2((0,L))$ and $\dN\in L^2((0,L)\times (-1/2,1/2))$, and a skew-symmetric matrix field $A\in W^{1,2}((0,L);\R^{3\times 3})$ satisfying 
\[
\Qo ^T A' \Qo  = 
	\brk{\begin{matrix}
		0 & 0 & -\dL \\
		0 & 0 & -\dM \\
		\dL & \dM & 0
		\end{matrix}} ,
\]
such that, up to subsequences, the following convergences hold:
\begin{align}
\label{eq:convergence_notion1}
\frac{t}{\e}(\nabla_t \hat{u}^t - \nabla_t\Psi_t) &\to A\Qo  \quad &\text{ in } &\quad L^2(U;\R^{3\times 3}) \\
\label{eq:convergence_notion2}
\frac{t}{w^2\e}\pl_2 \brk{\int_{-1/2}^{1/2} (\pl_2 \hat{u}^t - \pl_2 \Psi_t) \,dx_3} \cdot \tn_t &\weakly \dN \quad &\text{ in }& \quad W^{-1,2}((0,L)\times (-1/2,1/2)).
\end{align}
\end{theorem}

The next two theorems identify the $\Gamma$-limit of $\E_{t,w}$ at the correct scaling for Gauss-incompatible ribbons and for Codazzi-incompatible ribbons.
%%%%%%%%%
\begin{theorem}[Narrow ribbons, $\Gamma$-convergence -- Gauss-incompatibility]
\label{thm:narrow_Gauss_Gamma}
Assume \eqref{eq:w_ribbon}--\eqref{eq:w_narrow_ribbon} hold. 
Then the functional $\frac{1}{w^4} \E_{t,w}$ $\Gamma$-converges, with respect to the convergence notion defined in Theorem~\ref{thm:narrow_compactness} for $\e=w^2$, to
\[
\E_0^{G}(\dL,\dM,\bdN) = \frac{1}{12}\dashint_0^L \tcalQ_2\brk{x_1,\brk{\begin{matrix}
									\dL & \dM  \\
									\dM & \bdN 
									\end{matrix}}}\,dx_1
		+\frac{1}{720} \dashint_0^L  \tcalQ_1\brk{x_1}\DG(x_1)^2\,dx_1,
\]
where $\dL$, $\dM$, and $\dN$ are the functions defined in Theorem~\ref{thm:narrow_compactness}
and $\DG$ is the Gauss-deficit \eqref{eq:Gauss_def}.
\end{theorem}

%%%%%%%%%%
\begin{theorem}[Narrow ribbons, $\Gamma$-convergence -- Codazzi-incompatibility]
\label{thm:narrow_Codazzi_Gamma}
Assume  \eqref{eq:w_ribbon}--\eqref{eq:w_narrow_ribbon} hold and $\DG\equiv 0$. 
Then the functional $\frac{1}{t^2w^2} \E_{t,w}$ $\Gamma$-converges, with respect to the convergence notion defined in Theorem~\ref{thm:narrow_compactness} for $\e=tw$, to
\[
\E_0^{C}(\dL,\dM,\bdN) = \frac{1}{12}\dashint_0^L \tcalQ_2\brk{x_1,\brk{\begin{matrix}
									\dL & \dM  \\
									\dM & \bdN 
									\end{matrix}}}\,dx_1		
		+\frac{1}{144}\dashint_0^L \tcalQ_2^\circ\brk{x_1,\DC_1(x_1), \DC_2(x_1)}\,dx_1,
\]
where $\dL$, $\dM$, and $\dN$ are the functions defined in Theorem~\ref{thm:narrow_compactness}
and $\DC$ is the Codazzi-deficit \eqref{eq:Codazzi_def}.

\end{theorem}

The proofs of these theorems are the content of \S\ref{sec:narrow_compact}--\ref{sec:narrow_recovery}.

%%%%%%%%%
As both the energies $\E_0^{G}$ and $\E_0^{C}$ are minimized by $\dL=\dM=\bdN=0$, we immediately obtain the following corollary from the compactness and $\Gamma$-convergence results.
\begin{corollary}[Energy scalings for narrow ribbons]
\label{cor:scaling_narrow}
Assume  \eqref{eq:w_ribbon}--\eqref{eq:w_narrow_ribbon} hold. Then
\[
\lim_{t\to 0}\Big( \inf \frac{1}{w^4}\E_{t,w} \Big)\in (0,\infty) \qquad \text{if and only if} \,\, \DG\nequiv 0,
\]
that is, if and only if the ribbon is Gauss-incompatible.
If the ribbon is Gauss-compatible, we have
\[
\lim_{t\to 0} \Big( \inf \frac{1}{t^2w^2}\E_{t,w} \Big) \in (0,\infty) \qquad \text{if and only if} \,\, \DG\equiv 0, \, \DC\nequiv 0,
\]
that is, if and only if the ribbon is Codazzi-incompatible.
\end{corollary}
In particular, Corollary~\ref{cor:scaling_narrow} immediately implies Theorem~\ref{thm:narrow_informal}(1).

From the above results it follows that a sequence of approximate minimizers $u_t$ of $\E_{t,w}$ satisfies the conditions of Theorem~\ref{thm:narrow_compactness} with $\e = w^2$ (for a Gauss-incompatible ribbon) or $\e =t w$ (for a Codazzi-incompatible ribbon).
The convergence \eqref{eq:u_to_g}, together with \eqref{eq:DPsi_euc}, imply that the actual metric of the ribbon immersed by $u_t$ can be expanded as
\[
\begin{split}
(\nabla_t u_t)^T (\nabla_t u_t) 
	&= \brk{\begin{matrix}
			(1-\kappa(x_1)w x_2)^2 + O_{L^1}(w^2) & 0 & 0\\
			0 & 1 & 0 \\
			0 & 0 & 1
			\end{matrix}}
		-2tx_3 \brk{\begin{matrix}
			\II^0(x_1) & 0 \\
			0 & 0
			\end{matrix}}
		+O_{L^1}(tw, \e).
\end{split}
\]
For Codazzi-incompatible ribbons, since $\e \ll w^2$, the $O_{L^1}(w^2)$ term can be written as $w^2 K^\calS(x_1,0)x_2^2 + O_{L^1}(\e)$.
This implies Theorem~\ref{thm:narrow_informal}(2).

Finally, we have that the second fundamental form of the mid-surface of $u_t(U)$ along the midline $u_t(\ell)$ tends to the reference second form $\II^0$ along the midline.
The limiting variables $\dL,\dM,\bdN$ measure the leading order deviation of this second fundamental form from $\II^0$, at the scale that contributes to the limiting energy.
This will be seen quite explicitly in the construction of the recovery sequence below.

%%%%%%%%%%%%%%%%%%%%%%
\subsection{Compactness: Proof of Theorem~\ref{thm:narrow_compactness}}
\label{sec:narrow_compact}

We start by proving the existence of auxiliary rotation fields $R^t$ along the mid-surface, such that $\hat{u}^t\circ \Psi_t^{-1}$ is close to $R^t$.

\begin{lemma}\label{lem:rigidity}
Let $(u_t)$ and $\e$ satisfy the conditions of Theorem~\ref{thm:narrow_compactness}.
Then, there exists a sequence of rotations $R^t\in C^\infty([0,L]\times [-1/2,1/2] ; \SO(3))$ such that
\begin{enumerate}
\item $\|\nabla_t\hat{u}^t - R^t \nabla_t \Psi_t\|_{L^2(U;\R^{3\times3})} \le C \e$,
\item $\|R^t -I\|_{L^2((0,L)\times (-1/2,1/2);\R^{3\times3})} \le C \frac{\e}{t}$,
\item $\|\pl_1 R^t \|_{L^2((0,L)\times (-1/2,1/2);\R^{3\times3})} \le C \frac{\e}{t}$,
\item $\|\pl_2 R^t \|_{L^2((0,L)\times (-1/2,1/2);\R^{3\times3})} \le C \frac{w\e}{t}$,
\end{enumerate}
where $\hat{u}^t$ is a suitable rotation of $u^t$.
\end{lemma}

\begin{proof}
For any $Q\in \SO(3)$, it follows from \eqref{eq:tg_estimate} that
\[
|\nabla_{t} u^t\, (\nabla_t \Psi_t)^{-1}  - Q| \le |\nabla_{t} u^t \, (\nabla_t \Psi_t)^{-1} - Q \tg_t^{1/2}| + |\tg_t^{1/2} - I| \le C\brk{  |\nabla_{t} u \, (\nabla_t \Psi_t)^{-1}  \tg_t^{-1/2} - Q| + w^2}
\]
where $C$ depends only on $\g$. If $\DG\equiv 0$, then the error term is $wt$ using the assumption $w^2\ll t$ (otherwise we would get an error of order $w^3$).
Thus we have
\[
\int_U \dist^2(\nabla_{t} u^t \, (\nabla_t \Psi_t)^{-1} , \SO(3))\,dx \le C \brk{\int_U \dist^2(\nabla_{t} u^t \, (\nabla_t \Psi_t)^{-1}  \tg_t^{-1/2}, \SO(3))\,dx + w^4},
\] 
or with an error term $w^2t^2$ if $\DG\equiv 0$.
By the coercivity assumption (c) on $\calW$ the bound \eqref{eq:enb}, together with the conditions on $\e$, imply that
\[
\int_U \dist^2(\nabla_{t} u^t \, (\nabla_t \Psi_t)^{-1} , \SO(3))\,dx \le C\e^2.
\]
Changing variables to $\tilde{u}^t = u^t\circ \Psi_t^{-1}$, and using the fact that $\Psi$ is a bi-Lipschitz map, we obtain that
\[
\dashint_{\W_{t,w}}  \dist^2(\nabla \tilde{u}^t , \SO(3))\,dx \le C\e^2.
\]
One can now adapt the proof of \cite[Theorem~3.2]{FMP13} to the sequence $\tilde{u}^t$ on the domain $\W_{t,w}$. 
Indeed, one can apply the Friesecke--James--M\"uller rigidity estimate \cite{FJM02b} to $\tilde{u}^t$ on the set $\Psi(Q_t(i,j))$, where
$$
Q_t(i,j):=((i-1)\lambda_{1,t}, i\lambda_{1,t})\times (-w/2+(j-1)\lambda_{2,t}, -w/2+j\lambda_{2,t})\times (-t/2,t/2)
$$
for $i=1,\dots,L/\lambda_{1,t}$ and $j=1,\dots, w/\lambda_{2,t}$ and $\lambda_{1,t}:=L/\lfloor L/t \rfloor$, $\lambda_{2,t}:=w/\lfloor{w/t\rfloor}$. Here $\lfloor x \rfloor$ denotes the integer part of $x$.
Since the sets $\Psi(Q_t(i,j))$ are bi-Lipschitz images of the cube $(0,t)^3$ with uniform Lipschitz constants, the same value of the rigidity constant serves for each of these sets. This provides a sequence of piecewise constant rotations $\hat R^t:(0,L)\times(-1/2,1/2)\to SO(3)$ such that
\[
\dashint_{\W_{t,w}}  |\nabla \tilde{u}^t -\hat R^t|\,dx \le C\e^2
\]
and by changing variables back,
\[
\int_U |\nabla_{t} u^t \, (\nabla_t \Psi_t)^{-1}- \hat R^t|^2\,dx \le C\e^2.
\]
Arguing as in the proof of \cite[Lemma~3.1]{FMP13}, one can then deduce the existence of $R^t$ with all its stated properties.
\end{proof}
\medskip

Note that if $\DG\equiv 0$, the conclusion of Lemma~\ref{lem:rigidity} also holds when assuming $wt \lesssim \e \ll t$ and $w^3 \lesssim \e$ (without the stronger requirement $w^2\ll t$).

We now show that $R^t$ can be well-approximated by rotation fields along the midline. This will be important when proving the lower-semicontinuity estimates in the next section.
%%%%%%%%%
\begin{lemma}\label{lem:Rt0}
Let $R^t$ be as in Lemma~\ref{lem:rigidity}.
Then, for each $t>0$ there exists $x_2^t \in [-1/2,1/2]$ such that the function $R_0^t :(0,L)\to \SO(3)$ defined by
\[
R_0^t(x_1) := R^t(x_1,x_2^t)
\]
satisfies
\begin{enumerate}
\item $\|\pl_1 R_0^t \|_{L^2((0,L)\times (-1/2,1/2);\R^{3\times3})} \le C \frac{\e}{t}$,
\item $\|R^t - R_0^t\|_{L^2((0,L)\times (-1/2,1/2);\R^{3\times3})} \le C \frac{w\e}{t}$, 
\item $\|\sym( (R_0^t)^T R^t - I)\|_{L^1((0,L)\times (-1/2,1/2);\R^{3\times3})} \le C \frac{w^2\e^2}{t^2}$.
\end{enumerate}
\end{lemma}

\begin{proof}
By (3) in Lemma~\ref{lem:rigidity} we have that
\[
\int_{-1/2}^{1/2}   \int_0^L |\pl_1 R^t(x_1,x_2)|^2 \,dx_1 \,dx_2 \le C\frac{\e^2}{t^2}.
\]
In particular, for every $t$ there exists a set $S\subset(-1/2,1/2)$ of positive measure such that for every $x_2\in S$
\[
\int_0^L |\pl_1 R^t(x_1,x_2)|^2 \,dx_1 \le C \frac{\e^2}{t^2}
\]
with the same constant $C>0$.
Choose $x_2^t$ arbitrarily in $S$. This proves the first inequality.

We note that 
\[
R^t(x_1,x_2) - R_0^t(x_1) = \int_{x_2^t}^{x_2} \pl_2 R^t(x_1,s)\,ds,
\]
hence by (4) in Lemma~\ref{lem:rigidity},
\[
\int_0^L \int_{-1/2}^{1/2} |R^t(x_1,x_2) - R_0^t(x_1) |^2\,dx_2\,dx_1 \le \int_0^L \int_{-1/2}^{1/2} |\pl_2 R^t(x_1,x_2)|^2\,dx_2\,dx_1 \le C\frac{w^2\e^2}{t^2},
\]
which proves the second inequality.
Finally, using the equality
\beq\label{eq:R-I}
(R-I)^T(R-I) = -2\sym(R-I) \qquad \text{ for every } R\in \SO(3),
\eeq
we have that
\[
\|\sym( (R_0^t)^T R^t - I)\|_{L^1}= \frac{1}{2} \|( (R_0^t)^T R^t - I)^T( (R_0^t)^T R^t - I)\|_{L^1} = \frac{1}{2} \|(R_0^t)^T R^t - I\|_{L^2}^2 =  \frac{1}{2} \| R^t - R_0^t\|_{L^2}^2 \lesssim \frac{w^2\e^2}{t^2},
\]
which concludes the proof.
\end{proof}
\medskip

We now relate the rotation fields $R^t$ to the limiting fields $\dL,\dM,\dN$.
%%%%%%%%%
\begin{lemma}\label{lem:compactness1}
Let $R^t$ be as in Lemma~\ref{lem:rigidity}.
Then there exists skew-symmetric matrix fields $A\in W^{1,2}((0,L);\R^{3\times 3})$ and $B\in L^2((0,L)\times (-1/2,1/2) ; \R^{3\times 3})$ such that, up to subsequences,
\begin{enumerate}
\item $\frac{t}{\e}(R^t - I) \weakly A$ in $W^{1,2}((0,L)\times (-1/2,1/2) ;\R^{3\times 3})$,
\item $\frac{t^2}{\e^2} \sym(R^t - I) \to \frac{A^2}{2}$ in $L^2((0,L)\times (-1/2,1/2) ;\R^{3\times 3})$,
\item $\frac{t}{w\e} \pl_2 R^t \weakly B$ in $L^2((0,L)\times (-1/2,1/2) ; \R^{3\times 3})$,
\end{enumerate}
where
\[
A' \Qo  e_2 = B \Qo  e_1
\]
and $\Qo $ is defined in \eqref{eq:Q_0}.
In particular,
\[
\Qo ^T A' \Qo (x_1) = 
	\brk{\begin{matrix}
		0 & 0 & -\dL (x_1) \\
		0 & 0 & -\dM(x_1) \\
		\dL (x_1) & \dM (x_1) & 0
		\end{matrix}}
\qquad
\Qo ^T B \Qo (x_1,x_2) = 
	\brk{\begin{matrix}
		0 & 0 & -\dM(x_1) \\
		0 & 0 & -\dN(x_1,x_2) \\
		\dM(x_1) & \dN (x_1,x_2) & 0
		\end{matrix}}
\]
for some functions $\dL,\dM\in L^2((0,L))$ and $\dN\in L^2((0,L)\times (-1/2,1/2))$.
In addition, \eqref{eq:convergence_notion1} holds for $\hat{u}^t$ given by Lemma~\ref{lem:rigidity}.
\end{lemma}

\begin{proof}
From (2--4) in Lemma~\ref{lem:rigidity} it follows that $\frac{t}{\e}(R^t - I)$ is bounded in $W^{1,2}((0,L)\times (-1/2,1/2);\R^{3\times 3})$, hence weakly converging to some $A$, which must be skew-symmetric.
Property (4) in Lemma~\ref{lem:rigidity} implies that $A$ depends only on $x_1$, thus proving (1); it also implies (3).
Note that (1) implies that $\frac{t}{\e}(R^t - I)$ strongly converges in $L^p$ for any $p<\infty$.
This, together with \eqref{eq:R-I}, implies (2).
From (1) in Lemma~\ref{lem:rigidity} we have that
\[
\frac{t}{\e}\brk{\nabla_t \hat{u}^t - \nabla_t \Psi_t }=\frac{t}{\e}(R^t - I) \nabla_t \Psi_t + O_{L^2}(t),
\]
hence part (1) of this lemma, together with \eqref{eq:Q_0_Psi}, implies \eqref{eq:convergence_notion1}.

All that is left to prove is the relation between $A$ and $B$.
Let $\vp\in C_c^\infty(U;\R^3)$. 
We then have
\[
\begin{split}
\inner{\frac{t}{w\e} (\hat{u}^t - \Psi_t)\,,\,\pl_1\pl_2 \vp} 
	&= -\inner{\frac{t}{w\e}\pl_2 (\hat{u}^t - \Psi_t)\,,\,\pl_1 \vp} = -\inner{\frac{t}{\e}\nabla_t (\hat{u}^t - \Psi_t)e_2\,,\,\pl_1 \vp} \\
	&\to -\inner{A\Qo  e_2\,,\, \pl_1 \vp} = \inner{(A'\Qo  + A\Qo ') e_2\,,\, \vp}.
\end{split}
\]
On the other hand,
\[
\begin{split}
\inner{\frac{t}{w\e} (\hat{u}^t - \Psi_t)\,,\,\pl_1\pl_2 \vp} 
	&= -\inner{\frac{t}{w\e}\pl_1 (\hat{u}^t - \Psi_t)\,,\,\pl_2 \vp} = -\inner{\frac{t}{w\e}\nabla_t (\hat{u}^t - \Psi_t)e_1\,,\,\pl_2 \vp} \\
	&= -\inner{\frac{t}{w\e}(R_t - I)\nabla_t \Psi_t\,e_1\,,\,\pl_2 \vp} + \inner{\frac{t}{w\e}( R_t\nabla_t \Psi_t - \nabla_t \hat{u}^t) e_1\,,\,\pl_2 \vp} \\
	&= -\inner{\frac{t}{w\e}(R_t - I)\nabla_t \Psi_t\,e_1\,,\,\pl_2 \vp} + O\brk{\frac{t}{w}} \\
	&= \inner{\frac{t}{w\e}\pl_2 R_t\, \nabla_t \Psi_t \,e_1\,,\, \vp}+ \inner{\frac{t}{\e}(R_t - I)\frac{1}{w}\pl_2\nabla_t \Psi_t\,e_1\,,\, \vp} + O\brk{\frac{t}{w}}\\
	&= \inner{\frac{t}{w\e}\pl_2 R_t\, \nabla_t \Psi_t \,e_1\,,\, \vp}+ \inner{\frac{t}{\e}(R_t - I)\pl_1\nabla_t \Psi_t\,e_2\,,\, \vp} + O\brk{\frac{t}{w}} \\
	&\to \inner{B \Qo  e_1\,,\,\vp} + \inner{A\Qo ' e_2\,,\, \vp},
\end{split}
\]
where in the last line we used part (3) from this lemma and \eqref{eq:Q_0_Psi}.
This shows that $\Qo ^T A' \Qo  e_2 = \Qo ^T B \Qo  e_1$. 
Since $\Qo ^T A' \Qo $ and $\Qo ^T B \Qo $ are skew-symmetric, this completes the proof.
\end{proof}
\medskip

%%%%%%%%%

We can now prove Theorem~\ref{thm:narrow_compactness}.

\begin{proof1}{of Theorem~\ref{thm:narrow_compactness}:}
Let $R^t$ and $\hat{u}^t$ be as in Lemma~\ref{lem:rigidity}. The derivative estimate in \eqref{eq:u_to_Psi} follows from \eqref{eq:convergence_notion1}, which was shown
as part of Lemma~\ref{lem:compactness1}. 
The lefthand estimate in \eqref{eq:u_to_Psi} follows from the derivative estimate by Poincar\'e inequality, after translating $\hat{u}_t$ appropriately.

We now prove the metric estimate \eqref{eq:u_to_g}.
We have that
\[
\begin{split}
(\nabla_t \hat{u}_t)^T (\nabla_t \hat{u}_t) &= (R^t \nabla_t \hat{u}_t)^T (R^t \nabla_t \hat{u}_t) \\
	&= (\nabla_t \Psi^t)^T (\nabla_t \Psi^t) + O_{L^1}(\e) \\
	&= \g_t + O_{L^1}(\e), 
\end{split}
\]
where the transition to the second line follows from Lemma~\ref{lem:rigidity}, the $L^2$ boundedness of $\nabla_t \hat{u}_t$ and the uniform boundedness of $\nabla_t \Psi^t$, and the transition to the third line follows from \eqref{eq:DPsi_g} and the relation between $w,t$, and $\e$. 

We are left to prove \eqref{eq:convergence_notion2}. 
Note that for $\vp\in W^{1,2}_0((0,L)\times (-1/2,1/2))$ we have
\[
\begin{split}
&\inner{\frac{t}{w^2\e}\pl_2 \brk{\int_{-1/2}^{1/2} (\pl_2 \hat{u}^t - \pl_2 \Psi_t) \,dx_3} \cdot \tn_t\,, \,\vp} =
	- \inner{\frac{t}{w^2\e} \int_{-1/2}^{1/2} (\pl_2 \hat{u}^t - \pl_2 \Psi_t) \,dx_3 \,,\,\pl_2 (\vp\tn_t  )} \\
	&\qquad = - \inner{\frac{t}{w\e} (R^t-I) \nabla_t \Psi_t(x_1,x_2,0)\, e_2 \,,\,\pl_2 (\vp \tn_t )} + O(t/w) \\
	&\qquad = \inner{\frac{t}{w\e} (\pl_2 R^t \nabla_t \Psi_t(x_1,x_2,0))\, e_2\cdot \tn_t \,,\, \vp} 
		+ \inner{ \brk{\frac{t}{\e} (R^t-I)\frac{1}{w}\pl_2 \nabla_t \Psi_t(x_1,x_2,0)\, e_2}\cdot \tn_t \,,\, \vp} + O(t/w)\\
	&\qquad = \inner{\frac{t}{w\e}  (\Qo ^T\pl_2 R^t \Qo )_{32}\,,\, \vp} + \inner{\II^0_{22}\brk{\frac{t}{\e} (R^t-I)\Qo  \,e_3 }\cdot \tn_t \,,\, \vp} + O(t/w,w)\\
	&\qquad = \inner{\frac{t}{w\e}  (\Qo ^T\pl_2 R^t \Qo )_{32}\,,\, \vp} + \inner{\II^0_{22}\frac{t}{\e} (\Qo ^T(R^t-I)\Qo )_{33} \,,\, \vp} + O(t/w,w),\\
\end{split}
\]
where in the second equality we used (1) in Lemma~\ref{lem:rigidity} and in the fourth equality (2) and (4) in Lemma~\ref{lem:rigidity}, \eqref{eq:Q_0_Psi}, and \eqref{eq:pl2pl2Psi}.
Now, taking the limit we obtain that the first addend in the righthand side tends to $\inner{\dN,\vp}$ due to (3) in Lemma~\ref{lem:compactness1}, and the second tends to zero due to (1) in Lemma~\ref{lem:compactness1}.
Thus
\[
\inner{\frac{t}{w^2\e}\pl_2 \brk{\int_{-1/2}^{1/2} (\pl_2 \hat{u}^t - \pl_2 \Psi_t) \,dx_3} \cdot \tn_t\,, \,\vp} \to \inner{\dN ,\vp},
\]
which completes the proof.
\end{proof1}

%%%%%%%%%%%%%%%%%%%%%%
\subsection{Lower semicontinuity}
\label{sec:lsc}
In this section we prove the lower semicontinuity estimates for the $\Gamma$-convergence results in Theorems~\ref{thm:narrow_Gauss_Gamma}--\ref{thm:narrow_Codazzi_Gamma}.
Throughout this section, $R^t$ and $R^t_0$ are as in Lemma~\ref{lem:rigidity}--\ref{lem:Rt0}.

Let $(u^t)$ be a sequence in $W^{1,2}(U;\R^3)$ satisfying
\[
\E_{t,w}(u^t) \le C\e^2,
\]
where $\e$ is as in Theorem~\ref{thm:narrow_compactness}. Let $(\hat u^t)$ be the sequence provided by this theorem and let $\dL,\dM,\dN$ be the corresponding limiting fields.

%%%%%%%%
We now rescale and rewrite the energy as
\beq\label{eq:energyrew}
\begin{split}
\frac{1}{\e^2}\E_{t,w}(u^t) &= \frac{1}{\e^2}\E_{t,w}(\hat{u}^t) = \frac{1}{\e^2}\dashint_{U} \calW(\nabla_{t} \hat{u}^t \, (\nabla_t \Psi_t)^{-1}  \tg_t^{-1/2})\,dx \\
	& = \frac{1}{\e^2}\dashint_{U} \calW((R^t)^T\nabla_{t} \hat{u}^t \, (\nabla_t \Psi_t)^{-1}  \tg_t^{-1/2})\,dx \\
	& = \frac{1}{\e^2}\dashint_{U} \calW\brk{I+ \e \brk{G^t \tg_t^{-1/2} + \frac{\tg_t^{-1/2} - I}{\e}}}\,dx,
\end{split}
\eeq
where 
\[
G^t = \frac{(R^t)^T\nabla_{t} \hat{u}^t \, (\nabla_t \Psi_t)^{-1} -I}{\e}.
\]
Note that from \eqref{eq:tg_estimate} and \eqref{eq:Q_0_Psi} it follows that
\beq\label{eq:metlim}
\Qo ^T\frac{\tg_t^{-1/2} - I}{\e} \Qo  = - \frac{1}{2}\frac{w^2}{\e} \DG(x_1)x_2^2 e_1\otimes e_1 + O_{L^\infty}(w^3/\e, wt/\e),
\eeq
and if $\DG\equiv 0$,
\beq\label{eq:metlim2}
\Qo ^T\frac{\tg_t^{-1/2} - I}{\e} \Qo  = \frac{w t}{\e} \brk{ \begin{matrix} \DC_1(x_1) & \DC_2(x_1) & 0 \smallskip \\  \DC_2(x_1) & X_{22}(x_1) & 0 
\smallskip \\ 0 & 0 & 0 \end{matrix}}x_2x_3 + O_{L^\infty}(w^3/\e, t^2/\e).
\eeq
Further note that
$\tg_t^{-1/2} \to I$ uniformly.

%%%%%%%%%
\begin{lemma}\label{lemma:G}
Modulo a subsequence, we have that $G^t \weakly G$ in $L^2(U;\R^{3\times 3})$,
where
\[
\Qo ^TG \Qo (x) = 
	-x_3\brk{\begin{matrix}
		\dL & \dM & 0 \\
		\dM & \dN & 0 \\
		0 & 0 & 0
		\end{matrix}}
	+ \brk{\begin{matrix}
		f_0(x_1) + f_1(x_1) x_2 & \olda_{12} & 0 \\
		\olda_{21} & \olda_{22} & 0 \\
		\olda_{31} & \olda_{32} & 0
		\end{matrix}} +\oldg\otimes e_3
\]
with $f_0,f_1 \in L^2(0,L)$, $\olda_{ij} \in L^2((0,L)\times (-1/2,1/2))$, and $\oldg\in L^2(U;\R^3)$.
\end{lemma}

We note that the affine dependence on $x_2$ of the in-plane strain is a feature of ribbons that has not been observed before. 
In a sense, it encodes the fact that, along the midline, the (limiting) ribbon induced by $u_t$ satisfies the Gauss equation (a Gauss deficit would have appeared as a quadratic term).

\begin{proof}
The fact that $G^t$ is $L^2$-bounded follows from (1) in Lemma~\ref{lem:rigidity}.
By Lemma~\ref{lem:compactness1} it is enough to prove that, in a weak sense, the following holds:
\beq\label{eq:G_derivatives}
\begin{split}
\pl_3 (\Qo ^T G \Qo ) e_1 &= \Qo ^T A' \Qo  e_3, \\
\pl_3 (\Qo ^T G \Qo ) e_2 &= \Qo ^T B \Qo  e_3, \\
\pl_2^2 (\Qo ^T G \Qo )_{11} &= 0.
\end{split}
\eeq
Since $R^t$ is smooth and independent of $x_3$, we have the following equalities (in a weak sense):
\[
\begin{split}
\pl_3(G^t \nabla_t \Psi_t) e_1 &= \frac{1}{\e} \pl_3((R^t)^T \nabla_t \hat{u}^t - \nabla_t \Psi_t)e_1 \\
	&= \frac{t}{\e} (R^t)^T \pl_1 \brk{\frac{1}{t} \pl_3 \hat{u}^t} - \frac{t}{\e} \pl_1 \brk{\frac{1}{t}\pl_3 \Psi_t}\\
	&= \frac{t}{\e} (R^t)^T \pl_1 (\nabla_t \hat{u}^t -  \nabla_t \Psi_t) e_3 + \frac{t}{\e} ((R^t)^T-I) \pl_1 \nabla_t \Psi_t e_3. \\
\end{split}
\]
Since  $\pl_1 \nabla_t \Psi_t e_3 = \pl_1 \tn_t \to \Qo 'e_3$ uniformly, $\frac{t}{\e}(R^t - I) \weakly A$ in $L^2$ and $\frac{t}{\e} \pl_1 (\nabla_t \hat{u}^t -  \nabla_t \Psi_t) \weakly (A\Qo )'$ in $W^{-1,2}$, we can go to the limit and deduce that
\[
\pl_3(G^t \nabla_t \Psi_t) e_1 \weakly A'\Qo  e_3 + A \Qo ' e_3 + A^T \Qo ' e_3 = A'\Qo  e_3.
\]
On the other hand, since $G^t\weakly G$ in $L^2$ and $\nabla_t \Psi_t \to \Qo $ uniformly, we obtain the first equality in~\eqref{eq:G_derivatives}.

Similarly, for the second equality in \eqref{eq:G_derivatives} we have
\[
\begin{split}
\pl_3(G^t \nabla_t \Psi_t) e_2 
	&= \frac{1}{\e} (R^t)^T \pl_3 (\nabla_t \hat{u}^t -  \nabla_t \Psi_t) e_2 + \frac{1}{\e} ((R^t)^T-I) \pl_3 \nabla_t \Psi_t e_2. \\
	&= \frac{t}{w\e} (R^t)^T \pl_2 (\nabla_t \hat{u}^t -  \nabla_t \Psi_t) e_3 + \frac{t}{\e} ((R^t)^T-I) \frac{1}{t}\pl_3 \nabla_t \Psi_t e_2. \\
	&= \frac{t}{w\e} (R^t)^T \pl_2 ((R^t-I)\nabla_t \Psi_t) e_3 + \frac{t}{\e} ((R^t)^T-I) \frac{1}{t}\pl_3 \nabla_t \Psi_t e_2 + O_{W^{-1,2}}(t/w). \\
	&= \frac{t}{w\e} (R^t)^T (\pl_2 R^t)\tn_t -  \frac{t}{\e} ((R^t)^T-I)\frac{1}{w} \pl_2 \nabla_t \Psi_t e_3 +  \frac{t}{\e} ((R^t)^T-I) \frac{1}{t}\pl_3 \nabla_t \Psi_t e_2  + O_{W^{-1,2}}(t/w). \\
	&= \frac{t}{w\e} (R^t)^T (\pl_2 R^t)\tn_t + O_{W^{-1,2}}(t/w), \\
\end{split}
\]
and this tends to $B\Qo e_3$ since $R^t\to I$ strongly in $L^p$ for any $p<\infty$, $\frac{t}{w\e}\pl_2 R^t\weakly B$ in $L^2$, and $\tn_t \to \Qo e_3$ uniformly.

We now prove the third equality in \eqref{eq:G_derivatives} (here the assumption $w^2 \ll t$ is crucial).
To this end, let $R^t_0$ be as in Lemma~\ref{lem:Rt0},  and consider
\[
\begin{split}
((\nabla_t\Psi_t)^T (R^t_0)^T R^t G^t \nabla_t \Psi_t)_{11} 
	&= \nabla_t \Psi_t e_1 \cdot \brk{\frac{(R^t_0)^T\nabla_t\hat{u}^t - \nabla_t \Psi_t}{\e} - \frac{1}{\e}\sym((R_0^t)^T R^t - I)\nabla_t \Psi_t}e_1 \\
	&= \nabla_t \Psi_t e_1 \cdot \brk{\frac{(R^t_0)^T\nabla_t\hat{u}^t - \nabla_t \Psi_t}{\e} }e_1 +O_{L^1}(w^2\e/t^2), \\
\end{split} 
\]
where we used (3) in Lemma~\ref{lem:Rt0} in the transition to the second line.
Thus, we have
\[
\begin{split}
&\pl_2((\nabla_t\Psi_t)^T (R^t_0)^T R^t G^t \nabla_t \Psi_t)_{11} 
	=\pl_2 \brk{ \nabla_t \Psi_t e_1 \cdot \brk{\frac{(R^t_0)^T\nabla_t\hat{u}^t - \nabla_t \Psi_t}{\e} }e_1} +O_{W^{-1,1}}(w^2\e/t^2) \\
	&\qquad=(\pl_2  \nabla_t \Psi_t) e_1  \cdot \Biggl( \frac{(R^t_0-R^t)^T}{\e} \nabla_t\hat{u}^t + \frac{(R^t)^T\nabla_t\hat{u}^t - \nabla_t \Psi_t}{\e} \Biggl)\,e_1 \\
		&\qquad\qquad +\nabla_t \Psi_t e_1 \cdot \frac{1}{\e}\brk{(R^t_0)^T\pl_2 \nabla_t\hat{u}^t e_1-\pl_2 \nabla_t \Psi_t e_1}+ O_{W^{-1,1}}(w^2\e/t^2).
\end{split} 
\]
Combining (2) in Lemma~\ref{lem:Rt0} and (1) in Lemma~\ref{lem:rigidity} with the fact that $\pl_2  \nabla_t \Psi_t$ is of order $w$ in $L^\infty$, we deduce that
\[
\begin{split}
&\pl_2((\nabla_t\Psi_t)^T (R^t_0)^T R^t G^t \nabla_t \Psi_t)_{11} \\
&\qquad
	= \nabla_t \Psi_t e_1 \cdot \frac{1}{\e}\brk{(R^t_0)^T\pl_2 \nabla_t\hat{u}^t e_1-\pl_2 \nabla_t \Psi_t e_1}+ O_{L^1}(w,w^2/t)+ O_{W^{-1,1}}(w^2\e/t^2)
	\\
&\qquad	 =\nabla_t \Psi_t e_1 \cdot \frac{w}{\e}\brk{(R^t_0)^T\pl_1 \nabla_t\hat{u}^t e_2-\pl_1 \nabla_t \Psi_t e_2}+ O_{L^1}(w,w^2/t)+ O_{W^{-1,1}}(w^2\e/t^2)
	 \\
	&\qquad= \nabla_t \Psi_t e_1 \cdot \frac{w}{\e}\biggl((R^t_0)^T\pl_1 (\nabla_t\hat{u}^t - R^t \nabla_t \Psi_t) e_2 +(R^t_0)^T(\pl_1 R^t) \nabla_t \Psi_t e_2 + ((R^t_0)^TR^t-I)\pl_1\nabla_t \Psi_t e_2\biggl)\\
		&\qquad \qquad+ O_{L^1}(w,w^2/t)+ O_{W^{-1,1}}(w^2\e/t^2).
\end{split} 
\]
Using again (2) in Lemma~\ref{lem:Rt0} and (1) in Lemma~\ref{lem:rigidity}, we obtain
\[
\begin{split}
\pl_2((\nabla_t\Psi_t)^T (R^t_0)^T R^t G^t \nabla_t \Psi_t)_{11} 
& = \frac{w}{\e} \nabla_t \Psi_t e_1 \cdot (R^t_0)^T(\pl_1 R^t) \nabla_t \Psi_t e_2 
\\
&\qquad +O_{W^{-1,2}}(w)+ O_{L^1}(w,w^2/t)+O_{W^{-1,1}}(w^2\e/t^2).
\end{split}
\]
Differentiating again, we have
\[
\begin{split}
&\pl_2^2((\nabla_t\Psi_t)^T (R^t_0)^T R^t G^t \nabla_t \Psi_t)_{11} \\
	&\qquad = \frac{w}{\e}\underbrace{\pl_2 \nabla_t \Psi_t}_{O_{L^\infty}(w)} e_1 \cdot (R^t_0)^T\underbrace{(\pl_1 R^t)}_{O_{L^2}(\e/t)} \nabla_t \Psi_t e_2 
			+ \frac{w^2}{\e}\nabla_t \Psi_t e_1 \cdot (R^t_0)^T\pl_1\underbrace{\brk{ \frac{1}{w}\pl_2 R^t}}_{O_{L^2}(\e/t)} \nabla_t \Psi_t e_2 \\
		&\qquad \qquad +  \frac{w}{\e} \nabla_t \Psi_t e_1 \cdot (R^t_0)^T \underbrace{(\pl_1 R^t)}_{O_{L^2}(\e/t)}  \underbrace{\pl_2 \nabla_t \Psi_t}_{O_{L^\infty}(w)} e_2
			+O_{W^{-2,2}}(w)+ O_{W^{-1,1}}(w,w^2/t)+O_{W^{-2,1}}(w^2\e/t^2)\\
	&\qquad = O_{L^2}(w^2/t) + O_{W^{-1,2}}(w^2/t) +O_{W^{-2,2}}(w)+ O_{W^{-1,1}}(w,w^2/t)+O_{W^{-2,1}}(w^2\e/t^2),
\end{split}
\]
where we used (3) and (4) in Lemma~\ref{lem:rigidity}, together with the fact that $\pl_2  \nabla_t \Psi_t$ is of order $w$ in $L^\infty$.
Going to the limit and using the assumption $w^2 \ll t$, we therefore have 
\[
\pl_2^2((\nabla_t\Psi_t)^T (R^t_0)^T R^t G^t \nabla_t \Psi_t)_{11} \to 0\qquad \text{in $W^{-2,1}$}.
\]
On the other hand, since $\nabla_t\Psi_t\to \Qo $ uniformly, $(R^t_0)^T R^t \to I$ in $L^2$ and are uniformly bounded, and $G^t \weakly G$ in $L^2$, we obtain that 
\[
\pl_2^2((\nabla_t\Psi_t)^T (R^t_0)^T R^t G^t \nabla_t \Psi_t)_{11}  \weakly \pl_2^2(\Qo ^T G \Qo ) \qquad \text{in $W^{-2,1}$},
\]
which completes the proof.
\end{proof}
\medskip

%%%%%%%%%%%%%%%
We are now ready to prove the lower bound for Theorem~\ref{thm:narrow_Gauss_Gamma}.

\begin{proposition}\label{thm:w2lb}
Let $(u_t)$ be a sequence in $W^{1,2}(U;\R^3)$ satisfying
\[
\E_{t,w}(u^t) \le Cw^4.
\]
Then
\[
\liminf \frac{1}{w^4} \E_{t,w}(u^t)\geq \frac{1}{12}\dashint_0^L \tcalQ_2\brk{x_1,\brk{\begin{matrix}
									\dL & \dM  \\
									\dM & \bdN 
									\end{matrix}}}\,dx_1
		+\frac{1}{720} \dashint_0^L  \tcalQ_1\brk{x_1}\DG(x_1)^2\,dx_1,
\]
where $\dL$, $\dM$, and $\dN$ are the functions provided by Theorem~\ref{thm:narrow_compactness} and $\DG$ is the Gauss-deficit.
\end{proposition}

\begin{proof}
The sequence $(u_t)$ satisfies the assumptions of Theorem~\ref{thm:narrow_compactness}, thus all the results of the previous section hold true.
By \eqref{eq:energyrew}, Lemma~\ref{lemma:G}, and standard Taylor expansion arguments (e.g., \cite[p.~211]{FJM06}) we have
\[%\label{eq:liminf_ineq}
\begin{split}
\liminf \frac{1}{w^4} \E_{t,w}(u^t) &\ge \dashint_U \calQ_3\brk{G - \frac{1}{2}\DG(x_1)x_2^2 \Qo e_1 \otimes \Qo  e_1}\,dx \\
	& = \dashint_U \tcalQ_3\brk{x_1, \Qo ^T G \Qo  - \frac{1}{2}\DG(x_1)x_2^2 e_1 \otimes  e_1}\,dx \\
	&\ge \dashint_U \tcalQ_2\brk{x_1, -x_3\brk{\begin{matrix}
									\dL & \dM  \\
									\dM & \dN 
									\end{matrix}}
							+ \brk{\begin{matrix}
									f_0(x_1) + f_1(x_1) x_2 - \frac{1}{2}\DG(x_1)x_2^2 & \olda_{12}  \\
									\olda_{21} & \olda_{22}\\
								\end{matrix}}}\,dx. 
\end{split}
\]
Expanding the quadratic form we obtain
\[
\begin{split}				
&\dashint_U \tcalQ_2\brk{x_1, -x_3\brk{\begin{matrix}
									\dL & \dM  \\
									\dM & \dN 
									\end{matrix}}
							+ \brk{\begin{matrix}
									f_0(x_1) + f_1(x_1) x_2 - \frac{1}{2}\DG(x_1)x_2^2 & \olda_{12}  \\
									\olda_{21} & \olda_{22}\\
								\end{matrix}}}\,dx \\				
	&\qquad = \frac{1}{12}\dashint_{(0,L)\times (-\tfrac{1}{2},\tfrac{1}{2})} \tcalQ_2\brk{x_1,\brk{\begin{matrix}
									\dL & \dM  \\
									\dM & \dN 
									\end{matrix}}}\,dx_1\,dx_2
									\\
	&\qquad \qquad \qquad 	+ \dashint_{(0,L)\times (-\tfrac{1}{2},\tfrac{1}{2})} \tcalQ_2\brk{x_1, \brk{\begin{matrix}
									f_0(x_1) + f_1(x_1) x_2 - \frac{1}{2}\DG(x_1)x_2^2 & \olda_{12}  \\
									\olda_{21} & \olda_{22}\\
								\end{matrix}}}\,dx_1\,dx_2 \\
	&\qquad \ge \frac{1}{12}\dashint_0^L \tcalQ_2\brk{x_1,\brk{\begin{matrix}
									\dL & \dM  \\
									\dM & \bdN 
									\end{matrix}}}\,dx_1
		+ \dashint_{(0,L)\times (-\tfrac{1}{2},\tfrac{1}{2})} \tcalQ_1\brk{x_1}\brk{f_0(x_1) + f_1(x_1) x_2 - \frac{1}{2}\DG(x_1)x_2^2}^2\,dx_1\,dx_2\\
	&\qquad \ge \frac{1}{12}\dashint_0^L \tcalQ_2\brk{x_1,\brk{\begin{matrix}
									\dL & \dM  \\
									\dM & \bdN 
									\end{matrix}}}\,dx_1
		+ \frac{1}{4} \dashint_{(0,L)\times (-\tfrac{1}{2},\tfrac{1}{2})}\tcalQ_1\brk{x_1}\DG(x_1)^2 \brk{x_2^2-\frac{1}{12}}^2\,dx_1\,dx_2\\
	&\qquad = \frac{1}{12}\dashint_0^L \tcalQ_2\brk{x_1,\brk{\begin{matrix}
									\dL & \dM  \\
									\dM & \bdN 
									\end{matrix}}}\,dx_1
		+\frac{1}{720} \dashint_0^L  \tcalQ_1\brk{x_1}\DG(x_1)^2\,dx_1,
\end{split}
\]
which proves the thesis.
\end{proof}
\medskip

%%%%%%%%
Similarly, we now prove the lower bound for Theorem~\ref{thm:narrow_Codazzi_Gamma}.

\begin{proposition}\label{thm:w2lb2}
Assume $\DG\equiv 0$. Let $(u_t)$ be a sequence in $W^{1,2}(U;\R^3)$ satisfying
\[
\E_{t,w}(u^t) \le Cw^2t^2.
\]
Then
\[
\liminf \frac{1}{w^2t^2} \E_{t,w}(u^t)\geq \frac{1}{12}\dashint_0^L \tcalQ_2\brk{x_1,\brk{\begin{matrix}
									\dL & \dM  \\
									\dM & \bdN 
									\end{matrix}}}\,dx_1		
		+\frac{1}{144}\dashint_0^L \tcalQ_2^\circ\brk{x_1,\DC_1(x_1), \DC_2(x_1)}\,dx_1,
\]
where $\dL$, $\dM$, and $\dN$ are the functions provided by Theorem~\ref{thm:narrow_compactness} and $\DC$ is the Codazzi-deficit.
\end{proposition}

\begin{proof}
Arguing as in the proof of Proposition~\ref{thm:w2lb}, we have
\[
\begin{split}
&\liminf \frac{1}{w^2t^2} \E_{t,w}(u^t) \ge \dashint_U \calQ_3\brk{G +x_2x_3\Qo \brk{ \begin{matrix} \DC_1(x_1) & \DC_2(x_1) & 0 \\  \DC_2(x_1) & X_{22}(x_1) & 0 \\ 0 & 0 & 0 \end{matrix}}\Qo ^T}\,dx \\
	&\qquad = \dashint_U \tcalQ_3\brk{x_1, \Qo ^T G \Qo  +x_2x_3\brk{ \begin{matrix} \DC_1(x_1) & \DC_2(x_1) & 0 \\  \DC_2(x_1) & X_{22}(x_1) & 0 \\ 0 & 0 & 0 \end{matrix}}}\,dx \\
	&\qquad \ge \dashint_U \tcalQ_2\brk{x_1, -x_3\brk{\begin{matrix}
									\dL +\DC_1(x_1)x_2& \dM  +\DC_2(x_1)x_2 \\
									\dM +\DC_2(x_1)x_2 & \dN +X_{22}x_2
									\end{matrix}}
							+ \brk{\begin{matrix}
									f_0(x_1) + f_1(x_1) x_2 & \olda_{12}  \\
									\olda_{21} & \olda_{22}\\
								\end{matrix}}}\,dx \\
	&\qquad \geq \frac{1}{12}\dashint_{(0,L)\times (-\tfrac{1}{2},\tfrac{1}{2})} \tcalQ_2\brk{x_1,\brk{\begin{matrix}
									\dL +\DC_1(x_1)x_2& \dM  +\DC_2(x_1)x_2 \\
									\dM  +\DC_2(x_1)x_2 & \dN +X_{22}x_2
									\end{matrix}}}\,dx_1\,dx_2 \\
	&\qquad = \frac{1}{12}\dashint_0^L \tcalQ_2\brk{x_1,\brk{\begin{matrix}
									\dL & \dM  \\
									\dM & \bdN 
									\end{matrix}}}\,dx_1		
		+\frac{1}{12}\dashint_{(0,L)\times (-\tfrac{1}{2},\tfrac{1}{2})} \tcalQ_2\brk{x_1,\brk{\begin{matrix}
									\DC_1(x_1)x_2& \DC_2(x_1)x_2 \\
									\DC_2(x_1)x_2 & X_{22}x_2+\dN-\bdN
									\end{matrix}}}\,dx_1\,dx_2 \\
	&\qquad \ge \frac{1}{12}\dashint_0^L \tcalQ_2\brk{x_1,\brk{\begin{matrix}
									\dL & \dM  \\
									\dM & \bdN 
									\end{matrix}}}\,dx_1		
		+\frac{1}{144}\dashint_0^L \tcalQ_2^\circ\brk{x_1,\DC_1(x_1), \DC_2(x_1)}\,dx_1.
\end{split}
\]
\end{proof}

%%%%%%%%%%%%%%%%%%%%%%%
\subsection{Recovery sequence}
\label{sec:narrow_recovery}
In this section we construct the recovery sequences for Theorems~\ref{thm:narrow_Gauss_Gamma}--\ref{thm:narrow_Codazzi_Gamma}, thus completing their proofs.
We start by proving the existence of a recovery sequence for the scaling regime $\e=w^2$ as in Theorem~\ref{thm:narrow_Gauss_Gamma}.

\begin{proposition}\label{prop:narrow_Gauss_recseq}
Assume $\e=w^2$.
Let $\dL$, $\dM$, $\dN\in L^2((0,L))$. 
Then there exists a sequence $(u_t)$ in $W^{1,2}(U;\R^3)$ such that $(u^t)$ converges to $(\dL,\dM,\dN)$ in the sense of Theorem~\ref{thm:narrow_compactness}, and
\beq\label{eq:limsup1}
\lim \frac{1}{w^4} \E_{t,w}(u^t)=\frac{1}{12}\dashint_0^L \tcalQ_2\brk{x_1,\brk{\begin{matrix}
									\dL & \dM  \\
									\dM & \dN 
									\end{matrix}}}\,dx_1
		+\frac{1}{720} \dashint_0^L  \tcalQ_1\brk{x_1}\DG(x_1)^2\,dx_1.
\eeq
\end{proposition}

\begin{proof}
By density arguments we can assume that $\dL,\dM$, and $\dN$ belong to $C^\infty([0,L])$.
Let $\oldg\in C^\infty([0,L];\R^3)$ be such that
\beq\label{eq:gamma_equality}
	\tcalQ_3\brk{x_1,\brk{\begin{matrix}
		\dL & \dM & \oldg_1 \\
		\dM & \dN & \oldg_2 \\
		0 & 0 & \oldg_3
		\end{matrix}}}
	=\tcalQ_2\brk{x_1,\brk{\begin{matrix}
		\dL & \dM\\
		\dM & \dN 
		\end{matrix}}}
	\qquad \text{for every } x_1\in[0,L]
	\eeq
and let $\eta_2, \eta_3\in C^\infty([0,L];\R^3)$ be such that
	\beq\label{eq:eta_equality}
	\tcalQ_3\brk{x_1, \brk{\begin{matrix}
		-\DG & \eta_{12} & \eta_{13} \\
		0 & \eta_{22} & \eta_{23} \\
		0 & \eta_{32} & \eta_{33}
		\end{matrix}}}
	=\tcalQ_1(x_1) \DG(x_1)^2
	\qquad \text{for every } x_1\in[0,L].
	\eeq
Let $\bar{R}_t$ be the solution to
\beq\label{eq:barR_t_recovery}
\begin{cases}
\bar{R}_t'(s) = \frac{\e}{t} \bar{R}_t(s) A'(s) & \text{ for } s\in[0,L],\\
\bar{R}_t(0) = \exp\brk{\frac{\e}{t}A(0)},
\end{cases}
\eeq
where $A$ is a skew-symmetric matrix satisfying
\beq\label{eq:A_recovery}
A' = \Qo \brk{\begin{matrix} 0 & 0 & -\dL \\ 0 & 0 & -\dM \\ \dL & \dM & 0 \end{matrix}}\Qo ^T.
\eeq

Setting
\beq\label{defth}
\olds(x_1) = \Psi(x_1,0,0),
\eeq
we define
\beq\label{eq:rec_seq_def}
\begin{split}
u^t(x) &= \int_0^{x_1} \bar{R}_t(s)\olds'(s)\brk{1+ \frac{\e}{24}\DG(s)}\,ds 
	+ \bar{R}_t(x_1)\brk{\Psi_t(x) - \olds(x_1)} 
	-\e w x_2x_3 \bar{R}_t(x_1)\Qo (x_1) \brk{\begin{matrix} \dM \\ \dN \\ 0\end{matrix}} \\
	&\quad + \frac12 \e \bar{R}_t(x_1)\Qo (x_1) \brk{\frac{w^2}{t}x_2^2\brk{\begin{matrix} 0 \\ 0\\ \dN\end{matrix}} - t x_3^2 \oldg(x_1) + \frac{w}{3}\brk{x_2^3 - \frac{x_2}{4}}\eta_2(x_1)+ t\brk{x_2^2 -\frac{1}{12}}x_3 \eta_3(x_1)}.
\end{split}
\eeq
As shown below, the first line in \eqref{eq:rec_seq_def} gives us the terms that appear in the limiting energy, whereas the second line serves as a correction: the first addend anti-symmetrizes a $\frac{w \e}{t}$ error due to the second addend in the first line, and the other terms in the second line yield the functions $\oldg$ and $\eta$ that achieve the equalities \eqref{eq:gamma_equality}--\eqref{eq:eta_equality}.

Indeed, we have
\[
\begin{split}
\bar{R}_t^T\pl_1 u^t &= \pl_1 \Psi_t (x) + \frac{\e}{24}\DG(x_1)\olds'(x_1) + \frac{\e}{t}  A'(x_1) \brk{\Psi_t(x) - \olds(x_1)} + o(\e) \\
	&= \pl_1 \Psi_t (x) + \frac{\e}{24}\DG(x_1)\olds'(x_1) + \frac{\e}{t} A'(x_1) \Qo (x_1)\brk{wx_2 e_2 + tx_3 e_3} + o(\e) \\
	&= \pl_1 \Psi_t (x_1) + \e \Qo (x_1)\brk{\begin{matrix}
		\frac{1}{24}\DG - x_3 \dL \\
		-x_3 \dM \\
		 \frac{w}{t}x_2 \dM 
		\end{matrix}} + o(\e),
\end{split}
\]
where in the second line we used the Taylor expansion of $\Psi_t$, and in the third line that $\olds' = \Qo  e_1$.
The derivatives with respect to the second and third variables are more direct:
\[
\frac{1}{w}\bar{R}_t^T\pl_2 u^t = \frac{1}{w}\pl_2 \Psi_t + \e \Qo (x_1) \brk{\brk{\begin{matrix} -x_3 \dM \\ -x_3 \dN \\ \frac{w}{t}x_2 \dN\end{matrix}} + \frac12\brk{x_2^2 - \frac{1}{12}} \eta_2 } + o(\e),
\]
\[
\frac{1}{t}\bar{R}_t^T\pl_3 u^t = \frac{1}{t}\pl_3 \Psi_t + \e \Qo (x_1) \brk{\brk{\begin{matrix} -\frac{w}{t} x_2 \dM \\ -\frac{w}{t} x_2 \dN \\ 0\end{matrix}} - x_3 \oldg + \frac12 \brk{x_2^2 - \frac{1}{12}} \eta_3 }.
\]
Hence we obtain
\beq\label{eq:nabla_u_t_recovery}
\begin{split}
\bar{R}_t^T\nabla_t u^t & = \nabla_t \Psi_t + \e \Qo  \brk{ \frac{w}{t}x_2  \brk{\begin{matrix}
		0 & 0 & -\dM \\
		0 & 0 & -\dN \\
		\dM & \dN & 0
		\end{matrix}} - x_3\brk{\begin{matrix}
		\dL & \dM & \oldg_1 \\
		\dM & \dN & \oldg_2 \\
		0 & 0 & \oldg_3
		\end{matrix}}
	+ \brk{\begin{matrix}
		\frac{1}{24}\DG &0 &0 \\
		0 & 0 & 0 \\
		0 & 0 & 0
		\end{matrix}}}
		\\
	&  \qquad + \e \Qo \brk{\frac12\brk{x_2^2 - \frac{1}{12}}\brk{\begin{matrix}
		0 & \eta_{12} & \eta_{13} \\
		0 & \eta_{22} & \eta_{23} \\
		0 & \eta_{32} & \eta_{33}
		\end{matrix}}
	+o(1)}.
	\end{split}
\eeq
We see that $\bar{R}_t^T\nabla_t u^t - \nabla_t \Psi_t $ is merely of order $O(w\e/t)$ rather than $O(\e)$ (as in Lemma~\ref{lem:rigidity} for the rotations provided by the rigidity lemma), due to the matrix
\[
\tilde{B} = \brk{\begin{matrix}
		0 & 0 & -\dM \\
		0 & 0 & -\dN \\
		\dM & \dN & 0
		\end{matrix}}.
\]
In a sense $\bar{R}_t$ has the role of the field $R_0^t$ from Lemma~\ref{lem:Rt0}.
We fix it by defining
\[
\begin{cases}
\pl_2 R_t(x_1,x_2) = \frac{w\e}{t} R_t(x_1,x_2) \Qo  \tilde{B} \Qo ^T  & \text{ for }(x_1,x_2)\in[0,L]\times[-1/2,1/2],\\
R_t(x_1,0) = \bar{R}_t(x_1) & \text{ for }x_1\in[0,L].
\end{cases}
\]
Noting that 
\[
R_t(x_1,x_2) = \bar{R}_t(x_1)\brk{I + \frac{w\e}{t}x_2 \Qo  \tilde{B} \Qo ^T + O\brk{\frac{w^2\e^2}{t^2}}},
\]
we obtain from \eqref{eq:nabla_u_t_recovery}, the fact that $\tilde{B}$ is skew-symmetric, and that $\nabla_t \Psi_t = \Qo  + O(w)$, that
\[
\begin{split}
R_t^T\nabla_t u^t &= \bar{R}_t^T\nabla_t u^t - \frac{w\e}{t} x_2 \Qo  \tilde{B} \Qo ^T \nabla_t\Psi_t + O\brk{\frac{w^2\e^2}{t^2}}  \\
	&= \bar{R}_t^T\nabla_t u^t - \frac{w\e}{t} x_2 \Qo  \tilde{B} + O\brk{\frac{w^2\e}{t}}  \\
	&= \nabla_t \Psi_t + \e \Qo  \brk{ - x_3\brk{\begin{matrix}
		\dL & \dM & \oldg_1 \\
		\dM & \dN & \oldg_2 \\
		0 & 0 & \oldg_3
		\end{matrix}}
	+ \brk{\begin{matrix}
		\frac{1}{24}\DG &0 &0 \\
		0 & 0 & 0 \\
		0 & 0 & 0
		\end{matrix}}
	+ \frac12 \brk{x_2^2 - \frac{1}{12}}\brk{\begin{matrix}
		0 & \eta_{12} & \eta_{13} \\
		0 & \eta_{22} & \eta_{23} \\
		0 & \eta_{32} & \eta_{33}
		\end{matrix}}
	+o(1)},
\end{split}
\]
showing that
\[
\Qo ^T \frac{R_t^T\nabla_{t} u^t (\nabla_t \Psi_t)^{-1} -I}{\e} \Qo  = -x_3\brk{\begin{matrix}
		\dL & \dM & \oldg_1 \\
		\dM & \dN & \oldg_2 \\
		0 & 0 & \oldg_3
		\end{matrix}}
	+ \brk{\begin{matrix}
		\frac{1}{24}\DG &0 &0 \\
		0 & 0 & 0 \\
		0 & 0 & 0
		\end{matrix}}
	+ \frac12 \brk{x_2^2 - \frac{1}{12}}\brk{\begin{matrix}
		0 & \eta_{12} & \eta_{13} \\
		0 & \eta_{22} & \eta_{23} \\
		0 & \eta_{32} & \eta_{33}
		\end{matrix}}
	+ o(1).
\]
By \eqref{eq:energyrew}, Taylor expansion of $\calW$, and \eqref{eq:metlim} it is immediate to see that
\[
\begin{split}
\lim \frac{1}{w^4} \E_{t,w}(u^t) 
	& = \dashint_U x_3^2\,\tcalQ_3\brk{x_1, \brk{\begin{matrix}
		\dL & \dM & \oldg_1 \\
		\dM & \dN & \oldg_2 \\
		0 & 0 & \oldg_3
		\end{matrix}}
}\,dx \\
	&\qquad + \frac14 \dashint_U \brk{x_2^2 - \frac{1}{12}}^2 \tcalQ_3\brk{x_1, \brk{\begin{matrix}
		-\DG & \eta_{12} & \eta_{13} \\
		0 & \eta_{22} & \eta_{23} \\
		0 & \eta_{32} & \eta_{33}
		\end{matrix}}}\,dx,
\end{split}
\]
which proves \eqref{eq:limsup1} by \eqref{eq:gamma_equality} and \eqref{eq:eta_equality}.

We are left to verify that $u^t$ converges to $(\dL,\dM,\dN)$ as in Theorem~\ref{thm:narrow_compactness}.
Indeed, noting that $\bar{R}_t = I + \frac{\e}{t}A + O(\frac{\e^2}{t^2})$, it is immediate from \eqref{eq:nabla_u_t_recovery} that
\[
\nabla_t u^t - \nabla_t \Psi_t = \frac{\e}{t}A \nabla_t \Psi_t + o(\e/t) = \frac{\e}{t}A \Qo  + o(\e/t),
\]
hence the convergence to $(\dL,\dM)$ follows as in \eqref{eq:convergence_notion1}.
For the convergence to $\dN$, we note that from Theorem~\ref{thm:narrow_compactness}, Lemma~\ref{lem:compactness1}(3), and the fact that $R_t^T\nabla_t u^t - \nabla_t \Psi_t= O(\e)$, it is sufficient to show that $\frac{t}{w\e} \pl_2 R^t \to B$, where $B = \Qo  \tilde{B} \Qo ^T$.
This is immediate from the definition of $R_t$ and the fact that $R_t \to I$ uniformly.
\end{proof}
\medskip

%%%%%%%%%%
We now construct the recovery sequence for the scaling $\e=wt$ as in Theorem~\ref{thm:narrow_Codazzi_Gamma}.
\begin{proposition}
Assume $\DG\equiv 0$ and $\e=wt$. Let $\dL$, $\dM$, $\dN\in L^2((0,L))$.
Then there exists a sequence $(u_t)$ in $W^{1,2}(U;\R^3)$ such that $(u^t)$ converges to $(\dL,\dM,\dN)$ in the sense of Theorem~\ref{thm:narrow_compactness}, and
\[
\lim \frac{1}{w^2t^2} \E_{t,w}(u^t)=\frac{1}{12}\dashint_0^L \tcalQ_2\brk{x_1,\brk{\begin{matrix}
									\dL & \dM  \\
									\dM & \dN 
									\end{matrix}}}\,dx_1		
		+\frac{1}{144}\dashint_0^L \tcalQ_2^\circ\brk{x_1,\DC_1(x_1), \DC_2(x_1)}\,dx_1.
\]
\end{proposition}

\begin{proof}
By density arguments we can assume that $\dL$, $\dM$, and $\dN$ belong to $C^\infty([0,L])$.
Let $\olda\in C^\infty([0,L])$ be such that
\beq\label{eq:6th_line_eq}
	\tcalQ_2\brk{x_1,\brk{\begin{matrix}
		\DC_1 & \DC_2 \\
		\DC_2 & X_{22} +\olda 
		\end{matrix}}}
	=\tcalQ_2^\circ(x_1,\DC_1 , \DC_2)
	\qquad \text{for every } x_1\in[0,L]
	\eeq
and let $\oldg\in C^\infty([0,L]\times[-1/2,1/2];\R^3)$ be of the form $\oldg_i(x_1,x_2) = \oldg_i^0(x_1) + \oldg_i^1(x_1)x_2$ with
\beq
	\tcalQ_3\brk{x_1,\brk{\begin{matrix}
		\dL & \dM & \oldg_1^0 \\
		\dM & \dN & \oldg_2^0 \\
		0 & 0 & \oldg_3^0
		\end{matrix}}}
	=\tcalQ_2\brk{x_1,\brk{\begin{matrix}
		\dL & \dM\\
		\dM & \dN 
		\end{matrix}}}
	\qquad \text{for every } x_1\in[0,L],
	\eeq
	and
	\beq\label{eq:3rd_line_eq2}
	\tcalQ_3\brk{x_1,\brk{\begin{matrix}
		\DC_1 & \DC_2 & \oldg_1^1 \\
		\DC_2 & X_{22} +\olda  & \oldg_2^1 \\
		0 & 0 & \oldg_3^1
		\end{matrix}}}
	=\tcalQ_2\brk{x_1,\brk{\begin{matrix}
		\DC_1 & \DC_2\\
		\DC_2 & X_{22} + \olda 
		\end{matrix}}}
	\qquad \text{for every } x_1\in[0,L].
	\eeq
Similarly as in the proof of Proposition~\ref{prop:narrow_Gauss_recseq}, we consider
\[
\begin{split}
u^t(x) &= \int_0^{x_1} \bar{R}_t(s)\olds'(s)\,ds 
	+ \bar{R}_t(x_1)\brk{\Psi_t(x) - \olds(x_1)} 
	-\e w x_2x_3 \bar{R}_t(x_1)\Qo (x_1) \brk{\begin{matrix} \dM \\ \dN +\frac{1}{2}x_2\olda  \\ 0\end{matrix}}\\
	&\quad + \frac{1}{2}\e \bar{R}_t(x_1)\Qo (x_1) \brk{\frac{w^2}{t}x_2^2\brk{\begin{matrix} 0 \\ 0\\ \dN + \frac{1}{3}x_2 \olda \end{matrix}} - t x_3^2 (\oldg^0 + x_2 \oldg^1)  },
\end{split}
\]
where $\olds$ is defined in \eqref{defth} and $\bar{R}_t$ is given by \eqref{eq:barR_t_recovery}--\eqref{eq:A_recovery}, as before.

By \eqref{eq:Q_0_Psi}, we obtain
\[
\bar{R}_t\nabla_{t} u^t= \nabla_t \Psi_t -\e x_3\Qo \brk{\brk{\begin{matrix}
		\dL & \dM & \oldg_1^0 \\
		\dM & \dN & \oldg_2^0 \\
		0 & 0 & \oldg_3^0
		\end{matrix}}
	+x_2\brk{\begin{matrix}
		0 & 0 & \oldg_1^1 \\
		0 & \olda & \oldg_2^1 \\
		0 & 0 & \oldg_3^1
		\end{matrix}}
	+ o_{\sym}(1)},
\]
hence 
\[
\Qo ^T \frac{\bar{R}_t\nabla_{t} u^t (\nabla_t \Psi_t)^{-1} -I}{\e} \Qo  = -x_3\brk{\begin{matrix}
		\dL & \dM & \oldg_1^0 \\
		\dM & \dN & \oldg_2^0 \\
		0 & 0 & \oldg_3^0
		\end{matrix}}
	-x_2x_3\brk{\begin{matrix}
		0 & 0 & \oldg_1^1 \\
		0 & \olda & \oldg_2^1 \\
		0 & 0 & \oldg_3^1
		\end{matrix}}
	+ o_{\sym}(1).
\]
The rest of the proof is similar to the $\e=w^2$ case, using \eqref{eq:6th_line_eq}--\eqref{eq:3rd_line_eq2} and~\eqref{eq:metlim2}.
\end{proof}

%%%%%%%%%%%%%%%%%%%%%%%%%%%%%%%%%%
\subsection{Energy scaling of Gauss incompatible ribbons in all regimes}\label{sec:wide_from_narrow}
From the results of the previous sections we can deduce the energy scaling of Gauss-incompatible ribbons in all regimes. This completes the energy-scaling part of Corollary~\ref{cor:informal}(1).
\begin{corollary}\label{cor:wide_from_narrow}
Assume  \eqref{eq:w_ribbon} and that $\lim_{t\to 0} \frac{t}{w^2} < \infty$. 
If the ribbon is Gauss incompatible, that is, $\DG \nequiv 0$, then
\[
\liminf_{t\to 0} \Big( \inf \frac{1}{t^2}\E_{t,w} \Big) > 0.
\]
Furthermore, assume that for some $w_0>0$ there exists $W^{2,\infty}$-isometric immersions of $\calS_{w_0}$, that is,
\beq\label{eq:W2infty_iso}
\BRK{\phi \in W^{2,\infty}(\calS_{w_0} ; \R^3) ~:~ \nabla'\phi ^T\nabla' \phi = \a \text{ a.e.}} \ne \emptyset,
\eeq
where $\a$ is the induced metric of $\g$ on $\calS_{w_0}$, see \eqref{eq:g_and_a}, and $\nabla' = (\pl_1 \,|\, \pl_2)$.
Then, there exists a constant $C>0$ such that for every $w<w_0$,
\[
\inf \E_{t,w} \leq Ct^2.
\]
\end{corollary}

\begin{proof} We split the proof into two parts.

\paragraph{Lower bound.}
Assume that $t\le Cw^2$ for some $C>0$, and let $(\M_{t,w},\g)$ be a Gauss-incompatible ribbon.
Assume, by contradiction, that $\inf \E_{t,w} < \e^2 \ll t^2$.
In particular, we can assume that there exists maps $u^t\in W^{1,2}(U,\R^3)$ such that $\E_{t,w} (u^t) = \e^2 \ll t^2$.

Let $k=k(t)$ be natural numbers such that $\e\ll \frac{w^2}{k^2} \ll t \ll \frac{w}{k}$, which is possible by our assumptions, and define $\tilde{w} = w/k$.
Partition $\M_{t,w}$ into $k$ ribbons of width $\tilde{w}$, and let $\E_{t,\tilde{w}}^i$, $i=1,\ldots,k$, be the restriction of the energy (divided by the volume) to each of the ribbons.
Since
\[
\e^2 = \E_{t,w} (u^t) = \frac{1}{k} \sum_{i=1}^k \E_{t,\tilde{w}}^i(u^t) 
\]
there exists $j=j(t)$ such that $\E_{t,\tilde{w}}^j(u^t) \le \e^2$.
Now, for each $t$, the $j$th ribbon is a ribbon of thickness $t$, width $\tilde{w}$, hence we can choose natural Fermi coordinates adapted to this ribbon, as in \S\ref{sec:metrics}. 
In these coordinates, the domain is $\M_{t,w}^j = (0,L_t)\times (-\tilde{w}/2,\tilde{w}/2) \times (-t/2,t/2)$, and the metric is given by
\beq
\begin{split}
\bar{\g}^t(z) 
	&= \brk{\begin{matrix}
			(1-\kappa_t(z_1)z_2)^2 - K^\calS_t(z_1,0)z_2^2 + O(z_2^3) & 0 & 0\\
			0 & 1 & 0 \\
			0 & 0 & 1
			\end{matrix}}
		-2z_3 \brk{\begin{matrix}
			\II_t(z_1,z_2) & 0 \\
			0 & 0
			\end{matrix}}
		+O(z_3^2),
\end{split}
\eeq
and, because $\g$ is smooth, and the ribbons' midlines are at distance at most $w/2$ of the midline of the original ribbon, we have that $|L-L_t| \lesssim w$ and
\[\
\left\|\kappa_t\brk{\frac{L_t}{L}z_1} - \kappa(z_1)\right\|_{C^1([0,L])} \lesssim w,
\]
\[
\left\|K^\calS_t\brk{\frac{L_t}{L}z_1,z_2} - K^\calS(z_1,z_2)\right\|_{C^0([0,L]\times [-\tilde{w}/2,\tilde{w}/2])}
+\left\|\II_t\brk{\frac{L_t}{L}z_1,z_2} - \II(z_1,z_2)\right\|_{C^2([0,L]\times [-\tilde{w}/2,\tilde{w}/2])} \lesssim w.
\]
It follows that all these ribbons are Gauss-incompatible, and that the associated Euclidean ribbons maps $\Psi^t$ converge and that $D\Psi^t \to \Qo$ in $C^0$.
Therefore we are in the setting of ribbons with varying geometries, as described in Remark~\ref{remak:geom}.
It is easy to see that the analysis of this section applies to this case as well, and in particular we obtain by Corollary~\ref{cor:scaling_narrow} that
\[
\lim_{t\to 0}\Big( \inf \frac{1}{\tilde{w}^4}\E_{t,\tilde{w}}^j \Big) >0,
\]
in contradiction to  $\E_{t,\tilde{w}}^j(u^t) \le \e^2 \ll \tilde{w}^4$.

\paragraph{Upper bound.}
Let $\phi\in W^{2,\infty}(\calS_{w_0};\R^3)$ be an isometric immersion, and let $\nu_\phi:\calS_{w_0} \to \R^3$ be its normal.
The normal satisfies $\nu_\phi = \frac{\pl_1 \phi \wedge \pl_2 \phi}{|\pl_1 \phi \wedge \pl_2 \phi|}$, where $|\pl_1 \phi\wedge \pl_2 \phi| = \det \a >c>0$ for some $c>0$, and thus $\nu_\phi\in W^{1,\infty}(\calS_{w_0};\R^3)$.
Define a map $y\in W^{1,\infty}(\M_{t,w_0};\R^3)$ via
\[
y(z) = \phi(z') + z_3 \nu_\phi (z').
\]
A direct calculation shows that for almost every $z\in \M_{t,w_0}$,
\[
|\nabla y ^T(z) \nabla y(z) - \g(z)| \le C|z_3|
\]
for some uniform $C>0$.
Therefore, it follows from \eqref{deftg} and from the assumptions (a), (b), and (d) on $\calW$ that for almost every $z\in \M_{t,w_0}$
\[
\calW(\nabla y (\nabla \Psi)^{-1}(\tg\circ \Psi)^{-1/2}) \le C|z_3|^2.
\]
Therefore, for every $w<w_0$, defining $u(x) = y(x_1,wx_2,tx_3)\in W^{1,\infty}(U;\R^3)$, we have
\[
\|\calW(\nabla_{t} u\, (\nabla_t \Psi_t)^{-1} \tg_t^{-1/2})\|_{L^\infty(U)} \lesssim t^2,
\]
hence
\[
\E_{t,w} (u) \lesssim t^2,
\]
which concludes the proof.
\end{proof}
\medskip

Combining the previous result with Corollary~\ref{cor:scaling_narrow}, we immediately obtain the following.

\begin{corollary}\label{cor:wide_from_narrow2}
For Gauss-incompatible ribbons for which \eqref{eq:W2infty_iso} holds we have 
\beq\label{eq:scaling_Gauss}
\inf \E_{t,w} \sim \min\{t^2,w^4\}.
\eeq
\end{corollary} 

\begin{proof}
Assume that \eqref{eq:W2infty_iso} holds and that the ribbon is Gauss-incompatible.
If $t\ll w^2$ or $t\sim w^2$, the previous corollary guarantees that $\inf \E_{t,w} \sim t^2$. If $w^2\ll t \ll w \ll 1$, Corollary~\ref{cor:scaling_narrow} implies that $\inf \E_{t,w} \sim w^4$.
This proves \eqref{eq:scaling_Gauss}.
\end{proof}

%%%%%%%%%%%%%%%%%%%%%%%%%%%%%%%%%%%%%%%%%%%%
\section{Wide ribbons}\label{sec:wide}
In this Section we consider the double limit 
\[
\lim_{w\to 0} \lim_{t\to 0} \frac{1}{t^2 w^\alpha} \E_{t,w},
\]
in the sense of $\Gamma$-convergence,
for an appropriate power $\alpha$ (depending on the geometry). 
In \S\ref{sec:plate_limit} we consider the first $\Gamma$-limit (the plate limit), which applies to any geometry, and in the rest of the section we consider the second $\Gamma$-limit, restricted to the case of a flat mid-surface, for which the metric expansion is as by \eqref{eq:g_estimate_flat}.

%%%%%%
\subsection{The plate limit}\label{sec:plate_limit}
The first limit $ \lim_{t\to 0} \frac{1}{t^2} \E_{t,w}$ is that of non-Euclidean Kirchhoff plate theory, on which a large literature exists.
The case here is very similar to \cite{BLS16,LL20}, though it does not fall precisely under the scope of these papers.
The resulting $\Gamma$-convergence and compactness statements are described below.
As the proof follows the same lines of \cite{BLS16,LL20} with very mild changes, we outline them in Appendix~\ref{appendix}.

To describe the results, we first denote the following quantities, for a given $w>0$:
For $x'=(x_1,x_2)\in \WW = (0,L)\times (-1/2,1/2)$, we let 
\[
\begin{split}
\Fo(x') &= \lim_{t\to 0}\tg_t(x',0) 
	= (\nabla \Phi \,|\, \tn)|_{(x_1,wx_2)}^{-T} \g_t(x',0) (\nabla \Phi \,|\, \tn)|_{(x_1,wx_2)}^{-1} \\
	&= (\nabla \Phi \,|\, \tn)|_{(x_1,wx_2)}^{-T} \brk{\begin{matrix}(1-w\kappa(x_1)x_2)^2 - w^2K^\calS(x_1,0)x_2^2 + O(w^3) & 0 & 0 \\ 0 & 1 & 0 \\ 0 & 0 & 1\end{matrix}} (\nabla \Phi \,|\, \tn)|_{(x_1,wx_2)}^{-1},
\end{split}
\]
where $\Phi$ and $\tn$ are given in Proposition~\ref{prop:Phi}, and
\[
\Ao(x') = \Fo^{1/2}(x') (\nabla \Phi \,|\, \tn)|_{(x_1,wx_2)}.
\]
We omit the dependence of $\Fo$ and $\Ao$ on $w$ for notational simplicity, and note that $\Ao\to \Qo$ uniformly as $w\to 0$.
We further denote $\nabla_w = (\pl_1 \,|\, w^{-1}\pl_2)$.

%%%%%%
\begin{theorem}[Plates, compactness]
\label{thm:compactness_plates}
Fix $w>0$, and let $(u_t)$ be a sequence of configurations in $W^{1,2}(U;\R^3)$  satisfying
\[
\E_{t,w}(u^t)\le Ct^2.
\]
Then, modulo a subsequence and translations, we have
\beq\label{eq:conv_configurations_plates}
u^t \conv{W^{1,2}} \f , \qquad \nabla_t u^t \conv{L^2} (\nabla_w \f  \,|\, \bb_{\f }),
\eeq
where $\f \in W^{2,2}(\WW ; \R^3)$ is a \Emph{rescaled isometric immersion}, that is,
\[
(\nabla_w \f  \,|\, \bb_{\f })\Ao^{-1} \in \SO(3)\,\,\,a.e.,
\]
or, equivalently, $\nabla_w\f ^T\nabla_w \f (x') = \a(x_1,wx_2)$ a.e., where $\a$ is the restriction of $\g$ to $\Sw$, see \eqref{eq:g_and_a}, and $\bb_{\f }$ is the normal to $\f $.
\end{theorem}

%%%%%
\begin{theorem}[Plates, $\Gamma$-convergence]\label{thm:Gconv_plates}
\emph{Lower bound:} Under the assumptions of Theorem~\ref{thm:compactness_plates}, we have
\beq\label{eq:energy_plate}
\liminf_{t\to 0} \frac{1}{t^2}\E_{t,w}(u^t) \ge \tE_w(\f ) := \frac1{12}\dashint_{\WW} \QQw(x', \II^w_{\f }-\II^w)\, dx',
\eeq
where $\QQw: \WW\times \R^{2\times 2}\to\R$ is given by
\beq\label{eq:Q_plate}
\QQw(x',F) = \min_{G\in \R^{3\times 3}} \BRK{ \calQ_3(\Ao^{-T}(x')G\Ao^{-1}(x'))~:~ G_{2\times 2} = F},
\eeq
$\II^w_{\f }$ is the (rescaled) second fundamental form of $\f $, i.e.,
\[
\II^w_\f := \tr((\nabla_w^2 \f )(\bb_{\f })) = - \nabla_w \f ^T \nabla_w \bb_{\f },
\]
and
\[
\II^w(x_1,x_2) = \II(x_1,wx_2).
\]
\emph{Recovery sequence:} For any rescaled isometric immersion $\f  \in W^{2,2}(\WW ; \R^3)$, there exists configurations $u_t\in W^{1,2}(U;\R^3)$ such that \eqref{eq:conv_configurations_plates} holds and
\[
\lim_{t\to 0} \frac{1}{t^2}\E_{t,w}(u^t) = \tE_w(\f ).
\]
\end{theorem}

\medskip

We note that $\QQw(x',\cdot)$ is a positive definite quadratic form on $\R^{2\times 2}_{\text{sym}}$, and that it satisfies 
\begin{equation}\label{Q-coerc}
c_1|M|^2 \le \QQw(x',M) \le c_2|M|^2 \quad \text{ for every } M\in \R^{2\times 2}_{\text{sym}}, \,\, x'\in \WW,
\end{equation}
for some constants $c_2>c_1>0$, which are independent of $x'\in \WW$ and of $w$.
Furthermore, we have that for all $w>0$, the restriction of $\QQw$ to $x_2=0$ is $\tcalQ_2$ as defined in \eqref{eq:tcalQ2}, independently of $w$. 

As in \cite{BLS16,MS19}, an immediate corollary of Theorems~\ref{thm:compactness_plates}--\ref{thm:Gconv_plates} and of the bounds \eqref{Q-coerc} is the following:
\begin{corollary}\label{cor:plates}
Fix $w$ small enough, and suppose that the set
\[
\BRK{\f \in W^{2,2}(\WW ; \R^3) ~:~ \nabla_w\f ^T\nabla_w \f (x') = \a(x_1,wx_2) \text{ a.e.}}
\]
is non-empty, and that $\a$ and $\II$ do not satisfy the Gauss--Codazzi equations in $\Sw$.
Then, 
\[
\lim_{t\to 0}\Big(  \inf \frac{1}{t^2}\E_{t,w}\Big)  \in (0,\infty).
\]
\end{corollary}

\medskip

The functional $\tE_w$ is defined on $W^{2,2}$-rescaled isometric immersions; the following proposition establishes that the Gauss--Codazzi system remains valid in this regularity.

\begin{proposition}\label{prop:W22_Gauss_Codazzi}
Let $(\calS,\a)$ be a smooth, complete two-dimensional Riemannian manifold with Lipschitz boundary.
Let $\phi\in W^{2,2}(\calS;\R^3)$ be an isometric immersion, that is, $\nabla \phi^T\nabla \phi = \a$ a.e.
Then the normal $\nu_\phi$ of $\phi$ is in $W^{1,2}\cap L^\infty(\calS;\R^3)$, the Gauss equation $\det \II_\phi = K^\calS \det \a$ holds in $L^1$, and the Codazzi equations hold in $W^{-1,2}$.
\end{proposition}

\begin{proof}
We first note that since $\phi$ is an isometric immersion and $\a$ is smooth, $\nabla \phi\in W^{1,2} \cap L^\infty$, and $|\pl_1 \phi \wedge \pl_2 \phi| = \det \a \geq c >0$ is bounded away from zero.
A direct calculation shows that $|\nabla \nu_\phi| \le C|\nabla (\pl_1 \phi \wedge \pl_2 \phi)|/\det \a$, hence $\nabla\nu_\phi\in L^2$.
Thus $\nu_\phi \in W^{1,2}\cap L^\infty(\calS;\R^3)$.

This implies that the second fundamental form $(\II_\phi)_{ij} = \langle \pl_{ij} \phi , \nu_\phi \rangle\in L^2$ can also be written as $-\langle \pl_{i} \phi , \pl_j \nu_\phi \rangle$.
To prove the compatibility equations we use the formalism of \cite[Section 5]{Mar05}.
Define $F = (\pl_1 \phi , \pl_2 \phi, \nu_\phi)\in W^{1,2}\cap L^\infty$, then
\[
\pl_i F = F \Gamma_i
\]
where
\[
\Gamma_i = \brk{\begin{matrix} \Gamma_{i1}^1 & \Gamma_{i2}^1 & -(\II_\phi)_i^1 \smallskip\\ \Gamma_{i1}^2 & \Gamma_{i2}^2 & -(\II_\phi)_i^2 \smallskip\\ (\II_\phi)_{i1} & (\II_\phi)_{i2} & 0 \end{matrix}} \in L^2,
\]
$\Gamma_{ij}^k$ are the (smooth) Christoffel symbols of $\a$, $(\II_\phi)_{i}^j= (\II_\phi)_{ik}\a^{kj}$, and $(\a^{kj})$ denotes the inverse of $\a$.
The compatibility condition $\pl_1 \pl_2 F = \pl_2 \pl_1 F$ in $W^{-1,2}$ then implies $\pl_1 (F\Gamma_2) = \pl_2 (F\Gamma_1)$.
Since $(\pl_i F)\Gamma_j \in L^1$ and $F \pl_i \Gamma_j \in \mathcal{D}'$, we have that the Leibniz rule $\pl_i (F\Gamma_j) = (\pl_i F)\Gamma_j + F\pl_i\Gamma_j$ holds in $\mathcal{D}'$.
Thus, we can write the compatibility condition as
\[
F(\Gamma_1 \Gamma_2 - \Gamma_2 \Gamma_1 + \pl_1 \Gamma_2 - \pl_2 \Gamma_1) = 0 \quad \text{in}\,\, \mathcal{D}',
\]
hence 
\[
\Gamma_1 \Gamma_2 - \Gamma_2 \Gamma_1 + \pl_1 \Gamma_2 - \pl_2 \Gamma_1 = 0 \quad \text{in}\,\, \mathcal{D}'.
\]
Expanding, this formula implies that the Gauss and Codazzi equations hold in $\mathcal{D}'$ (see \cite[eq. (28)--(29)]{Mar05} ).
Since the expression in the Gauss equation is defined in $L^1$ and those in the Codazzi equations in $W^{-1,2}$, these equations hold in these spaces as well.
\end{proof}

%%%%%%%%%%%%%%%%
\subsection{Gauss incompatibility}\label{sec:wide_gauss}
In this section we consider the behavior of the plate energy $\tE_w$ obtained in \eqref{eq:energy_plate}, as $w\to 0$, and analyze its limit.
From here we restrict ourselves to flat mid-surfaces, in which the metric has the structure \eqref{eq:g_estimate_flat}, and thus the map $\Fo$ from the previous section has the explicit form
\[
\Fo(x') = (\nabla \Phi \,|\, \tn)|_{(x_1,wx_2)}^{-T} \brk{\begin{matrix}(1-w\kappa(x_1)x_2)^2 & 0 & 0 \\ 0 & 1 & 0 \\ 0 & 0 & 1\end{matrix}} (\nabla \Phi \,|\, \tn)|_{(x_1,wx_2)}^{-1}.
\]
While our two-dimensional plate energy \eqref{eq:energy_plate}--\eqref{eq:Q_plate} is slightly different (and more general) from the one considered in \cite{FHMP16b}, 
the compactness statement and its proof are the same as in \cite[Lemma~2]{FHMP16b}, as they only depend on the isometry constraint and on the coercivity \eqref{Q-coerc} of $\QQw$. 
Therefore we state the result here without proof.
The assumption of a flat mid-surface is essential in the construction of the recovery sequence, while the compactness statement and the lower bound actually hold for any geometry $\a$ of the mid-surface.
\begin{theorem}[Wide ribbons, compactness]
\label{thm:compactness_wide_ribbons}
Let $(\f_w)$ be a sequence of rescaled isometric immersions in $W^{2,2}(\WW;\R^3)$ such that 
$$
\tE_w(\f_w)\leq C
$$
for every $w>0$. Then, modulo a subsequence and translations, we have
\begin{equation}\label{Sad-conv}
\f_w \wconv{W^{2,2}} \f, \quad \nabla_w \f_w \wconv{W^{1,2}} (d_1,d_2), \quad \II^w_{\f_w} \wconv{L^2}  \II_d :=
\brk{\begin{matrix} d_1'\cdot d_3 & d_2'\cdot d_3 \\ d_2'\cdot d_3 & \gamma \end{matrix}},
\end{equation}
where $\f\in W^{2,2}((0,L);\R^3)$, $d=(d_1,d_2,d_3)\in W^{1,2}((0,L);\SO(3))$ with $d_1 = y'$ and $d_1'\cdot d_2=\kappa$, and $\gamma\in L^2(\WW)$.
\end{theorem}

%%%%%%%%
\begin{theorem}[Wide ribbons, $\Gamma$-convergence]
\label{thm:Gconv_wide_ribbons}
Assume that the mid-surface metric $\a$ is flat.

\emph{Lower bound:} 
Under the assumptions of Theorem~\ref{thm:compactness_wide_ribbons}, we have
$$
\liminf_{w\to 0} \tE_w(\f_w) \ge J\big(\overline{\II}_d\big)  :=  
	\frac1{12} \dashint_0^L \brk{\tcalQ_2(x_1,\overline{\II}_d-\II^0(x_1)) + \alpha_\calQ^+(x_1) (\det \overline{\II}_d)^+ + \alpha_\calQ^-(x_1)(\det \overline{\II}_d)^-}\,dx_1,
$$
where $\overline{\II}_d$ is the limiting second fundamental form along the midline given by
\[
\overline{\II}_d(x_1) := \int_{-1/2}^{1/2} \II_d(x_1,x_2)\,dx_2 = \brk{\begin{matrix} d_1'\cdot d_3 & d_2'\cdot d_3 \\ d_2'\cdot d_3 & \bar{\gamma} \end{matrix}}, \qquad
\bar{\gamma}(x_1):= \int_{-1/2}^{1/2} \gamma(x_1,x_2)\,dx_2,
\]
and
\[
\alpha^\pm_\calQ(x_1) := \sup\BRK{\alpha>0 ~:~ \tcalQ_2(x_1,M) \pm \alpha \det M \ge 0 \,\,\text{for every}\,\, M\in \R^{2\times 2}_\text{sym}}.
\]

\emph{Recovery sequence:} 
For any $\f \in W^{2,2}((0,L);\R^3)$ and $d=(d_1,d_2,d_3)\in W^{1,2}((0,L);\SO(3))$ with $d_1 = y'$ and $d_1'\cdot d_2=\kappa$, and $\gamma\in L^2(0,L)$, there exists a sequence $(\f_w)$ of rescaled isometric immersions in $W^{2,2}(\WW;\R^3)$ such that convergences \eqref{Sad-conv}
are satisfied and
\[
\lim_{w\to 0} \tE_w(\f_w)= J(\II_d),
\]
where
\[
\II_d(x_1) = \brk{\begin{matrix} d_1'\cdot d_3 & d_2'\cdot d_3 \\ d_2'\cdot d_3 &\gamma \end{matrix}}.
\]
\end{theorem}
\medskip%

Theorem~\ref{thm:Gconv_wide_ribbons} thus establishes $J$ as the $\Gamma$-limit of $\tE_w$.
Note that $\tcalQ_2$ is the pointwise limit of $\QQw$, as $w\to0$; thus, $\tcalQ_2$ inherits the coercivity properties \eqref{Q-coerc}:
\begin{equation}\label{Q-coerc2}
c_1|M|^2 \le \tcalQ_2(x_1,M) \le c_2|M|^2 \quad \text{ for every } M\in \R^{2\times 2}_{\text{sym}}, \,\, x_1\in (0,L).
\end{equation}
In particular, this guarantees that there exist $C_0> c_0>0$ such that $c_0\leq \alpha_\calQ^\pm(x_1)\leq C_0$ for all $x_1\in (0,L)$.

\begin{remark}
The expression of the $\Gamma$-limit in Theorem~\ref{thm:Gconv_wide_ribbons} differs slightly from that presented in \cite{FHMP16,FHMP16b}. In these works the limiting functional only depends  on $d$, as the energy density is already minimized pointwise over all admissible $\bar\gamma\in L^2(0,L)$.
Here, instead, we opted for a more detailed form of the $\Gamma$-limit. This choice is motivated by the fact that the full limiting second fundamental form along the midline, $\overline{\II}_d$, is the key geometric object through which the shape transition is manifested (compare with the narrow ribbon limit, where the limiting second fundamental form coincides with the reference form $\II^0$ along the midline). 
\end{remark}

\begin{remark}
The compactness result Theorem~\ref{thm:compactness_wide_ribbons} and the lower bound in Theorem~\ref{thm:Gconv_wide_ribbons} continue to hold (with the same proofs) for arbitrary geometries, that is, not only for ribbons with flat mid-surfaces.
The only change is that $\det \overline{\II}_d$ should be replaced by $\det \overline{\II}_d - K^\calS(x_1,0)$. 
However, we suspect that in this general case this lower bound is not tight, and cannot be matched by a recovery sequence.
\end{remark}

%%%%%
\begin{proof1}{of Theorem~\ref{thm:Gconv_wide_ribbons}:} We split the proof into two parts.
\paragraph{Lower bound.}
Let $x'\in\WW$ and $w>0$. For every $F\in \R^{2\times 2}$ let $G_F=G_F(x')\in \R^{3\times 3}$ be such that $(G_F)_{2\times 2} = F$ and
$$
\QQw(x',F)=\calQ_3(\Ao^{-T}(x')G_F\Ao^{-1}(x')).
$$
Thus, we have
\begin{eqnarray}
\QQw(x',F) & \geq & \calQ_3(\Qo(x_1)G_F \Qo^{T}(x_1)) - c\brk{|\Ao^{-T}(x')-\Qo(x_1)|+|\Ao^{-1}(x')-\Qo^{T}(x_1)| }|F|^2
\nonumber
\\
& \geq & \tcalQ_2(x_1,F)- c\brk{\|\Ao^{-T}-\Qo\|_{C^0}+\|\Ao^{-1}-\Qo^{T}\|_{C^0} }|F|^2, \nonumber
\end{eqnarray}
where we used the definition of $\tcalQ_2$.
Similarly, we obtain the converse inequality and conclude that
\beq
\label{lbound-vw}
\Abs{\QQw(x',F) -  \tcalQ_2(x_1,F)} \lesssim \brk{\|\Ao^{-T}-\Qo\|_{C^0}+\|\Ao^{-1}-\Qo^{T}\|_{C^0} }|F|^2
\eeq
for every $F\in \R^{2\times 2}$.
Therefore, writing $\tcalQ_2(x_1,F)=\frac12 \mathcal L(x_1)F\cdot F$, we obtain
\begin{eqnarray*}
\liminf_{w\to 0} \tE_w(\f_w) & \geq & \liminf_{w\to 0} \frac1{12}\dashint_{\WW} \tcalQ_2(x_1, \II^w_{\f_w}-\II^w)\, dx
\\
& \geq &  \liminf_{w\to 0}\frac1{12}\dashint_{\WW} \brk{\tcalQ_2(x_1, \II^w_{\f_w}) +\tcalQ_2(x_1, \II^w) -\mathcal L(x_1)\II^w_{\f_w}\cdot \II^w}\, dx,
\end{eqnarray*}
where we used \eqref{lbound-vw}, the $L^2$-boundedness of $ \II^w_{\f_w}$ (which follows from \eqref{Sad-conv}), and the fact that  $\Ao \to Q_0$ uniformly as $w\to 0$.
The lower bound now follows from the established convergence for $\II^w_{\f_w}$, the Gauss equation for $W^{2,2}$-isometric immersion $\det \II^w_{\f_w} = 0$ (Proposition~\ref{prop:W22_Gauss_Codazzi}), and from \cite[Proposition~9]{FHMP16b}.

\paragraph{Recovery sequence.}
We follow the proof of \cite[Theorem~5(ii)]{FHMP16b}. 
Let $\f \in W^{2,2}((0,L);\R^3)$ and $d=(d_1,d_2,d_3)\in W^{1,2}((0,L);\SO(3))$ with $d_1 = y'$ and $d_1'\cdot d_2=\kappa$, and let $\gamma\in L^2(0,L)$.
We set $\mu:=d_1'\cdot d_3$, $\tau:=d_2'\cdot d_3$.
By \cite[Proposition~9]{FHMP16b} there exists a sequence $(\tilde M^\delta)$ in $L^2((0,L);\R^{2\times 2}_\text{sym})$ such that
$\tilde M^\delta\wconv{L^2} \II_d$, as $\delta\to0$, $\det\tilde M^\delta=0$ for every $\delta>0$, and
$$
\lim_{\delta\to0} \frac1{12}\dashint_0^L \tcalQ_2(x_1,\tilde M^\delta-\II^0)\, dx_1=J(\II_d).
$$
For every $\delta>0$ let $r^\delta\in W^{1,2}((0,L);\SO(3))$ be the frame satisfying the Darboux frame equation
\[
\brk{\begin{matrix}r^\delta_1 \smallskip\\ r^\delta_2 \smallskip\\ r^\delta_3 \end{matrix}}'
= \brk{\begin{matrix}r^\delta_1 \smallskip\\ r^\delta_2 \smallskip\\ r^\delta_3 \end{matrix}}
\brk{\begin{matrix} 0 & -\kappa & -\mu^\delta \\ \kappa & 0 & -\tau^\delta \\ \mu^\delta & \tau^\delta & 0 \end{matrix}}
\]
with $r^\delta_i(0)=d_i$ for $i=1,2,3$, where $\mu^\delta:=\tilde M^\delta_{11}$ and $\tau^\delta:=\tilde M^\delta_{12}$.
Let
$$
\f^\delta(x_1):=\f(0)+\int_0^{x_1} r^\delta(t)e_1\,dt
$$
for every $x_1\in(0,L)$.

Recall that the restriction of $\chi:[0,L]\times [-w_0/2,w_0/2]\to \R^2$, as defined in \eqref{eq:def_chi}, to $\Sw$, is its isometric immersion into $\R^2$, whose midline is $\tilde{\ell}$.
By \cite[Lemma~16 and Proposition~13]{FHMP16b} for every $\delta>0$ there exist a neighborhood $U_\delta$ of $\tilde{\ell}$ and 
a map $u^\delta\in W^{2,2}(U_\delta;\R^3)$ such that $(\nabla u^\delta)^T\nabla u^\delta=I$ in $U_\delta$, $u^\delta\circ{\tilde{\ell}}=\f^\delta$, 
$\nabla u^\delta\circ{\tilde{\ell}} = r^\delta e_1\otimes \Rot e_1 + r^\delta e_2\otimes \Rot e_2$, and
$\II_{u^\delta}\circ\tilde{\ell} =\Rot \tilde M^\delta \Rot ^T$, where $\Rot$ is defined by \eqref{eq:R_def}.
A version of this construction, with details and intuition, appears below in the proof of Theorem~\ref{thm:Gconv_wide_Codazzi}.

Now, for every $w>0$ let $\f^\delta_w(x_1,x_2):=u^\delta(\chi(x_1,w x_2))$ for every $(x_1,x_2)\in\WW$.
It is immediate to see that $\f^\delta_w$ is a rescaled isometric immersion and that
$$
\f^\delta_w \conv{W^{2,2}} \f^\delta, \quad \nabla_w \f^\delta_w \conv{W^{1,2}} (r^\delta e_1,r^\delta e_2), \quad \II^w_{\f^\delta_w} \conv{L^2}  
\Rot^T(\II_{u^\delta}\circ{\tilde{\ell}}) \Rot =\tilde M^\delta,
$$
as $w\to0$. By \eqref{lbound-vw} and dominated convergence we have
$$
\lim_{w\to 0} \tE_w(\f^\delta_w)=\frac1{12}\dashint_{\WW} \tcalQ_2(x_1, \tilde M^\delta-\II^0)\, dx.
$$
The conclusion follows by a diagonal argument.
\end{proof1}
\medskip

With Theorem~\ref{thm:Gconv_wide_ribbons} in hand, we are now in a position to complete the proof of part 1 of Theorem~\ref{thm:wide_informal}.
Indeed, note that if $\det \II^0(x_1) \ne 0$, i.e., there is a \Emph{Gauss incompatibility} between $\g$ and $\II$ at $(x_1,0)$, then
\[
\min \bQ(x_1,\cdot,\cdot) > 0.
\]
This is because $\tcalQ_2(x_1,M-\II^0(x_1))\ge c_1|M-\II^0(x_1)|^2$ and $\alpha^\pm(x_1)>0$, so we cannot have both $\tcalQ_2(x_1,M-\II^0(x_1))=0$ and $\alpha_\calQ^+ (x_1)(\det M)^+ +  \alpha_\calQ^- (x_1)(\det M)^-= 0$ at the same time.
The following proposition provides a more precise estimate and thus establishes part 1 of Theorem~\ref{thm:wide_informal}.

%%%%%%%%%%
\begin{proposition}\label{prop:Jbehavior}
Let $\kappa_1,\kappa_2$ be the principal curvatures (i.e., eigenvalues) of $\II^0$, ordered such that $|\kappa_1|\le |\kappa_2|$.
Then the limiting energy $J$ satisfies the following bounds:
\begin{equation}\label{Jbounds}
c\int_0^L \kappa_1^2\,dx_1 \le \min J \le C\int_0^L \kappa_1^2\,dx_1,
\end{equation}
for some $C>c>0$; in particular,  $\min J =0$ if and only if $\det \II^0\equiv 0$.
That is, the energy scaling $\E_w \sim 1$ is the natural one if and only if the ribbon is Gauss incompatible according to Definition~\ref{def:deficit}.
Furthermore, if $\overline{\II}_*$ is a minimizer of $J$, we have that for some $c>0$,
\begin{equation}\label{Jbounds2}
\Abs{\overline{\II}_*(x_1) - \II_0(x_1)} \ge c |\kappa_1(x_1)| \quad \text{ for a.e. } x_1\in(0,L),
\end{equation}
hence in the case of Gauss incompatibility the limiting second fundamental form differs from the reference one.
\end{proposition}

\begin{proof}
For every $x_1\in(0,L)$ and $M\in \R^{2\times 2}_\text{sym}$ let
$$
f(x_1,M):= \tcalQ_2(x_1,M-\II^0(x_1)) + \alpha_\calQ^+(x_1) (\det M)^+ +  \alpha_\calQ^-(x_1) (\det M)^-.
$$
Minimizing $J$ is equivalent to solving the pointwise minimization problem
\begin{equation}\label{pw-min}
\min_{M\in \R^{2\times 2}_\text{sym}} f(x_1,M)
\end{equation}
for every $x_1\in(0,L)$. Since $f(x_1,\cdot)$ is continuous and coercive, such a minimum exists.

Let $x_1\in(0,L)$ and let $M_*$ be a minimizer of \eqref{pw-min}. We first show that $M_*$ satisfies the bound 
\begin{equation}\label{Jbounds2r}
\Abs{M_* - \II_0(x_1)} \ge c |\kappa_1(x_1)| 
\end{equation}
for a constant $c>0$ independent of $x_1$, thus proving \eqref{Jbounds2}. 

Assume $\det M_*>0$. Writing $\tcalQ_2(x_1,M)=\frac12\tcalL_2(x_1)M\cdot M$ and differentiating $f$ with respect to $M$, we obtain that $M_*$ satisfies the condition
$$
\tcalL_2(x_1)(M_*-\II^0(x_1)) + \alpha_\calQ^+(x_1) \cof M_*=0,
$$
hence for a uniform constant $c>0$ 
$$
|M_*|=|\cof M_*|\leq c |M_*-\II^0(x_1)|
$$
and by triangle inequality
$$
|\kappa_1(x_1)|\leq |\II^0(x_1)|\leq (c+1) |M_*-\II^0(x_1)|.
$$
The same conclusion holds if $\det M_*<0$, since in this case $M_*$ satisfies
$$
\tcalL_2(x_1)(M_*-\II^0(x_1)) - \alpha_\calQ^-(x_1) \cof M_*=0.
$$
Assume now $\det M_*=0$. If $|M_*-\II^0(x_1)|\geq |\kappa_2(x_1)|$, then \eqref{Jbounds2r} is trivally satisfied.
If $|M_*-\II^0(x_1)|< |\kappa_2(x_1)|$, the formula for $2\times 2$ symmetric matrices $\det(A+B) = \det A + \cof A\cdot B+ \det B$ gives
$$
0= \det M_* = \det\II^0(x_1)+\cof \II^0(x_1)\cdot (M_*-\II^0(x_1))+\det(M_*-\II^0(x_1)),
$$
so that
\begin{align*}
|\kappa_1(x_1)\kappa_2(x_1)| = |\det\II^0(x_1)| & \leq |\II^0(x_1)|\,|M_*-\II^0(x_1)|+|M_*-\II^0(x_1)|^2,
\\
& \leq \left(\sqrt{2}+1\right)|\kappa_2(x_1)| \,|M_*-\II^0(x_1)|,
\end{align*}
which implies \eqref{Jbounds2r}.

By \eqref{Q-coerc2} and \eqref{Jbounds2r} we immediately deduce that for some constant $c>0$
$$
\min_{M\in \R^{2\times 2}_\text{sym}} f(x_1,M)\geq c |\kappa_1(x_1)|^2
$$
for every $x_1\in(0,L)$. This proves the lower bound in \eqref{Jbounds}.

For the upper bound, we choose $M =P^T\operatorname{diag}(0,\kappa_2(x_1))P$, where $P\in SO(2)$ is such that $\II^0(x_1) = P^T\operatorname{diag}(\kappa_1(x_1),\kappa_2(x_1))P$. By \eqref{Q-coerc2} we have that
$f(x_1,M)\leq c_2 |\kappa_1(x_1)|^2$. This concludes the proof.
\end{proof}

%%%%%%
\subsection{Codazzi incompatibility}\label{sec:wide_codazzi}
In this section we consider the limit of $\frac{1}{w^2} \tE_w$, where $\tE_w$ is the plate energy obtained in \eqref{eq:energy_plate}, as $w\to0$.
It immediately follows from Proposition~\ref{prop:Jbehavior} that this limit is not infinite only if the ribbon is Gauss-compatible, i.e., $\det \II^0\equiv 0$, which we assume throughout this section.
To simplify the notation in this section, we shall write
\[
\II = \brk{\begin{matrix} \bl & \bm \\ \bm&\bn \end{matrix}}, \quad \II^0:= \II|_{z_2=0} = \brk{\begin{matrix} \bl_0 & \bm_0 \\ \bm_0&\bn_0 \end{matrix}}, \quad \II^1 := \pl_2\II|_{z_2=0} = \brk{\begin{matrix} \bl_1 & \bm_1 \\ \bm_1&\bn_1 \end{matrix}}.
\]
In this notation, the Codazzi equation \eqref{eq:Codazzi_def} reads
\beq\label{eq:Codazzi_def_wide}
\DC = \brk{\begin{matrix} \DC_1 \smallskip\\ \DC_2 \end{matrix}} = \brk{\begin{matrix} \bl_1 -\bm_0' + \kappa(\bl_0+\bn_0) \\ \bm_1 -\bn_0' - \kappa \bm_0\end{matrix}}.
\eeq 
Since we assume that $\DG = \det \II^0 = 0$, we have that
\[
\det \II = z_2 (\bl_1 \bn_0 + \bn_1 \bl_0 - 2\bm_0\bm_1) + O(z_2^2). 
\]
Thus we define the \Emph{First order Gauss-deficit} (the zeroth order is $\DG$)
\beq\label{eq:Gauss_def_first_order}
\delta^{G,1} := \bl_1 \bn_0 + \bn_1 \bl_0 - 2\bm_0\bm_1.
\eeq

%%%%%
We start with the following compactness result.

\begin{proposition}[Low-energy wide ribbons, compactness]\label{prop:basic}
Let $(\f_w)$ be a sequence of rescaled isometric immersions in $W^{2,2}(\WW;\R^3)$ such that 
\[
\tE_w(\f_w) \to 0.
\]
Then $\det \II^0\equiv 0$ and, modulo a subsequence and translations, we have
\[
\f_w \conv{W^{2,2}} \f, \quad \nabla_w \f_w \conv{W^{1,2}} (d_1,d_2), \quad \II^w_{\f_w} \conv{L^2} \II^0,
\]
where $\f\in W^{2,2}((0,L);\R^3)$ and $d=(d_1,d_2,d_3)\in L^2((0,L);\SO(3))$ with $d_1 = y'$ and $d_1'\cdot d_2=\kappa$.
\end{proposition}

\begin{proof}
The fact that $\det \II^0\equiv 0$ follows from Theorems~\ref{thm:compactness_wide_ribbons}
and~\ref{thm:Gconv_wide_ribbons} and from Proposition~\ref{prop:Jbehavior}.
Since $\tE_w(\f_w)\to 0$ and $\II^w\to \II^0$ uniformly, we obtain from \eqref{Q-coerc} that $\II^w_{\f_w}\to \II^0$ strongly in $L^2$.
From the isometry constraint it follows that $\nabla_w \f_w$ is bounded in $L^\infty$ and
\beq\label{eq:nabla2ye-11}
\pl_1^2 \f_w =\frac{1}{2\a^w_{11}}\pl_1 \a^w_{11}\pl_1\f_w -\frac12\frac1w\pl_2 \a^w_{11} \frac1w \pl_2\f_w +(\II^w_{\f_w})_{11}\nu_{\f_w},
\eeq
\beq\label{eq:nabla2ye-12}
\pl_{12}^2 \f_w =\frac{1}{2\a^w_{11}}\frac1w\pl_2\a^w_{11}\pl_1\f_w +(\II^w_{\f_w})_{12}\nu_{\f_w},
\eeq
\beq\label{eq:nabla2ye}
\pl_2^2 \f_w = (\II^w_{\f_w})_{22}\nu_{\f_w},
\eeq
where $\a^w(x_1,x_2) = \a(x_1,wx_2)$.
Since $\a$ is bounded in $C^1$, this implies that $\nabla_w^2 \f_w$ is bounded in $L^2$. 
Therefore, $\nabla_w \f_w\wconv{} (d_1,d_2)$, both weakly in $W^{1,2}$ and weakly$^*$ in $L^\infty$, which in turn implies that $\nu_{\f_w}\to d_3$ strongly in $L^p$ for any $p<\infty$ (the fact that $d\in \SO(3)$ and $d_1'\cdot d_2=\kappa$ follows from Theorem~\ref{thm:compactness_wide_ribbons}).
In particular, by \eqref{eq:nabla2ye-11}--\eqref{eq:nabla2ye} and dominated convergence we deduce that $\nabla_w^2 \f_w$ converges strongly in $L^2$, which in turn implies the other convergences.
\end{proof}

\begin{remark}
Proposition~\ref{prop:basic} holds for any metric (not necessarily flat) with condition $\det \II^0\equiv 0$ replaced by $\DG\equiv 0$. 
This follows from the fact that $f_w$ satisfy the Gauss equation (Proposition~\ref{prop:W22_Gauss_Codazzi}), 
hence $\DG\equiv 0$ is a consequence of the strong $L^2$ convergence of the second forms.
\end{remark}
\medskip

We now state and prove the main compactness and $\Gamma$-convergence results of this section.
%%%%%%%%%
\begin{theorem}[Codazzi scaling: wide ribbons, compactness]\label{prop:Bstructure}
Assume $\det \II^0\equiv 0$ and \eqref{eq:assumption_II11}, that is, $\bl_0(x_1) \ne 0$ for any $x_1\in [0,L]$.
Let $(\f_w)$ be a sequence of rescaled isometric immersions in $W^{2,2}(\WW;\R^3)$ such that 
\beq\label{eq:w2scaling}
\tE_w(\f_w) \le Cw^2,
\eeq
and let 
\[
B_{\f_w} =\frac{1}{w}\brk{ \II^w_{\f_w}-\II^w}.
\]
Then, modulo a subsequence,
\beq\label{eq:defBstructure}
B_{\f_w} \wconv{L^2} B = B^0+x_2 B^1,
\eeq
where
\beq\label{eq:defB0}
B^0=\mat{\dL  & \dM \\ \dM & \frac{1}{\bl_0}(2\bm_0\dM - \bn_0 \dL) }
\eeq
for some $\dL,\dM\in L^2(0,L)$, 
\beq\label{eq:defB1}
B^1=-\mat{\DC_1 & \DC_2 \smallskip \\ \DC_2 & \frac{1}{\bl_0}( \delta^{G,1} - \bn_0\DC_1  + 2\bm_0 \DC_2)  },
\eeq
and $\DC$ and $\delta^{G,1}$ are the geometric deficits as in \eqref{eq:Codazzi_def_wide}--\eqref{eq:Gauss_def_first_order}.
\end{theorem}
%%%%%%%

%%%%%%%%%%%%
\begin{theorem}[Codazzi scaling: wide ribbons, $\Gamma$-limit]
\label{thm:Gconv_wide_Codazzi}
\emph{Lower bound:} Under the assumptions of Theorem~\ref{prop:Bstructure}, we have
\[
\liminf_{w\to 0} \frac{1}{w^2} \tE_w(\f_w) \ge \mathcal I(\dL,\dM) :=  
	\frac1{12}\dashint_0^L \tcalQ_2\brk{x_1,B^0}\,dx_1 
	 + \frac{1}{144}\dashint_0^L \tcalQ_2\brk{x_1, B^1}\,dx_1,
\]
where $B^0$ and $B^1$ are defined as in \eqref{eq:defB0} and \eqref{eq:defB1}.

\emph{Recovery sequence:} For any $\dL, \dM\in L^2(0,L)$, there exist rescaled isometric immersions $\f_w\in W^{2,2}(\WW;\R^3)$ such that $B_{\f_w} \wconv{L^2} B(\alpha,\beta)$ as defined in \eqref{eq:defBstructure}--\eqref{eq:defB1}, and 
\[
\lim_{w\to 0}\frac{1}{w^2} \tE_w(\f_w)= \mathcal I(\dL,\dM) .
\]
\end{theorem}
%%%%%%%%%%%

These results, the structure \eqref{eq:defB1} of $B^1$ and the bounds \eqref{Q-coerc} on $\tcalQ_2$, immediately implies the following.

%%%%%%%%%%%%
\begin{corollary}\label{cor:liminf}
Assume that $\det \II^0\equiv 0$ and \eqref{eq:assumption_II11} hold. Then, as $w\to 0$,
\beq\label{eq:energy_scaling_wide_ribbon}
 \inf  \frac{1}{w^2}\tE_w \sim \int_0^L |\DC(x_1)|^2+ |\delta^{G,1}(x_1) - \bn_0(x_1)\DC_1(x_1)  + 2\bm_0(x_1) \DC_2(x_1)|^2\,dx_1,
\eeq
where $\DC$ and $\delta^{G,1}$ are the geometric deficits as in \eqref{eq:Codazzi_def_wide}--\eqref{eq:Gauss_def_first_order}.
If $(\f_w)$ is a sequence of approximate minimizers of $\tE_w$ such that
\beq\label{appr-min}
\lim_{w\to0}  \frac{1}{w^2}\Big(\tE_w(f_w) -   \inf \tE_w \Big)=0,
\eeq
then $|\II^w_{\f_w} - \II^0| = O_{L^2}(w)$ and
\[
\frac{1}{w}\brk{\int_{-1/2}^{1/2} \II_{\f_w}^w \,dx_2 - \II^0} \conv{L^2} 0.
\] 
In particular, this implies part 2(a) in Theorem~\ref{thm:wide_informal} and, together with Corollary~\ref{cor:scaling_narrow}, the scaling in part~2 of Corollary~\ref{cor:informal}.
\end{corollary}

%%%%%%%
\begin{proof1}{of Theorem~\ref{prop:Bstructure}:}
The fact that $B_{\f_w}$ is bounded in $L^2$ follows immediately from \eqref{Q-coerc} and \eqref{eq:w2scaling}.
Thus, modulo a subsequence, $B_{\f_w}$ converges to $B$ weakly in $L^2$ for some $B\in L^2(\WW;\R^{2\times 2}_\text{sym})$.

We now deduce the structure of $B$ from the Gauss-Codazzi compatibility equations of the isometric immersions $\phi_w(z_1,z_2) = \f_w(z_1,z_2/w)$ associated with 
$\f_w$ (Proposition~\ref{prop:W22_Gauss_Codazzi}).
The rescaled second forms $\II^w_{\f_w}$ satisfy the rescaled Gauss-Codazzi equations.
The Codazzi equations hold in $W^{-1,2}$ and read as follows:
\begin{eqnarray}
(1-\kappa wx_2) \Bigg(\frac{1}{w}\pl_2(\II^w_{\f_w})_{11} -\pl_1(\II^w_{\f_w})_{12} \Bigg) &\!\!\! =\!\!\!&
-\kappa (\II^w_{\f_w})_{11} +\kappa' wx_2(\II^w_{\f_w})_{12} -\kappa(1-\kappa wx_2)^2(\II^w_{\f_w})_{22},
\label{GC1}
\\
(1-\kappa wx_2) \Bigg(\frac{1}{w}\pl_2(\II^w_{\f_w})_{12} -\pl_1(\II^w_{\f_w})_{22} \Bigg) &\!\!\! = \!\!\! &
\kappa (\II^w_{\f_w})_{12}.
\label{GC2}
\end{eqnarray}
By Proposition~\ref{prop:basic} the righthand side of the first equation converges to $-\kappa(\bl_0+\bn_0)$ strongly in $L^2$, whereas 
$\pl_1(\II^w_{\f_w})_{21}$ converges to $\bm_0'$ strongly in $W^{-1,2}$. Since $\frac{1}{w}\pl_2(\II^w_{\f_w})_{11}$ converges to $\pl_2 B_{11} + \bl_1$ weakly in $W^{-1,2}$ 
by the weak convergence of $B_{\f_w}$ to $B$ in $L^2$, we obtain that
$$
\pl_2 B_{11}= \bm_0' -\bl_1-\kappa(\bl_0+\bn_0)=-\DC_1.
$$
Similarly, passing to the limit in \eqref{GC2} yields
$$
\pl_2 B_{12}= \bn_0'- \bm_1 + \kappa \bm_0 = -\DC_2.
$$
Therefore we deduce that
\[
B(x_1,x_2) = \bar{B}(x_1)  +  \mat{ -\DC_1 x_2 & -\DC_2 x_2 \smallskip\\ -\DC_2 x_2 & x_2 \sigma(x_1) + \gamma(x_1,x_2)},
\]
for some functions $\bar{B}\in L^2((0,L);\R^{2\times 2}_\text{sym})$, $\sigma\in L^2((0,L))$, and $\gamma\in L^2(\WW)$ with 
\begin{equation}\label{gamma}
\int_{-1/2}^{1/2} \, \gamma\,dx_2 = \int_{-1/2}^{1/2} x_2 \, \gamma\,dx_2 =0.
\end{equation}

We now write
\[
\II^w_{\f_w} = w B_{\f_w} + \II^0 + wx_2 \II^1 +  O_{L^\infty}(w^2) 
\]
and consider the Gauss equation
\[
\begin{split}
0 &= \det \II^w_{\f_w} =  w \tr \II^0 \tr (x_2 \II^1   + B_{\f_w}) - w \tr(\II^0 (x_2\II^1 + B_{\f_w})) + O_{L^1}(w^2) \\
	&=w \delta^{G,1}x_2 + w \tr \II^0 \tr B_{\f_w} - w \tr(\II^0  B_{\f_w}) + O_{L^1}(w^2), 
\end{split}
\]
where we used the fact that $\det \II^0= 0$, the formula for $2\times 2$ matrices $\det(A+B) = \det A + \det B + \tr A \tr B - \tr AB$, and
the definition~\eqref{eq:Gauss_def_first_order}.
Dividing by $w$ and taking the limit as $w\to 0$, we obtain
\[
0 = \bn_0\bar{B}_{11} - 2 \bm_0 \bar{B}_{12} + \bl_0\bar{B}_{22} + \brk{\delta^{G,1}  - \bn_0\DC_1 + \bl_0\sigma +2\bm_0 \DC_2}x_2 + \bl_0 \gamma.
\]
By \eqref{gamma} and orthogonality in the $L^2$ sense, we deduce that
$$
\bn_0\bar{B}_{11} - 2 \bm_0 \bar{B}_{12} + \bl_0\bar{B}_{22}=  \delta^{G,1}  - \bn_0\DC_1 + \bl_0\sigma +2\bm_0 \DC_2=\bl_0\gamma=0,
$$
which gives the wanted structure of $B$ using the assumptions $\det \II^0= 0$ and $\bl_0 \ne 0$.
\end{proof1}\\

%%%%%%%%%%%%

%%%%%%%%%
\begin{proof1}{of Theorem~\ref{thm:Gconv_wide_Codazzi}:}
The lower bound follows immediately from \eqref{lbound-vw}, the convexity of $\tcalQ_2$, and Theorem~\ref{prop:Bstructure}.
We now construct the recovery sequence.
By a density argument it is sufficient to prove the result for smooth $\dL$ and $\dM$.

\paragraph{Step 1: notation.}
Let $\dL, \dM\in C^\infty([0,L])$.
Since $\bl_0$ is smooth and never zero, there exists $\bar w>0$ with the following property: for every $w\in(0,\bar w)$ there is a unique smooth function $\gamma_w=\gamma_w(x_1)$
such that the matrix
$$
A_w:= \II^0 + w\mat{\dL& \dM \\ \dM &  \gamma_w } 
$$
has determinant equal to zero on $[0,L]$. Note that $\gamma_w$ converges to $(2\bm_0\dM - \bn_0 \dL)/\bl_0$ in $C^1([0,L])$, as $w\to0$.
Furthermore, we can fix some $\e_0>0$ and extend $\ell$, $\II^0$, $\dL$, $\dM$, $\gamma_w$, and $\kappa$ smoothly to $\hat\ell:= [-\e_0, L + \e_0]$ such that all these properties (including $\bl_0\ne 0$) hold on this slightly larger midline.
We denote by $\hSw := \hat\ell\times (-w/2,w/2)$ the slightly enlarged ribbon domain.

Since $0$ is an eigenvalue of multiplicity one for $\II^0$, for every $w\in(0,\bar w)$ we can find a smooth function $q_w: [-\e_0, L + \e_0]\to \R^2$ such that $|q_w|=1$, $A_wq_w=0$ on $[-\e_0, L + \e_0]$, and $q_w$ converge to $q_0$ in $C^1([-\e_0, L + \e_0])$, as $w\to0$, where $|q_0|=1$ and $\II^0q_0=0$ on $[-\e_0, L + \e_0]$.
We should think of unit length $q_w$ as vector fields in $T\hSw|_{\hat\ell}$.
Note that the assumption \eqref{eq:assumption_II11}, that $\bl_0(x_1)\neq0$ for every $x_1\in[-\e_0, L + \e_0]$, implies that $q_0$ is never parallel to $e_1=\frac{\pl}{\pl z_1}$. In particular, $q_w$ is never parallel to $e_1$ for $w\in[0,\bar w)$, if $\bar w$ is small enough.

\paragraph{Step 2: the main idea.}
Our goal is to construct maps $v_w:\Sw\to \R^3$, which are isometric immersions of $\a$ and satisfy $\II_{v_w}|_\ell = A_w$.
Their rescaled version $\f_w$ will then provide the recovery sequence.
Let 
\[
r_w := (\pl_1 v_w,\pl_2 v_w, \pl_1v_w \times \pl_2 v_w) : (0,L) \to \SO(3)
\]
 be the Darboux frame of the surface $v_w$ (it is indeed a rotation since $\nabla v_w^T\nabla v_w|_\ell = I$). 
Denote by $\hat{r}_w$ its restriction to the first two columns. Geometrically, we view $\hat{r}_w$ as a map $\hat{r}_w :T\Sw|_{\ell} \to \R^3$.
Arguing as in Proposition~\ref{prop:Phi}, $r_w$ satisfies the Darboux equations
\beq\label{eq:r_w}
r_w'= r_w \mat{0  & -\kappa & -(A_w)_{11} \\ \kappa & 0 & -(A_w)_{12} \\ (A_w)_{11} & (A_w)_{12} & 0}. 
\eeq

Since $\a$ is a flat metric, any (smooth enough) isometric immersion of it is a ruled surface; our aim is to construct such a ruled surface
\beq
\label{eq:ruled_surface}
(t,s)\mapsto \int_0^t v_1(\tau)\,d\tau + s v_2(t)
\eeq
for some unit vectors $v_1,v_2$, such that the directrix (the first addend) is the midline (note that the coordinates $(t,s)$ are not the coordinates $(z_1,z_2)$).
From basic theory of ruled surfaces, a ruled surface is flat if and only if $\dot{v}_2 \cdot (v_1 \times v_2) = 0$, in which case $\frac{\pl}{\pl s}$ is the zero eigenvector of the second fundamental form (in the coordinates $(t,s)$).
Thus, as $v_1$ is the tangent direction of the midline, and $v_2$ should be the image of the unit null-vector of the second form, we have
\[
v_1 = \hat{r}_w e_1, \qquad v_2 = \hat{r}_w q_w.
\]
However, we still need to obtain the correct coordinates $(t,s)$.
Note that this is indeed a surface because $q_w$ is not parallel to $e_1$, which follows from \eqref{eq:assumption_II11}.

To this end, recall that we already constructed an isometric immersion $\chi:\Sw\to \R^2$ in \eqref{eq:def_chi}, as a ruled surface. 
Since $\tSw = \chi(\Sw)$ is a domain in $\R^2$, we can express it as a ruled surface (with the same midline $\tilde{\ell}$) in many ways.
The image of the null-direction $q_w$ in $\tSw$ is $\Rot q_w$, and we can choose the ruling in the direction $\Rot q_w$.
This is given by the map
\beq\label{eq:def_phi_w}
\phi_w(t,s) = \int_0^{t} \Rot(\tau)e_1\,d\tau+s\Rot(t) q_w(t).
\eeq
We will show below that for some appropriate domain of $(t,s)$ this map is indeed a coordinate map on $\tSw$ (in order to obtain this we needed to extend the various fields to $[-\e_0,L+\e_0]$).
With these coordinates, the map \eqref{eq:ruled_surface} maps the midline to the wanted midline, and the ruling in $\tSw$ of null-direction of the second fundamental form to the wanted ruling direction.

The isometric immersion $v_w:\Sw\to \R^3$ is thus given by 
\beq\label{eq:def_v_w}
v_w=k_w\circ \chi,
\eeq
where $k_w: \tSw\to \R^3$ maps the $Pq_w$-ruling into the ruling of the image, that is,
\beq\label{eq:def_k_w}
k_w\circ \phi_w(t,s) = \int_0^t \hat{r}_w (\tau)e_1\,d\tau + s \hat{r}_w(t) q_w(t).
\eeq
Below we show that this construction is well-defined and provides indeed a recovery sequence.

%%%
\paragraph{Step 3: The coordinate change.}
We first show that we can define coordinates on $\tSw$ via the map \eqref{eq:def_phi_w}.
Indeed, let $\eta>0$ and define $\phi_w:[-\e_0,L+\e_0]\times[-\eta/2,\eta/2]\to\R^2$ via \eqref{eq:def_phi_w}.

Note that, by \eqref{eq:R_def} we have
\beq\label{eq:nabla_phi}
\nabla \phi_w(t,s) = \Rot(t)\brk{(e_1 + sq_w'(t)-s\kappa q_w^{\perp})\otimes e_1 + q_w(t) \otimes e_2},
\eeq
where $q_w^\perp = \brk{\begin{matrix} 0 & 1 \\ -1 & 0 \end{matrix}} q_w$.
Our aim is to show that for a suitable choice of $\eta$ and $\bar w$ the function $\phi_w$ is invertible and its image $U_w:=\phi_w([-\e_0,L+\e_0]\times[-\eta/2,\eta/2])$ contains $\tS_w$ for every $w\in(0,\bar w)$. 
Arguing as in \cite[Lemma~12]{FHMP16b}, one can prove that for $\eta$ small enough the function $\phi_0$ is invertible on $[-\e_0,L+\e_0]\times[-\eta/2,\eta/2]$
and its inverse $\phi_0^{-1}$ is Lipschitz continuous on $\phi_0([-\e_0,L+\e_0]\times[-\eta/2,\eta/2])$.
For every $(t,s), (t',s')\in [-\e_0,L+\e_0]\times[-\eta/2,\eta/2]$ and every $w\in(0,\bar w)$ we have
$$
\begin{array}{rcl}
|\phi_w(t,s)-\phi_w(t',s')| & \geq &  |\phi_0(t,s)-\phi_0(t',s')| - |s(q_w(t)-q_0(t))-s'(q_w(t')-q_0(t'))|
\smallskip\\
& \geq & c_0 \big(|t-t'|+ |s-s'|\big) - \big|(s-s')(q_w(t)-q_0(t))\big| 
\smallskip\\
& & -\big|s' (q_w(t)-q_0(t)-q_w(t')+q_0(t'))\big|,
\end{array}
$$
where $c_0>0$  and we used the Lipschitz continuity of $\phi_0^{-1}$. The $C^1$ convergence of $q_w$ to $q_0$ guarantees that for $\bar w$ small
$$
|q_w(t)-q_0(t)|\leq \frac{c_0}4, \qquad |q_w(t)-q_0(t)-q_w(t')+q_0(t')|\leq \frac{c_0}{2\eta} |t-t'|
$$
for every $t,t'\in[-\e_0,L+\e_0]$ and $w\in(0,\bar w)$. Therefore, we deduce that
$$
|\phi_w(t,s)-\phi_w(t',s')| \geq \frac{c_0}4 \big(|t-t'|+ |s-s'|\big)
$$
for every $(t,s), (t',s')\in [-\e_0,L+\e_0]\times[-\eta/2,\eta/2]$, hence $\phi_w$ is injective on this set for every $w\in(0,\bar w)$.
Note also that
$$
\det \nabla\phi_w(t,s)=e_2\cdot q_w(t) +O_{L^\infty}(s).
$$
Using again the $C^1$ convergence of $q_w$ to $q_0$ and the fact that 
$q_0$ is never parallel to $e_1$, we deduce that $\nabla\phi_w$ is an invertible matrix at every point, for $\eta$ small enough.
 
To prove that $U_w \supset \tSw$, we recall that $\tSw$ is the image of the map
\[
(t,s) \mapsto \chi(t,s) = \int_0^{t} \Rot(\tau)e_1\,d\tau + s \Rot(t)e_2, \qquad (t,s)\in (0,L)\times (-w/2,w/2).
\]
Since $q_w$ is never parallel to $e_1$, there exists $\theta_0\in [0,\pi/2)$ such that $|q_w(t)\cdot e_2| \geq \cos \theta_0$ for all $t\in [-\e_0,L+\e_0]$ and all $w\in [0,\bar w)$.
Assume first that $\Rot$ is constant. Without loss of generality we can take $\Rot(t)=I$. 
Then, $\tSw = (0,L) \times (-w/2,w/2)$, and simple trigonometry shows that in this case $U_w \supset [0,L] \times [-w_0/2,w_0/2]$ where $w_0 = \min\{ \eta \cos \theta_0, \e_0 \cot \theta_0\}$, hence in particular $U_w \supset \tSw$ for all $w< w_0$.

For a general $\Rot(t)$, cover $[0,L]$ with a finite number of subintervals $[t_i - \delta, t_i + \delta]$ such that for each $i$ one has $|\Rot(t)- \Rot(t_i)|\leq \cos \theta_0/10$ for $t\in [t_i - 3\delta, t_i + 3\delta]$. 
Since $|q_w(t)\cdot e_2| \geq \frac9{10}\cos \theta_0$ for $t\in [t_i - 2\delta, t_i + 2\delta]$, a similar argument as above shows that the image of the map
\begin{equation}\label{eq:rest}
(t,s) \mapsto \int_0^{t} \Rot(\tau)e_1\,d\tau + s \Rot(t_i)e_2, \qquad (t,s)\in (t_i - 2\delta, t_i + 2\delta)\times (-w/2,w/2)
\end{equation}
is contained in the image of the restriction of $\phi_w$ to $[t_i - 3\delta, t_i + 3\delta]\times [-\eta/2,\eta/2]$.
On the other hand, the image of the map \eqref{eq:rest} contains the restriction of $\chi$ to the set $(t_i - \delta, t_i + \delta) \times (-aw/2,aw/2)$, 
where $a>0$ is some constant depending only on $\theta_0$ and $\delta$. Therefore, for a suitable choice of $\bar w$ we have $U_w \supset \tSw$ for every $w\in(0,\bar w)$.
Thus the map $v_w$ given by \eqref{eq:def_v_w} is well-defined.

%%%
\paragraph{Step 4: Proof that {\rm$v_w\in \Wiso(\Sw;\R^3)$} and {\rm$\II_{v_w}|_{\ell}=A_w$}.} 
A simple computation using \eqref{eq:r_w} shows that
\[
\hat{r}_w' q_w = -\kappa \hat{r}_w q_w^\perp + (e_1^T A_w q_w )r_w e_3=  -\kappa \hat{r}_w q_w^\perp.
\]
where in the last equality we used that $A_w q_w=0$.
We thus obtain, using \eqref{eq:def_k_w} and \eqref{eq:nabla_phi}, that
\[
\nabla k_w(\phi_w(t,s))\nabla\phi_w(t,s) = \hat{r}_w(t)P^T(t) \nabla \phi_w(t,s),
\]
and since $\nabla\phi_w$ is invertible, this implies that
\begin{equation}\label{vw_isom}
\nabla k_w(\phi_w(t,s))=\hat{r}_w(t)P^T(t),
\end{equation}
hence $(\nabla k_w)^T\nabla k_w=I$ and thus, $v_w$ is an isometric immersion of $\a$, that is, $v_w\in \Wiso(\Sw;\R^3)$.

From \eqref{vw_isom} it immediately follows that $\nu_{k_w}=\pl_1 k_w \wedge \pl_2 k_w=r_w e_3$.
By differentiating \eqref{vw_isom} with respect to $t$ and using \eqref{eq:R_def} and \eqref{eq:r_w} we deduce that at $t=0$ we have
$$
(\II_{k_w}\circ\chi)\Rot e_1 = (A_w)_{11}\Rot e_1+(A_w)_{12}\Rot e_2,
$$
hence
$$
\Rot ^T(\II_{k_w}\circ\chi)\Rot e_1=A_w e_1.
$$
Since $\Rot ^T(\II_{k_w}\circ\chi)\Rot$ and $A_w$ are both symmetric matrices with zero determinant and $(A_w)_{11}\neq0$, we conclude that
$$
\Rot ^T(\II_{k_w}\circ\chi)\Rot =A_w,
$$
hence $\II_{v_w}=A_w$ on $(0,L)\times\{0\}$. This concludes the proof of the step.

%%%
\paragraph{Step 5: Energy estimates.}
Let now $\f_w(x_1,x_2) := v_w(x_1, wx_2)$. Since $\f_w$ are smooth rescaled isometric immersions, the Gauss-Codazzi equations \eqref{GC1}--\eqref{GC2} are satisfied.
Passing into the limit in these equations and using that $\II^w_{\f_w}\to\II^0$ in $C^1$ we deduce that
$$
\frac{1}{w}\pl_2(\II^w_{\f_w})_{11} \to \bm_0'
-\kappa (\bl_0+\bn_0)
$$
and
$$
\frac{1}{w}\pl_2(\II^w_{\f_w})_{12}\to  \bn_0' -\kappa \bm_0
$$
uniformly. Moreover, by differentiating the Gauss equation we obtain
$$
(\II^w_{\f_w})_{11}\frac{1}{w}\pl_2(\II^w_{\f_w})_{22}=2(\II^w_{\f_w})_{12}\frac{1}{w}\pl_2(\II^w_{\f_w})_{12} -(\II^w_{\f_w})_{22}\frac{1}{w}\pl_2(\II^w_{\f_w})_{11}, 
$$
so that the above convergence properties yield
$$
\frac{1}{w}\pl_2(\II^w_{\f_w})_{22}\to \frac1{\bl_0}\big(-\bn_0\bm_0'+2\bm_0\bn_0'+\kappa\bn_0(\bl_0+\bn_0)-2\kappa\bm_0^2\big)
$$
uniformly. By Taylor expansion we deduce that
$$
\frac{1}{w}(\II^w_{\f_w} - \II^w) \to B
$$
uniformly.
By dominated convergence we obtain the required convergence of energies.
\end{proof1}

%%%%%%
\begin{proof1}{of Corollary~\ref{cor:liminf}:}
From Theorems~\ref{prop:Bstructure} and~\ref{thm:Gconv_wide_Codazzi}, and standard $\Gamma$-convergence arguments, it follows that
\[
\lim_{w\to 0}\Big(\inf  \frac{1}{w^2}\tE_w \Big) = \min \mathcal I = \mathcal I(0,0) = \frac{1}{144}\dashint_0^L \tcalQ_2\brk{x_1, B^1}\,dx_1,
\]
where we used the fact that $\mathcal I$ is obviously minimized at $\alpha=\beta=0$.
Now, \eqref{eq:energy_scaling_wide_ribbon} immediately follows from the structure of $B^1$ and from \eqref{Q-coerc}.
This implies that \eqref{eq:w2scaling} holds for approximate minimizers $\f_w$, hence by Theorem~\ref{prop:Bstructure} we obtain that $|\II^w_{\f_w} - \II^0| = O_{L^2}(w)$.
Furthermore, these minimizers satisfy $B_{\f_w}\wconv{} B(0,0) = x_2 B^1(x_1)$. On the other hand,
by \eqref{lbound-vw} and \eqref{appr-min} we have
$$
\lim_{w\to0} \frac{1}{12}\dashint_{(0,L)\times (-1/2,1/2)} \tcalQ_2(x_1,B_{\f_w} )\, dx'=
\lim_{w\to0}  \frac{1}{w^2}\tE_w(f_w) = \lim_{w\to 0}\Big(\inf  \frac{1}{w^2}\tE_w \Big)= \frac{1}{144}\dashint_0^L \tcalQ_2\brk{x_1, B^1}\,dx_1.
$$
Since $\tcalQ_2$ is a coercive quadratic form in its second argument, this implies that $(B_{\f_w})$ converges strongly in $L^2$, hence 
\[
\frac{1}{w}\brk{\int_{-1/2}^{1/2} \II^w_{\f_w} \,dx_2 - \II^0} = \frac{1}{w}\brk{\int_{-1/2}^{1/2} (\II^w_{\f_w} -\II^w) \,dx_2} + O_{L^\infty}(w) = \int_{-1/2}^{1/2} B_{\f_w} \,dx_2 + O_{L^\infty}(w) \conv{L^2} 0,
\]
which completes the proof.
\end{proof1}

%%%%%%
\subsection{Geometries in which condition \eqref{eq:assumption_II11} fails}\label{sec:wide_examples}

We now consider two examples of a Gauss-compatible, Codazzi-incompatible ribbon, in which $\bl_0=0$ (i.e., \eqref{eq:assumption_II11} fails).
In the first one, which corresponds to Figure~\ref{fig:ribbons}(d), we show that the energy scaling of $\tE_w$ is indeed different than in the $\bl_0 \ne 0$ case considered in \S\ref{sec:wide_codazzi}.
In the second one, which corresponds to Figure~\ref{fig:ribbons}(b), we believe this is also true, and we present a very partial result that suggests it.  
In both examples, for completeness, we present the best known ansatzes in the wide and narrow regimes, and the resulting conjectured energy scalings.
These arguments essentially appeared in \cite{LSSM21}, where the ansatzes were shown to match the observed minimizers; here we simply represent them in a more explicit and mathematically-careful way. 

%%%%%
\subsubsection{The geometry of Figure~\ref{fig:ribbons}(d)}
Consider a ribbon with reference first and second fundamental forms
\beq\label{eq:II_L_0}
\a(z_1,z_2) = \mat{(1-\kappa z_2)^2 & 0 \\ 0 & 1}, \qquad \II=\II^0 = \mat{0 & 0 \\ 0 & \bn}
\eeq
for some constants $\kappa,\bn\ne 0$.
This is the geometry depicted in Figure~\ref{fig:ribbons}(d).
This ribbon is Codazzi-incompatible:
Gauss equation is obviously satisfied, however it is immediate to check that $\DC_1 = \kappa \bn \ne 0$, hence the ribbon is Codazzi incompatible, yet condition \eqref{eq:assumption_II11} is not satisfied.

The following proposition shows that in the wide ribbon regime, the energy of this geometry satisfies $\inf \E_{t,w} \gg t^2 w^{1+\e}$ for every $\e>0$ (meaning that there is a transition in the scaling compared to the narrow ribbon scaling of $t^2 w^2$), and that approximate minimizers converge to $\II^0$ at a slower rate compared to those of Codazzi-incompatible ribbons that satisfy \eqref{eq:assumption_II11}.
This proves part 2(b) of Theorem~\ref{thm:wide_informal}, and completes the proof of part 2 of Corollary~\ref{cor:informal}.

%%%%%%%%%%
\begin{proposition}\label{prop:ribbon_d}
For $\a$ and $\II$ as in \eqref{eq:II_L_0} we have
\[
\lim_{w\to0}\Big(\inf \frac{1}{w^{1+\e}} \tE_w \Big)= \infty
\]
for any $\e>0$.
Moreover, for any sequence of rescaled isometries $\f_w$, we have that
\[
\|\II^w_{\f_w} - \II^0\|_{L^2} \gg w^{(1+\e)/2}
\]
for any $\e>0$.
\end{proposition}

\begin{proof}
Let $\e>0$. To simplify notation, we write $\beta = (1+\e)/2$. 
Assume by contradiction that there is a sequence of rescaled isometries $\f_w$ such that
\[
\tE_w(\f_w) \le Cw^{2\beta}.
\]
Note that this is equivalent to $\|\II^w_{\f_w} - \II^w\|_{L^2}\le C w^\beta$, and thus also to $\|\II^w_{\f_w} - \II^0\|_{L^2}\le C w^\beta$.
Writing
\[
\II^w_{\f_w} = \II^w + w^{\beta} B^w,
\]
we have that, modulo a subsequence, $B^w$ converges to some $B$ weakly in $L^2$. Moreover, by Proposition~\ref{prop:W22_Gauss_Codazzi} the rescaled second forms $\II^w_{\f_w}$ satisfy the rescaled Gauss-Codazzi equations \eqref{GC1}--\eqref{GC2}. 
In particular, equation \eqref{GC1} reads
\[
(1-\kappa wx_2)\brk{w^{\beta-1}\pl_2 B^w_{11} - w^\beta\pl_1 B^w_{12}} = -\kappa w^\beta B^w_{11}  - \kappa (1-\kappa wx_2)^2  (w^\beta B^w_{22} + \bn).
\]
Passing to the limit as $w\to 0$, we obtain that
$w^{\beta-1}\pl_2 B^w_{11}$ converges to $-\kappa \bn$  strongly in $W^{-1,2}$.

Now, the Gauss equation reads
\[
0 = \det \II^w_{\f_w} = w^{\beta}\bn B^w_{11} + w^{2\beta} \det B^w,
\]
and thus
\[
0  = w^{\beta-1} \pl_2 B^w_{11} + w^{2\beta-1} \bn^{-1} \pl_2 \det B^w.
\]
Since $2\beta-1>0$, and $\det B^w$ is bounded in $L^1$, we have that the second addend in the righthand side tends to zero in $W^{-1,\infty}$.
Thus, taking the limit as $w\to 0$, we obtain that $w^{\beta-1} \pl_2 B_{11}$ tends to $0$ in $W^{-1,\infty}$, in contradiction.
\end{proof}
\medskip

%%%%%%

Proposition~\ref{prop:ribbon_d} provides a lower bound to the energy of wide ribbons of this geometry. However, the optimal energy scaling is currently unclear.
The best known ansatz, which seems to agree well with the experiments, at least if the ribbon is not too long (see \cite[Fig.~7]{LSSM21}), is the following:
For any $\delta>0$, the (rescaled) second forms
\[
\tilde{\II}_w = \delta \bn \mat{ \brk{(\delta^2+1)(wx_2 \kappa -1)^2 - 1}^{1/2} & (1- \kappa wx_2)^{-1} \\ (1- \kappa wx_2)^{-1} & (1- \kappa wx_2)^{-2}\brk{(\delta^2+1)(wx_2 \kappa -1)^2 - 1}^{-1/2} }
\]
satisfy the (rescaled) Gauss--Codazzi system (with respect to $\a$ in \eqref{eq:II_L_0}). 
Thus, $\tilde{\II}_w$ corresponds to some rescaled isometric immersion $\f_w$, as long as it is well-defined for $x_2\in (-1/2,1/2)$, that is, if $\delta > (\kappa w)^{1/2} \frac{(4-\kappa w)^{1/2}}{2-\kappa w}$.
Taking $\delta \approx (\kappa w)^{1/3}$ seems to be the optimal choice, and yields energy
\[
\tE_w(\f_w) \sim \kappa^{2/3}\bn^{2} w^{2/3}.
\]
This suggests that for wide ribbons we have 
\[
\inf \E_{t,w} \lesssim t^2 w^{2/3}.
\]
On the other hand, we know that $\Psi_t$ gives the correct energy scaling for narrow ribbons ($w^2\ll t$).
A simple computation shows that
\[
\E_{t,w} (\Psi_t) \sim \max\{w^6, w^2t^2\}.
\]
This suggests that the shape transition occurs when $t^2 w^{2/3} \sim w^6$, that is, when $t \sim w^{8/3}$, which is slightly better than the $t \sim w^3$ obtained by asymptotic analysis (the experiments and numerics \cite[\S3.4]{LSSM21} cannot differentiate between such close exponents).

%%%%%
\subsubsection{The geometry of Figure~\ref{fig:ribbons}(b)}
Consider a ribbon with reference first and second fundamental forms
\beq\label{eq:ribbon_b}
\a = \mat{1 & 0 \\ 0 & 1}, \qquad \II=\II^0 = \mat{0 & 0 \\ 0 & x_1}.
\eeq
This is the geometry depicted in Figure~\ref{fig:ribbons}(b).
Again, it is immediate to see that $\DG=0$ and $\DC_2 = -1 \ne 0$, hence the ribbon is Codazzi incompatible, yet condition \eqref{eq:assumption_II11} is not satisfied.

In this case we do not have rigorous results showing that $\inf \tE_\w \gg w^2$, yet experiments suggest this is the case, and that $\inf \tE_w \sim w$ (or $\inf \E_{t,w} \sim t^2w$ for wide enough ribbons), with an optimal shape that matches the following ansatz \cite[Fig.~3]{LSSM21}:
The following (rescaled) second fundamental forms satisfy the Gauss-Codazzi system with respect to $\a = I$:
\[
\tilde{\II}_w = (x_1 + \sqrt{w}x_2)\mat{w & \sqrt{w} \\ \sqrt{w} & 1}
\]
Thus, $\tilde{\II}_w$ corresponds to some rescaled isometric immersion $\f_w$, whose energy satisfies 
\[
\tE_w(\f_\w) \sim w.
\]
On the other hand, for narrow ribbons, simple computation shows that
\[
\E_{t,w} (\Psi_t) \sim \max\{w^8, w^2t^2\}
\]
(the different scaling compared to the previous geometry is due to the fact that the metric \eqref{eq:DPsi_euc} in this case happens to have no $z_2^3$ terms).
Assuming these ansatzes give the correct scaling, we expect that a transition occurs at $t^2 w \sim w^8$, that is, at $t\sim w^{7/2}$ (which is slightly better than the $t \sim w^4$ obtained by asymptotic analysis \cite[\S3.2]{LSSM21}).

Recall that for Codazzi-incompatible ribbons that satisfy \eqref{eq:assumption_II11}, the recovery sequence in the $t^2w^2$ scaling (Theorem~\ref{thm:Gconv_wide_Codazzi}) is based on constructing isometric immersions for which the second fundamental form along the midline is $\II^0$ (or small perturbations thereof). 
These are solutions of the Gauss--Codazzi system, as a system of equations for the second fundamental form, given the metric $\a$, with initial value $\II^0$.
While, as stated, we do not have a rigorous proof that the geometry \eqref{eq:ribbon_b} has a different scaling, the following shows that such solutions to the Gauss--Codazzi system do not exist for this geometry, at least in the classical sense.

\begin{proposition}
The Gauss--Codazzi system for $\a=I$, which reads
\[
\det\mat{L & M \\ M & N} =0, \qquad 
\pl_2\mat{ L \\ M} - \pl_1\mat{M\\ N} =0
\]
with the initial data
\[
\at{\mat{L & M \\ M&N}}_{z_2=0} = \mat{0 & 0 \\ 0 & x_1}
\]
admits no $C^2$ solutions.
\end{proposition}

\begin{proof}
The Codazzi equations, combined with the initial data, imply that
\[
\at{\pl_2\mat{ L \\ M}}_{z_2=0} = \mat{0 \\ 1},
\]
and
\[
\pl_2^2 L|_{z_2=0} = \pl_1 (\pl_2 M|_{z_2=0}) = \pl_1 1 =0.
\]
On the other hand, differentiating the Gauss equation twice with respect to $z_2$ we obtain
\[
N{ \pl_2^2 L} + 2\pl_2 N {\pl_2 L} + {L}\pl_2^2 N -2{ M}\pl_2^2 M - 2(\pl_2M)^2 = 0.
\]
Letting $x_2 = 0$ and plugging in the above estimates, we obtain that the first four terms are all $0$, whereas the last term $- 2(\pl_2M)^2$ is equal to $-2$, which is a contradiction.
\end{proof}

\section{Appendix}\label{appendix}

In this short appendix we outline the proofs of Theorems~\ref{thm:compactness_plates} and~\ref{thm:Gconv_plates}.

\begin{proof1}{of Theorem~\ref{thm:compactness_plates}:}
Set $A_t:=(\nabla_t \Psi_t)^{-1} \tg_t^{-1/2}$ and note that $A_t$, $A_t^{-1}$, and the gradient of $A_t$ are all uniformly bounded with respect to $t$.
Writing 
$$
\E_{t,w}(u^t)=\dashint_{U} \calW(\nabla_{t} u^t \, A_t^{-1})\,dx,
$$
one can argue as in the proof of \cite[Lemma~2.3]{BLS16} (see also \cite[Theorem~2.3]{LP11}) and show the existence of $Q^t\in W^{1,2}(\WW;\R^{3\times 3})$ such that 
$$
\dashint_{U} |\nabla_t u^t -Q^t|^2\,dx \leq C(\E_{t,w}(u^t)+ t^2 )
$$
and
$$
\dashint_{\WW} |\nabla Q^t|^2\,dx \leq C\left(\frac1{t^2}\E_{t,w}(u^t)+ 1 \right).
$$
In particular, we deduce that
$$
\dashint_{U}\dist^2(Q^tA_t^{-1}, \SO(3))\,dx \leq C\dashint_{U} |\nabla_t u^t -Q^t|^2\,dx + C\E_{t,w}(u^t).
$$
Thus, up to subsequences, $Q^t\wconv{}Q$ weakly in $W^{1,2}$ for some $Q\in W^{1,2}(\WW;\R^{3\times 3})$ satisfying $Q\Ao^{-1}\in \SO(3)$ a.e.,
$\nabla_t u^t\to Q$ strongly in $L^2$, and, modulo a translation, $u^t\to f$ strongly in $W^{1,2}$ with
$\f \in W^{2,2}(\WW ; \R^3)$ such that $\nabla_w\f =Q_{3\times2}$. Since 
$$
Q^TQ=\Ao^T\Ao=\brk{\begin{matrix}\a(x_1,wx_2) & 0 \\ 0 & 1 \end{matrix}}, 
$$ 
we conclude that $f$ is a rescaled isometric immersion and $Qe_3=\bb_{\f }$,
where $\bb_{\f }$ is the normal to $f$.
\end{proof1}

\begin{proof1}{of Theorem~\ref{thm:Gconv_plates}:}
For the lower bound we use the same notation as in the proof of Theorem~\ref{thm:compactness_plates}. 
Write $A_t=\Ao(I+tx_3\Bo +O_{L^\infty}(t^2))$, where $B_0(x')=\frac1t \partial_3 A_t (x_1, x_2,0)$.  
The same argument as in \cite[Theorem~2.1]{BLS16} leads to the estimate
$$
\liminf_{t\to 0} \frac{1}{t^2}\E_{t,w}(u^t) \ge \dashint_{U} x_3^2\calQ_3(\Ao^{-T}\Go\Ao^{-1})\, dx,
$$
where $(\Go)_{2\times 2}=\II^w_f+(\Ao^T\Ao\Bo)_{2\times2}$. A direct computation shows that the symmetric part of $(\Ao^T\Ao\Bo)_{2\times2}$ coincides with $-\II^w$.
By \eqref{eq:Q_symmetric} and the definition of $\QQw$ we deduce the lower bound.
The construction of the recovery sequence is the same as in \cite[Theorem~3.1]{BLS16}.
\end{proof1}

%%%%%%%%%%%%%%%%%%%%%%%%%%%%%%%%%%%%%%%%%%%%%%%%%%%%%%%%%%%%%%%%%%%%%%%%%%%%%%%%%%%

%%%%%%%%%%%%%%%%%%%%%%%%%%%%%%%%%%%%%%%%%%%%%%
{\footnotesize
\bibliographystyle{amsalpha}
\bibliography{}	
}

\Addresses

\end{document}